\newif\ifpreprint

\preprinttrue  

\ifpreprint

\documentclass[12pt]{article}
\pdfoutput=1

\usepackage{amsmath,amssymb,amsthm}

\usepackage{fullpage}

\usepackage[round,authoryear]{natbib}
\else

\documentclass[ijoo,nonblindrev]{informs-ijoo}

\OneAndAHalfSpacedXI

\usepackage{natbib}
 \bibpunct[, ]{(}{)}{,}{a}{}{,}%
\TheoremsNumberedThrough     
\ECRepeatTheorems

\EquationsNumberedThrough    

\MANUSCRIPTNO{}

\fi

\usepackage{graphicx,subcaption,pgfplots}
\graphicspath{{./figures/}}

\usepackage{amsmath,amssymb,enumitem, mathtools}
\usepackage{multirow}
\usepackage{comment}
\usepackage{algorithm}
\usepackage[noend]{algpseudocode}
\usepackage{wasysym}
\usepackage[round-mode=places, round-precision=2,
            scientific-notation = true,
            table-format=-1.2e-1
            ]{siunitx}
\usepackage{csvsimple}  
\usepackage{booktabs}   
\usepackage{adjustbox}  

\usepackage{tikz}
\usetikzlibrary{patterns}
\usetikzlibrary{calc}
\usetikzlibrary{positioning}
\definecolor{splitColor}{RGB}{202, 200, 200}
\definecolor{classColor1}{RGB}{163, 31, 52}
\definecolor{classColor2}{RGB}{4, 30, 65}
\definecolor{classColor3}{RGB}{0, 105, 143}
\definecolor{classColor4}{RGB}{4, 30, 65}
\usepackage{tkz-euclide}
\usepackage{wrapfig}
\usepackage{tikz-3dplot}

\usepackage{multicol}
\usepackage{xcolor}
\usepackage{authblk}

\usepackage[xindy, acronyms]{glossaries}   
\makeglossaries

\usepackage{booktabs, csvsimple}
\usepackage{siunitx}

\definecolor{borderBoxTodo}{rgb}{0,0.309,0.6}
\definecolor{backgroundBoxTodo}{rgb}{0.701,0.854,1}
\definecolor{lineColorTodo}{rgb}{0.701,0.854,1}
\RequirePackage[colorinlistoftodos, backgroundcolor=backgroundBoxTodo, bordercolor = borderBoxTodo, linecolor = lineColorTodo]{todonotes}

\newcommand{\ie}{{\it i.e.}}

\renewcommand{\vec}[1]{\boldsymbol{#1}}

\newcommand{\reals}{{\mathbb{R}}}

\newcommand{\define}{\coloneqq}

\newcommand{\cp}{\overline{U}}

\newcommand{\cpc}{C}
\newcommand{\zcp}{z_{\rm cp}}
\newcommand{\zro}{z_{\rm ro}}
\newcommand{\zacp}{z_{\rm acp}}
\newcommand{\zaro}{z_{\rm aro}}
\newcommand{\uri}{V_i}
\newcommand{\usi}{W_i}
\newcommand{\xro}{\boldsymbol{x}^\star_{\rm ro}}
\newcommand{\yro}{\boldsymbol{y}^\star_{\rm ro}}
\newcommand{\xcp}{\boldsymbol{x}^\star_{\rm cp}}
\newcommand{\ycp}{\boldsymbol{y}^\star_{\rm cp}}
\newcommand{\xaro}{\boldsymbol{x}^\star_{\rm aro}}
\newcommand{\yaro}{\boldsymbol{y}^\star_{\rm aro}}
\newcommand{\uaro}{\boldsymbol{u}^\star_{\rm aro}}
\newcommand{\xacp}{\boldsymbol{x}^\star_{\rm acp}}
\newcommand{\yacp}{\boldsymbol{y}^\star_{\rm acp}}
\newcommand{\uacp}{\boldsymbol{u}^\star_{\rm acp}}

\newcommand{\zsi}{z_W}

\newcommand{\q}{q}
\newcommand{\ri}{\rho_{\rm ro}}
\newcommand{\ari}{\rho_{\rm aro}}
\newcommand{\arho}{\rho_{\rm adapt}}
\newcommand{\gi}{\gamma_{\rm ro}}
\newcommand{\agi}{\gamma_{\rm aro}}

\ifpreprint
\newtheorem{theorem}{Theorem}[section]  
\newtheorem{lemma}{Lemma}[section]  
\newtheorem{corollary}{Corollary}[theorem]  

\newtheorem{assumption}{Assumption}[section]
\newtheorem{definition}{Definition}[section]
\newtheorem{example}{Example}[section]

\fi

\newacronym{RO}{RO}{robust optimization}
\newacronym{aro}{ARO}{adaptive robust optimization}
\newacronym{cw}{constraint-wise}{constraint-wise}

\usepgfplotslibrary{fillbetween}
\pgfplotsset{
        discard if not symbolic/.style 2 args={
            filter discard warning=false,
            x filter/.append code={
                \edef\tempa{\thisrow{#1}}
                \edef\tempb{#2}
                \ifx\tempa\tempb
                \else
                    
                \fi
            },
        },}

\pgfplotsset{every axis/.append style={
		no markers,
		{thick}
}}

\ifpreprint \newcommand{\reviewChanges}[1]{{#1}}
\else
\newcommand{\reviewChanges}[1]{{#1}}
\fi

\begin{document}

\newcommand{\myabstract}{%
Despite the modeling power for problems under uncertainty, \gls{RO} and \gls{aro} can exhibit too conservative solutions in terms of objective value degradation compared to the nominal case.
One of the main reasons behind this conservatism is that, in many practical applications, uncertain constraints are directly designed as constraint-wise without taking into account couplings over multiple constraints.
In this paper, we define a coupled uncertainty set as the intersection between a \acrlong{cw} uncertainty set and a coupling set.
We study the benefit of coupling in alleviating conservatism in \gls{RO} and \gls{aro}.
We provide theoretical tight and computable upper and lower bounds on the objective value improvement of \gls{RO} and \gls{aro} problems under coupled uncertainty over constraint-wise uncertainty.
In addition, we relate the power of adaptability over static solutions with the coupling of uncertainty set.
Computational results demonstrate the benefit of coupling in applications.
}

\ifpreprint
\title{The Benefit of Uncertainty Coupling in Robust and Adaptive Robust Optimization}
\author{Dimitris Bertsimas, Liangyuan Na, Bartolomeo Stellato, Irina Wang}
\date{}
\maketitle
\begin{abstract}
	\myabstract
\end{abstract}

\fi

\section{Introduction}
Robust optimization (RO) and adaptive robust optimization (ARO) have become popular methods to deal with uncertainty in optimization problems~\citep{bertsimas2022book}.
In RO and ARO, the modeler constructs an uncertainty set to incorporate possible values for the uncertain parameters.
A robust optimal solution must satisfy the constraints under all realizations of the uncertain parameters in that set set while minimizing the worst case objective.
Despite its computational tractability and modeling power, RO can be overly conservative~\citep{bertsimas2004price,ben2017global}, especially for situations where some of the decisions can be made after the uncertainty is revealed.
In these cases, ARO can alleviate conservatism by allowing part of the decision variables (\ie, the wait and see decisions) to be functions of the uncertainty realizations.
However, it is often intractable to compute an optimal adaptive robust solution for all possible realizations of uncertainties since it is infinite-dimensional.

In many robust problems, the uncertain parameters take values independently across constraints.
We refer to such formulation as a \acrshort{cw} uncertainty set, which characterizes the uncertainty for each constraint and can be split into blocks such that each constraint depends solely on some components of the uncertainty~\citep[Section 14.2]{ben2009robust},~\citep{bertsimas2022book}.
On the contrary, in reality, the uncertain parameters in different constraints are often coupled due to common or related sources of uncertainty.
In these scenarios, the coupling restricts the worst case values the uncertain parameters can take, thereby reducing the overall size of the uncertainty set.
Therefore, considering the relationships between uncertain parameters in different constraints might result in less conservative solutions compared to a pure constraint-wise approach.
Previous appearances of coupled uncertainty include RO toy examples~\citep{GORISSEN2015practical,atamturk2007two} and application-specific ARO problems~\citep{ben2004adjustable,bertsimas2016duality,zhen2018adjustable}.
Yet, to our knowledge, there has been a limited focus in the literature on the generic framework of coupled uncertainty, and in particular, how uncertainty couplings across constraints benefit the optimal robust and adaptive robust solutions.
Furthermore, the coupling of the uncertainty is related to the improvement of adaptive upon static solutions.
For linear problems under \acrlong{cw} compact uncertainty set, it is known that a static robust solution gives a fully adaptive result~\citep{ben2004adjustable}.
When the uncertainties are coupled across constraints, adaptive solutions could improve the objective over static solutions.
This suggests possible relationships on how the coupling of the uncertainty quantitatively affects the improvement of adaptive over static solutions.

In this work, we study the performance gains in robust and adaptive robust optimization by taking into account the coupling between uncertain parameters.
We define {\it the coupled uncertainty set} as the intersection of a \acrlong{cw} uncertainty set and a coupling set, where the coupling set imposes constraints that model the relationship between uncertainties in different constraints.
We analyze and demonstrate the benefit of coupling on the objective of the optimal solution both theoretically and numerically.
We do so by first constructing tight upper and lower bounds on the objective improvement resulting from coupling, from which further bounds are derived to relate the power of adaptive over static solutions with the coupling of uncertainty set.
Numerically, we conduct computational experiments that quantify the practical benefits of coupling in various examples.

\subsection{Supply chain example}%
\label{sub:supply_chain_example}

\tikzstyle{startstop} = [rectangle, rounded corners, minimum width=3cm, minimum height=1cm,text centered, draw=black]
\tikzstyle{process} = [rectangle, minimum width=3cm, minimum height=1cm, text centered, draw=black=]
\tikzstyle{arrow} = [thick,->,>=stealth]

\begin{figure}[b]
\centering
\includegraphics[width=0.8\textwidth]{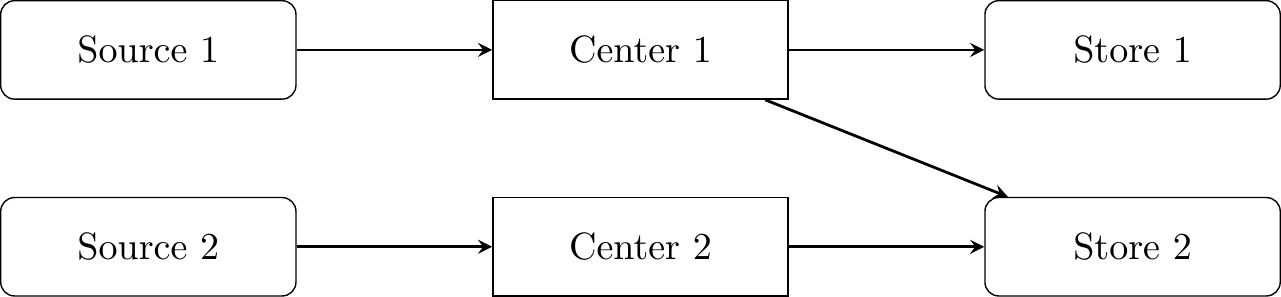}
\caption{Example of a supply chain network.}
\label{network}
\end{figure}

As an illustrative example of coupled uncertainty, we consider the supply chain problem in Figure~\ref{network}.
In this network, we ship products from two sources to two stores via two intermediate centers.
We formulate the problem as a robust optimization problem. Decision variables $x_{jk}$ and $y_{ki}$ denote the quantity of products to ship from source $j$ to center $k$, and from center $k$ to store $i$.
We define $c_{jk},s_{ki}$ denote the unit cost of the products shipped from source $j$ to center $k$, and from center $k$ to store $i$.
Also, $t,p$ denote the capacities of all routes in the first and second stage.
We define $u_i$ as the amount of uncertain demand at store $i$ which takes value from a \acrlong{cw} box uncertainty set
\begin{equation*}
    U = \left\{(u_1,u_2) \mid 0 \leq u_1 \leq 1, \quad 0 \leq u_2 \leq 1\right\}.
\end{equation*}
We have the following formulation,
\begin{equation}
\label{eq:scex}
{
\begin{array}{llclcl}
\displaystyle \text{minimize} & \multicolumn{3}{l}{c_{11} x_{11}+c_{22} x_{22} + s_{11} y_{11}+ s_{22}y_{22}+s_{12}y_{12}} \\
\text{subject to} & y_{11}  \geq u_1, y_{12}+y_{22}  \geq u_2, && \forall (u_1,u_2) \in U \\
& 0 \leq x_{11},x_{22} \leq t \\
& 0 \leq y_{11},y_{12},y_{22} \leq p\\
& x_{11} \geq y_{11}+y_{12}\\
&x_{22} \geq y_{22}.
\end{array}}
\end{equation}
The objective is to minimize the total transportation cost. The first constraint requires that each store satisfies the demand under uncertainty. The second and third constraints denote the shipping capacities on each route. The fourth and fifth constraints ensure that the outflow from the centers is no more than the inflow to the centers in the network.

In the constraint-wise uncertainty set, we do not model any correlation between $u_1$ and $u_2$. In reality, however, uncertainties are often correlated.
\reviewChanges{We hereby present two distinct scenarios in which the uncertainties are correlated, resulting in {\it coupled} uncertainty sets. Throughout the paper, we refer back to these scenarios to demonstrate conditions under which coupling may bring benefits, and where it does not. }

Consider scenario (\textit{a}) where the total demand of the two stores is less than a fixed amount $\eta$.
Then the demand is restricted in the coupled uncertainty set
\begin{equation*}
    \cp_a = \left\{(u_1,u_2) \mid 0 \leq u_1 \leq 1,\quad 0 \leq u_2 \leq 1, \quad u_1+ u_2 \leq \eta\right\}.
\end{equation*}
Consider scenario (\textit{b}) where store 2 has a fixed demand between $\alpha$ and $\beta$ from its own source, and both stores have a common portion of uncertain demand.
Then demands at the two stores satisfy a relationship $\alpha \leq u_2 - u_1 \leq \beta.$ This results in the coupled uncertainty set
\begin{equation*}
    \cp_b = \left\{(u_1,u_2) \mid 0 \leq u_1 \leq 1, \quad 0 \leq u_2 \leq 1, \quad \alpha \leq u_2 - u_1 \leq \beta\right\}.
\end{equation*}
The uncertainty sets are shown in Figure~\ref{fig:ualpha}.
\reviewChanges{We note that this is only one of numerous ways to model such correlation, as we can always define transformations of variables to adjust the bounds. In fact, the coupled uncertainty set models common problems in which the uncertainty appears in multiple constraints. By identifying the interactions between multiple uncertain parameters and summarizing them in a coupled uncertainty set, we generalize the modeling of such problems, and subsequently generalize the analysis of the impact on the optimal solutions.}

In the \acrlong{cw} problem under uncertainty set $U$, the worst case occurs at $\vec{u}^\star = (1,1)$, where both stores have demand 1.
In the problem under coupled uncertainty set, the worst case occurs at
$\vec{\overline{u}}^\star_a = (1,1)$ for scenario (\textit{a}) and at $\vec{\overline{u}}^\star_b = (1-\alpha,1)$ for scenario (\textit{b}).
With parameters $c_{11}=c_{22}= 100, s_{11}=s_{22}=s_{12}=200, t = 1, p = 1$, the original problem has an optimal cost of 600. With $\eta = 3/2, \alpha=1/2, \beta = 3/4$, scenario (\textit{a}) has cost 600 and scenario (\textit{b}) has cost 450 with $25\%$ improvement.
This suggests a potential benefit from coupling.

\begin{figure}
\centering
\begin{subfigure}{0.49\textwidth}
\centering
\includegraphics[width=0.7\textwidth]{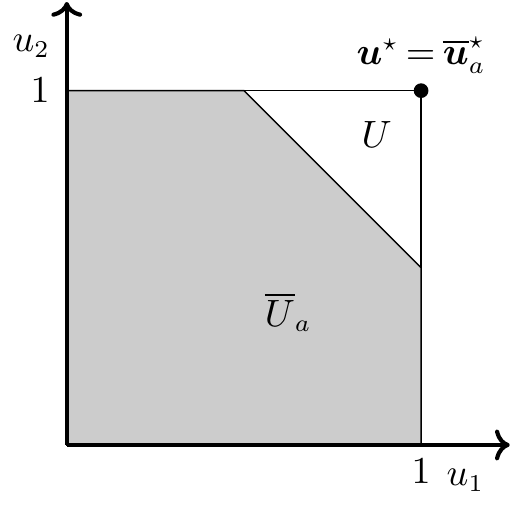}
\caption{$u_1+ u_2 \leq \eta = 3/2$ .}
\label{fig:ualphaa}
\end{subfigure}
\begin{subfigure}{0.49\textwidth}
\centering
\includegraphics[width=0.7\textwidth]{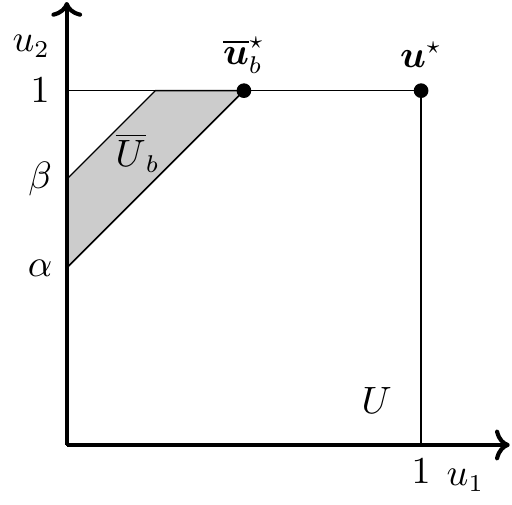}
\caption{$\alpha = 1/2 \leq u_2 - u_1 \leq 3/4 = \beta$.}
\label{fig:ualphab}
\end{subfigure}
\caption{Two scenarios of uncertainty sets for the supply chain example.}
\label{fig:ualpha}
\end{figure}

Moreover, if the second-stage decision on the shipment from centers to stores can be made after the demands are realized, we formulate the following adaptive robust problem:
\begin{equation}
\label{eq:scexadapt}
{
\begin{array}{llclcl}
\displaystyle \text{minimize} & \multicolumn{3}{l}
{c_{11} x_{11}+c_{22} x_{22} + \max\limits_{u \in U} s_{11} y_{11}(\vec{u})+ s_{22}y_{22}(\vec{u})+s_{12}y_{12}(\vec{u})} \\
\text{subject to} & y_{11}(\vec{u})  \geq u_1,y_{12}(\vec{u})+ y_{22}(\vec{u})  \geq u_2, && \forall \vec{u} \in U \\
& 0 \leq x_{11},x_{22} \leq t \\
& 0 \leq y_{11}(\vec{u}),y_{12}(\vec{u}),y_{22}(\vec{u}) \leq p, && \forall \vec{u} \in U\\
& x_{11} \geq y_{11}(\vec{u})+y_{12}(\vec{u}), && \forall \vec{u} \in U\\
&x_{22} \geq y_{22}(\vec{u}),&& \forall \vec{u} \in U.
\end{array}}
\end{equation}
For the adaptive case, the worst case could occur on the line segment where $u_1+ u_2 = \eta$ for scenario (\textit{a}) and at $(1-\alpha,1)$ for scenario (\textit{b}).
With the same parameters,
the \acrlong{cw} problem has an optimal cost of 600 and both scenarios under coupled uncertainty have cost 450 with $25\%$ improvement.
This suggests that coupling could potentially bring more benefits when the problem is adaptive.
We will analyze how coupling affects the objectives for static and adaptive problems and refer back to this supply chain example throughout the paper for illustration purposes.

\subsection{Related work}
Parameters in optimization problems are often subject to uncertainty from a variety of sources. The uncertainty can affect the solution to be infeasible or suboptimal, and optimization affected by parameter uncertainty therefore has been an important topic. Two main classes of methods that deal with uncertainty in optimization are stochastic optimization and robust optimization.

\paragraph{Stochastic optimization.}
In stochastic optimization, uncertain parameters are given with a probability distribution. The constraints and the objective are taken with regard to expectation over the distributions and processed with techniques to be solved. We refer to the textbooks by~\cite{kall1994stochastic,birge2011introduction,shapiro2014lectures} for a more detailed description.
Expectation constraints display great modeling power and ensure explicit probabilistic guarantees on constraint violation.
However, the main disadvantage of stochastic optimization is the dependence on inexact estimations of the true distribution of uncertain parameters.
Also, such programs are often hard to solve because of the difficulties involving expectations and the often nonconvexity of the resulting feasible regions.
An analogy of coupled uncertainty in a stochastic optimization setting could be a joint probability distribution for all uncertain parameters across constraints.

\paragraph{\Glsdesc{RO}.}
Instead of assuming a stochastic nature of the uncertainty, \gls{RO} characterizes the uncertainties in an uncertainty set. A robust optimal solution is found to satisfy all constraints under all uncertainties in the uncertainty set. An early idea is introduced in~\cite{soyster1973convex} to construct a solution which is feasible for all uncertain data from a convex set, and a theory for \gls{RO} is developed in~\cite{ben1998robust,ben1999robust,ben2000robust,el1997robust}. Robust optimization is computationally tractable for various types of problems and uncertainty sets,
and has a broad modeling power for applications~\citep{bertsimas2004price,bertsimas2011theory} such as~\cite{bertsimas2006robust,bertsimas2012adaptive}.

\paragraph{Alleviating conservatism in \glsdesc{RO}.}
There have been various efforts in alleviating the conservatism in \gls{RO}.
Since \gls{RO} enforces all constraints to be satisfied for all parameters with no tolerance, it often imposes stricter constraints than what occurs in practice.
Several methods address this issue by allowing some violations in the constraints.
Light robustness~\citep{Fischetti2009} minimizes a weighted sum of all constraint violations among solutions that attain a threshold objective value.
Globalized \glsdesc{RO}~\citep{ben2006extending,ben2017global} specifies two uncertainty sets, imposes hard constraints for a more normal range of uncertainties, and allows some violation in another set.
Robustness optimization~\citep{satisficing} specifies an objective target and bounds the violation by adversarial impact measure in the full support of the uncertainties.
These methods attempt to reduce conservatism by relaxing the constraints, but do not address another main cause of conservatism in \gls{RO}, which results from its constraint-wise nature.
By protecting the worst case for each constraint separately, \gls{RO} can impose stricter constraints if uncertainties across constraints are related in reality.
The constraint-wise issue is incorporated in~\cite{hertog2019reduce} from a Distributionally Robust Optimization~\citep{postek2018robust} perspective, which ensures constraints to hold for probability distributions in an ambiguity set.
Their formulation combines all constraints into one constraint and requires the total worst case expected violation of all constraints to be below a certain value.
However, such aggregation requires the modeler to rescale the constraints and might lose some individuality for each constraint.
In comparison, the coupling framework addresses the issue while retaining the original problem structure.

\paragraph{Adaptive robust optimization.}
For situations in which practitioners can make some of the decisions after knowing the values of some or all uncertain parameters, adaptive robust optimization can reduce conservatism by providing flexibility and adapting the values of some decisions to realizations of uncertainties.
The concept of ARO was first introduced in~\cite{ben2004adjustable}, and we refer to~\cite{YANIKOGLU2019survey} that surveys the literature on the theory, methodology, and applications of ARO.
In ARO, first-stage here-and-now decisions are made before the uncertainties are realized, and second-stage wait-and-see decisions can be adjusted to some or all parts of the revealed uncertain data.
The framework can also be extended to model multi-stage decision making~\citep{ben2004adjustable}.
Since ARO is a complex problem to solve in general, methods are developed to obtain approximate solutions, such as adopting linear decision rules~\citep{ben2004adjustable,chen2009uncertain}, finite scenario approach~\citep{hadjiyiannis2011finite}, and Benders decomposition~\citep{bertsimas2012adaptive}, among others.

\paragraph{Coupled uncertainty and its benefit in robust and adaptive optimization.}
Problems under coupled uncertainty \reviewChanges{have long since appeared as examples} in robust optimization.
For instance, an example to illustrate that RC may not be equivalent to a minmax reformulation if uncertainty is not constraint-wise appears in~\cite{GORISSEN2015practical}.
A network flow example is used in~\cite{atamturk2007two} to illustrate that adaptability can improve robust solutions when the uncertainty set is coupled.
To solve robust problems, the robust counterpart is obtained by projecting the uncertainty set onto the space of each constraint~\cite[Section 1.2]{ben2009robust}.
Therefore, the robust counterpart of coupled problems can be formulated by taking the projections directly, although it is in general intractable to compute projections of sets, especially in high dimensions.
In adaptive robust optimization, numerical examples have demonstrated the benefit of adaptability over static solutions when the uncertainties across different constraints are dependent on each other.
In lot sizing on a network~\citep{bertsimas2016duality,zhen2018adjustable}, a budgeted uncertainty set~\citep{bertsimas2004price} restricts the total demand for all stores;
in the inventory management problem~\citep{ben2004adjustable}, robust constraint at each time period is affected by sum of previous uncertain demands up to the current period;
in the location transportation example~\citep{atamturk2007two}, a cardinality-restricted uncertainty set~\citep{bertsimas2004price} relates demand deviations in different locations.
\reviewChanges{The aforementioned studies model specific problems and applications, and one could contend that all data-driven applications in which the uncertainty set is determined from past data, whether through the use of statistical tests~\citep{bertsimas_data-driven_2018}, construction of a Wasserstein ball~\citep{mohajerin_esfahani_data-driven_2018}, or the intersection of different ambiguity sets~\citep{tanoumanddata}, make use of the coupling between different dimensions of the uncertainty. 
Yet,} the general framework of coupled uncertainty and how the coupling affects the improvement of objectives have not been studied in depth.
Recently, \cite{nohadani2022optimization} studies a general framework modeling multi-period uncertainty sets in which realizations from previous periods affect parameters of future uncertainty sets.
They reformulated several common families of uncertainty sets and demonstrated on numerical robust and distributionally robust optimization problems.
Similarly, our work studies application-independent general modeling frameworks where uncertain parameters are dependent or correlated with each other.
However, one fundamental difference is that we consider the coupling of uncertainties that appeared across multiple constraints, whereas they focus on connections between uncertainties in multiple time periods in a time series setting.
\reviewChanges{Lastly, we remark that finite scenario-based approaches, such as the one proposed by~\cite{hadjiyiannis2011finite}, where the uncertainty set is replaced by a finite discrete subset, implicitly satisfies coupling constraints. There is no need to model dependencies between uncertain parameters when scenarios are explicitly given. However, this approach could still be used as a baseline for comparisons against other methods. }

\paragraph{Performance of static solution in adaptive optimization.}
If a linear optimization problem has a constraint-wise and compact uncertainty set, the optimal values of RO and ARO are the same~\citep{ben2004adjustable}.
The result on the optimality of static solutions under constraint-wise uncertainty is also extended to nonlinear problems under certain conditions~\citep{marandi2018static}.
When the constraint-wise condition on uncertainty is not satisfied, the static solution is not necessarily optimal to the adaptive problem.
Several bounds have been studied to characterize the objective of static solutions relative to adaptive solutions.
An upper bound on the power of static solutions is developed in~\cite{bertsimas2010power,bertsimas2011geometric} based on a measure of the symmetry of a convex compact uncertainty set introduced in~\cite{minkowski1911allegemeine} for problems with right hand side uncertainty.
The results are later generalized to problems with convex constraints and objective functions with uncertain coefficients in~\cite{bertsimas2013approximability}, but the bound can be quite loose in some instances.
Another upper bound related to a measure of non-convexity of a transformation of the uncertainty set gives a tight characterization~\citep{bertsimas2015tight}, but it is in general not necessarily tractable to compute for an arbitrary convex compact set and the transformation is not very straightforward to interpret.
We relate the improvement of adaptive over static solutions to the coupling of the uncertainty set in comparison with constraint-wise uncertainty sets, in which the static solution is optimal.
Our bound can be obtained by solving convex programs and provides simple geometric insight into the power of adaptability.

\paragraph{\reviewChanges{Relationship to robust satisficing.}} 
\reviewChanges{There is a stream of research coined robust satisficing~\citep{satisficing,satisficing2},
in which, instead of restricting the radius of the uncertainty set, an acceptable cost target is provided. 
The size of the uncertainty set, or, in particular cases, a multiplicative scaling of the possible deviations from the empirical distribution, is calibrated as a variable within the optimization problem. 
Classical robust optimization problems can be rewritten in this framework, where the maximum uncertainty set is to be found given a target objective value~\citep{satisficing2}. 
The shrinkage factors we will define and use in later sections may seem, at first glance, similar to the radius calibrations done in the above manner. 
However, there is a fundamental difference between this line of research and ours; our shrinkage factors are obtained independently from the objective of the robust problem, as we focus only on the uncertainty sets, both nominal and coupled. 
Instead of identifying a maximum uncertainty set for a target objective value, we find shrinkage factors of fixed-radius uncertainty sets to establish a desired hierarchy, then use these shrinkage factors to infer the possible improvements in the objective brought by coupling.
The power of our approach lies in its simplicity, as the analysis is based on the structure of the uncertainty alone.}

\subsection{Our contributions}
In this paper, we study the theoretical and computational benefits of a generic uncertainty coupling framework in robust and adaptive robust optimization.
In comparison to a constraint-wise uncertainty set that considers uncertainty for each constraint separately, a coupled uncertainty set captures relationships between uncertainties for different constraints.
Our specific contributions include:
\begin{enumerate}
    \item We provide tight and easily computable theoretical upper and lower bounds on how much improvement of the objective value can be obtained from uncertainty coupling in~RO and~ARO.
These results apply to linear and nonlinear problems in both cases when the uncertainties appear on the right hand side and the coefficient of the constraints. We further generalize these results to (not necessarily linear) problems when the uncertainty affects both objectives and constraints.
    \item We characterize the performance of adaptive solutions relative to static solutions under coupled uncertainty.
     \item Our computational experiments show the effect of coupling parameters on the improvement, illustrate the bounds, and test different solution methods for linear RO and ARO problems.
    Numerical results in supply chain management, portfolio optimization, and lot sizing in a network demonstrate that incorporating coupling can bring benefits in practical applications.
\end{enumerate}

\subsection{Outline}
The structure of the paper is as follows.
In Section~\ref{sec:modeling}, we formulate robust and adaptive robust problems under constraint-wise and coupled uncertainty respectively with right hand side uncertainty.
In Section~\ref{sec:theory}, we derive theoretical bounds on the objective improvement of coupling, compare the improvement for static and adaptive cases, and characterize the power of adaptability over static solutions under coupled uncertainty.
In Section~\ref{sec:theorycoeff}, we extend the theoretical results for linear problems with constraint coefficient uncertainty.
We further generalize the results to nonlinear problems in Section~\ref{sec:theorynl}.
In Section~\ref{sec:computation}, we conduct computational experiments to demonstrate the numerical effects of coupling and test different solution methods in multiple example problems.
Section~\ref{sec:conclusion} summarizes our conclusions.

\section{Robust and adaptive robust modeling}
\label{sec:modeling}
In this section, we formulate RO and ARO problems with right hand side uncertainties in \acrlong{cw} and coupled uncertainty sets respectively.
In this paper, we use $[m]$ for $m \in \mathbb{Z}^+$ to denote the set $\left\{1,2,\dots,m\right\}$.

\subsection{Constraint-wise robust formulation}
We consider a linear robust optimization problem with $m$ robust constraints and uncertainty on the right hand side.
The constraint-wise formulation is of the form
\begin{equation*}
{
\begin{array}{llclcl}
\displaystyle
& \text{minimize} && {\vec{c}^T\vec{x} + \vec{d}^T\vec{y}} \\
& \text{subject to} &&
 \vec{a}_i^T\vec{x} + \vec{g}_i^T\vec{y}\geq  u_i , \quad \forall u_i \in U_i, \quad \forall i \in [m],
\end{array}}
\end{equation*}
where $\vec{x} \in \reals^{n_1}$ and $\vec{y} \in \reals^{n_2}$ are the decision variables, and $\vec{c} \in \reals^{n_1}, \vec{d} \in \reals^{n_2}, \vec{a}_i \in \reals^{n_1}, \vec{g}_i \in \reals^{n_2},u_i \in \reals$.
Let the uncertainty for the problem be $\vec{u} = (u_1,u_2,...,u_m) \in \reals^{m}$. The  uncertainty set for the problem is $$U = U_1 \times U_2 \times ... \times U_m = \left\{(u_1,u_2,...,u_m) \mid u_1 \in U_1, u_2 \in U_2, ..., u_m \in U_m\right\}.$$
The problem formulation is equivalently
\begin{equation}
\label{eqn:ro}
{
\begin{array}{llclcl}
\displaystyle \zro =
& \text{minimize} && {\vec{c}^T\vec{x} + \vec{d}^T\vec{y}} \\
& \text{subject to} &&
 \vec{a}_i^T\vec{x} + \vec{g}_i^T\vec{y}\geq u_i , \quad \forall \vec{u} \in U, \quad \forall i \in [m],
\end{array}}
\end{equation}
where $\zro$ is the objective value of the constraint-wise robust optimization.

\subsection{Robust modeling under coupled uncertainty}
We now introduce a set of coupling constraints that models the relationships between uncertainty across different constraints.
We denote the coupling uncertainty set by $\cpc$.
We formulate the {\em coupled uncertainty set} as the intersection $$\cp = U \cap \cpc.$$
The robust optimization problem under \reviewChanges{the} coupled uncertainty set is
\begin{equation}
\label{eqn:cp}
{
\begin{array}{llclcl}
\displaystyle \zcp =
& \text{minimize} && {\vec{c}^T\vec{x} + \vec{d}^T\vec{y}} \\
& \text{subject to} &&
 \vec{a}_i^T\vec{x} + \vec{g}_i^T\vec{y}\geq u_i, \quad \forall \vec{u} \in \cp, \quad \forall i \in [m].
\end{array}}
\end{equation}
where $\zcp$ is the objective value of the robust problem under coupled uncertainty.
By construction, $\cp \subseteq U$, and it follows directly $$\zcp \leq \zro.$$ The problem under coupled uncertainty can achieve an objective at least as good as the constraint-wise problem. \reviewChanges{There are two cases, which we illustrate with our motivating example.}

\begin{example}
For the supply chain example in the introduction, the two scenarios of coupling uncertainty sets as illustrated in Figure~\ref{fig:ualpha} are
$$\cpc_a = \left\{(u_1,u_2) \mid u_1 + u_2 \leq \eta\right\}, \quad \cpc_b = \left\{(u_1,u_2) \mid \alpha \leq u_2 - u_1 \leq \beta\right\}.$$
\begin{description}
	\item[No improvement.] When $\eta \geq 1$ for scenario (\textit{a}) and when $\alpha \leq 0, \beta \geq 0$ for scenario (\textit{b}), coupling cannot improve the objective value.
	\item[Improvement.] When $0 \leq \eta < 1$ for scenario (\textit{a}) and when $0 < \alpha \leq \beta $ or $ \alpha \leq \beta < 0$ for scenario (\textit{b}), there exist problems such that coupling improves the objective value.
\end{description}
\end{example}

\subsection{Adaptive robust modeling under constraint-wise uncertainty}

For cases when some decisions can adapt to realizations of some uncertainties, we consider an adaptive robust optimization problem under constraint-wise uncertainty set $U$
\begin{equation}
\label{eqn:aro}
{
\begin{array}{llclcl}
\displaystyle \zaro =
& \text{minimize} && {\vec{c}^T\vec{x} + \max\limits_{\vec{u} \in U} \vec{d}^T\vec{y}(\vec{u})} \vspace{3pt}\\
& \text{subject to} &&
 \vec{a}_i^T\vec{x} + \vec{g}_i^T\vec{y}(\vec{u})\geq u_i, \quad \forall \vec{u} \in U, \quad \forall i \in [m],
\end{array}}
\end{equation}
where $\vec{x} \in \reals^{n_1}$ is the first-stage decision which is made before $\vec{u}$ is realized, $\vec{y}(\vec{u}) \in \reals^{n_2}$ is the second-stage decision which can be adjusted based on the realization of $\vec{u}$, and
$\zaro$ is the objective value of the \acrlong{cw} adaptive robust optimization. 

\subsection{Adaptive robust modeling under coupled uncertainty}
The adaptive robust optimization problem under coupled uncertainty set $\cp$ is
\begin{equation}
\label{eqn:acp}
{
\begin{array}{llclcl}
\displaystyle  \zacp = & \text{minimize} && {\vec{c}^T\vec{x} + \max\limits_{\vec{u} \in \cp} \vec{d}^T\vec{y}(\vec{u})} \vspace{3pt}\\
& \text{subject to} &&
 \vec{a}_i^T\vec{x} + \vec{g}_i^T\vec{y}(\vec{u})\geq u_i, \quad \forall \vec{u} \in \cp, \quad \forall i \in [m],
\end{array}}
\end{equation}
where $\zacp$ is the objective value of the adaptive robust problem under coupled uncertainty.
Similarly, as $\cp \subseteq U$, we have $$\zacp \leq \zaro.$$

\begin{example}
For the adaptive supply chain example, scenario (\textit{b}) has the same two cases as the static problem.
For scenario (\textit{a}), coupling cannot improve the objective when $\eta \geq 2$ and could improve when $0 \leq \eta < 2$.
Particularly, when $1 \leq \eta < 2$, coupling might improve the adaptive problem but not the static problem, suggesting a larger potential improvement in the adaptive case.
\end{example}

\subsection{\reviewChanges{Coupling versus obtaining the same worst case across all constraints}} 
\label{sec:worst-case}
\reviewChanges{We end this section with a discussion on coupling versus the idea of obtaining the same worst case across all constraints; notably, that the two concepts are not equivalent. Here, we consider general uncertain functions of $\vec{u}$, instead of only right hand side uncertainty. } 

\reviewChanges{Recall that, for static constraint-wise uncertainty, we require the uncertainty for individual constraints to be fully separable, \ie, the $i^{\text{th}}$ constraint involves only $u_i$, which lives in the set $U_i$, and which has no interactions with any $U_j$ for $j \neq i$.
It follows that the worst case values for different constraints are independent from each other. 
On the other hand, a coupled uncertainty set relates the uncertainty across different constraints, such that the possible values of $u_i$ for the $i^{\text{th}}$ constraint depend also on the possible values of $u_j$ for $j \neq i$. 
The uncertainty is no longer separable; for each constraint, instead of considering only $u_i$, $\vec{u}$ is considered holistically.
This is also the case for adaptive robust optimization problems. 
However, contrary to expectation, the worst case values for different constraints remain independent, even if the condition $u_1 = u_2 = \dots = u_m$ is enforced in the coupling set. 
Essentially, while for each constraint the coupling constraints must be satisfied, each constraint may still find a different worst case realization of $\vec{u}$.
For instance, even if the uncertainty is coupled with $\left\{(u_1,u_2) \mid u_1 + u_2 \leq \eta\right\}$, one constraint might have $\vec{u}$ take values $u_1 = \eta, u_2 = 0$, while a different constraint might have $u_1 = 0, u_2 = \eta$.
Therefore, while coupling restricts the worst case values that the uncertain parameters can take, it does not remove the individuality of the different constraints. 
Robust optimization problems under coupled uncertainty can be considered an intermediary formulation, where the possible worst case values are more restricted than in fully constraint-wise problems, but less restricted than when all constraints are aggregated into a single constraint, in which case the same worst case is obtained across all original constraints.}

\reviewChanges{We remark, however, that one of the $m$ realizations of $\vec{u}$ {\it can} be taken as the worst case across all constraints. Nonetheless, without aggregating all constraints, it is difficult to identify which of the $m$ constraints contributes this extreme worst case. }

\section{Theoretical improvement under right hand side uncertainty}
\label{sec:theory}
In this section, we study how much improvement can be achieved from coupling for problems with right hand side uncertainty. We provide an upper bound and a lower bound on the ratio of objectives for coupled and constraint-wise problems in static and adaptive cases. We illustrate the bounds in examples and show that our bounds are tight.
We further utilize the results to bound improvement of adaptive solutions over static solutions under coupled uncertainty set.

\subsection{Robust optimization}
\label{sec:ro_intro}
\begin{definition}
We define the projection of $U$ on the space of $\vec{u}_i$ as
$$\Pi_i(U) = \left\{\vec{u}_i \mid \exists \vec{u}_1,\dots,\vec{u}_{i-1},\vec{u}_{i+1},\dots,\vec{u}_m \text{ }\rm s.t. \text{ } (\vec{u}_1,\dots,\vec{u}_m) \in U\right\},$$
and we define $\Pi(U) = \Pi_1(U) \times \dots \times \Pi_m(U)$.
\end{definition} \reviewChanges{The hierarchy of the currently defined sets is
\begin{equation*}
\cp \subseteq \Pi(\cp) \subseteq \Pi(U) = U,
\end{equation*}
where the projection of the coupled set contains the coupled set, but cannot be larger than the original uncertainty set. The projection of the original uncertainty set is itself.
An illustration of the projections of a coupled uncertainty set $\cp$, with $\vec{u}_1 = (u_1^1,u_1^2) \in \reals^2$, $u_2 \in \reals$, and $\vec{u} = (\vec{u}_1,u_2) \in \reals^3$ is shown in Figure~\ref{fig:proj}.}
\begin{figure}[H]
    \centering{%
\includegraphics[width=0.25\textwidth]{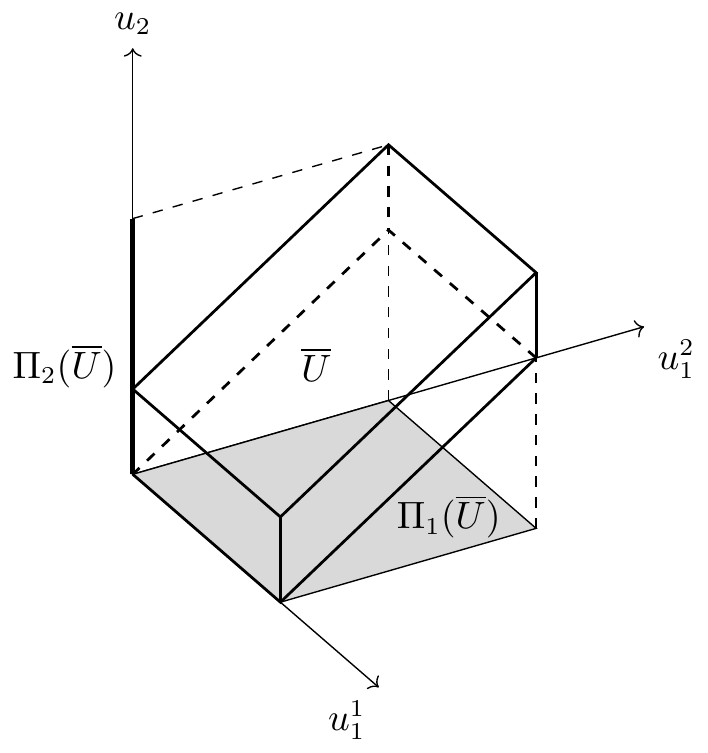}
}
\caption{\reviewChanges{Illustration of projections of a coupled set $\cp$. The set $\Pi(\cp)$ would be the cube $\Pi_1(\cp) \times \Pi_2(\cp)$, which contains $\cp$. }}
\label{fig:proj}
\end{figure}

\reviewChanges{For constraint-wise uncertainty, we assume $U \subset \reals_+^m$ is a convex set, and introduce the notion of a down-hull.} 
\begin{definition}
\label{def:downhull}
The down-hull of a set $S \subseteq \reals_+^m$ is defined as
$$S^\downarrow = \left\{\vec{t} \in \reals_+^m \mid \exists s \in S : \vec{t} \leq \vec{s}\right\} \subseteq \reals_+^m.$$
\end{definition}
\reviewChanges{We work with the down-hull of the sets and observe the hierarchy
\begin{equation*}
\cp \subseteq \cp^\downarrow \subseteq \Pi\left(\cp^\downarrow\right) \subseteq \Pi\left(U^\downarrow\right) =U^\downarrow.
\end{equation*}
Notably, solving a constraint-wise robust problem under an uncertainty set $U$ is equivalent to solving it under the down-hull $U^\downarrow$ (See Appendix~\ref{append:down}). Therefore, without loss of generality, in this section we solve the robust problems with respect to the down-hulls of the uncertainty sets.
Furthermore, by the constraint-wise nature of the uncertainty, we can assume without loss of generality that $\vec{u}_i \in \reals$ for all $i\in [m]$. It follows that 
$$\underset{\vec{u} \in \Pi\left(\cp^\downarrow\right)}{\text{maximize}}~u_i = \underset{\vec{u} \in \cp^\downarrow}{\text{maximize}}~u_i, \quad \forall i \in [m],$$
and that the robust problem~\eqref{eqn:cp} under $\cp$ is equivalent to}
\begin{equation}
\label{eqn:cp_proj}
{
\begin{array}{llclcl}
\displaystyle
\zcp =
& \text{minimize} && {\vec{c}^T\vec{x} + \vec{d}^T\vec{y}} \\
& \text{subject to} &&
 \vec{a}_i^T\vec{x} + \vec{g}_i^T\vec{y}\geq u_i, && \forall \vec{u} \in \reviewChanges{\Pi\left(\cp^\downarrow\right)}, \quad \forall i \in [m].
\end{array}}
\end{equation}
\reviewChanges{It appears that the relationship between $\zcp$ and $\zro$ depends on the relationship between $U^\downarrow$ and $\Pi\left(\cp^\downarrow\right)$, which we now seek to quantify.}

Let the rescaling of a set $S$ by $r \in \reals^+$ be $$ r S = \left\{rx \mid x \in S\right\}.$$
We define two auxiliary sets and shrinkage factors to bound the improvement in the objective value.

\begin{definition}
\label{def:rhogam_func}
For two sets $S_1$ and $S_2$,
we define the maximum shrinkage factor function of the two sets
$$\rho(S_1, S_2) = \max\left\{\rho \in \reals_+ \mid \rho S_1 \subseteq S_2\right\},$$
if such a number exists.
Similarly, we define the minimum shrinkage factor function of the two sets
$$\gamma(S_1, S_2) = \min\left\{\gamma \in \reals_+ \mid S_2 \subseteq \gamma  S_1\right\},$$
if such a number exists.
\end{definition}

\begin{definition}
\label{def:rhogam}
For uncertainty sets $U$ and $\cp$,
we define their maximum and minimum shrinkage factors,
\reviewChanges{$\ri = \rho\left(U^\downarrow, \Pi\left(\cp^\downarrow\right)\right)$ and $\gi = \gamma\left(U^\downarrow, \Pi\left(\cp^\downarrow\right)\right),$ from Definition~\ref{def:rhogam_func}.
We then construct sets
$V = \ri \, U^\downarrow$ and
$W = \gi \, U^\downarrow$,
such that the relationship between the sets is $V = \ri \, U^\downarrow \subseteq \Pi\left(\cp^\downarrow\right) \subseteq \gi \, U^\downarrow = W $.}
\end{definition}

\begin{lemma}
\label{lem:exist}
 \reviewChanges{$\ri$ exists; $\gi$ always exists.}
\end{lemma}
\ifpreprint \begin{proof} \else
\proof{Proof.}
\fi
If $\ri$ does not exist, then for all \reviewChanges{$\rho \in \reals_+$, $\rho \, U^\downarrow \not\subseteq \Pi\left(\cp^\downarrow\right)$.
Then $0 \cdot U^\downarrow \not\subseteq \reviewChanges{\Pi\left(\cp^\downarrow\right)}$ and $0 \notin \reviewChanges{\Pi\left(\cp^\downarrow\right)}$.
Since $\gamma =1$ satisfies $\reviewChanges{\Pi\left(\cp^\downarrow\right)} \subseteq \gamma \, U^\downarrow$, $\gi$ always exists.}
\ifpreprint \end{proof} \else
\Halmos
\endproof
\fi

\begin{lemma}
\label{lem:rigi}
Assume \reviewChanges{$U$ is a convex set.} Then $W \subseteq U^\downarrow$ and
\begin{equation*}
   0 \leq \ri \leq \gi \leq 1.
\end{equation*}
\end{lemma}
\ifpreprint \begin{proof} \else
\proof{Proof.}
\fi
\reviewChanges{This follows from the convexity and hierarchy of the sets.}
\ifpreprint \end{proof} \else
\Halmos
\endproof
\fi




\paragraph{Obtaining shrinkage factors.}
 \reviewChanges{Given these definitions, let} $\reviewChanges{\Pi\left(\cp^\downarrow\right)} = [0,\overline{d}_1]\times\dots\times[0,\overline{d}_m]$
and $\reviewChanges{U^\downarrow} = [0,d_1]\times\dots\times[0,d_m]$.
The shrinkage factors can be obtained by 
\begin{equation}
\label{eq:gamma_static}
    \ri = \min\limits_{i \in [m]} \overline{d}_i/d_i, \quad \gi = \max\limits_{i \in [m]} \overline{d}_i/d_i.
\end{equation}
The interpretation is that $\ri$ and $\gi$ characterize the maximum and minimum levels of shrinkage obtained in the projections of coupled uncertainty set.

\begin{example}
\label{ex:scrho}
We illustrate the shrinkage factors with the supply chain example in Figure~\ref{fig:ualphaproj}. \reviewChanges{We note that in this case, $U = U^\downarrow$.}
For the case of no improvement in Figure~\ref{fig:noimp}, the projections are the same for $U$ and $\cp^\downarrow$.
Then $U = \reviewChanges{\Pi\left(\cp^\downarrow\right)}=V=W$ is the same line segment and $\ri = \gi = 1$.
For the case of an improvement in Figure~\ref{fig:imp}, we illustrate the projections of $\cp^\downarrow$.
Coupling shrinks the projections with $\ri = 1-\alpha~\reviewChanges{= 1/2} < 1$, $\gi = 1$.
\end{example}

\begin{figure}
\centering
\begin{subfigure}{0.49\textwidth}
\centering
\includegraphics[width=0.81\textwidth]{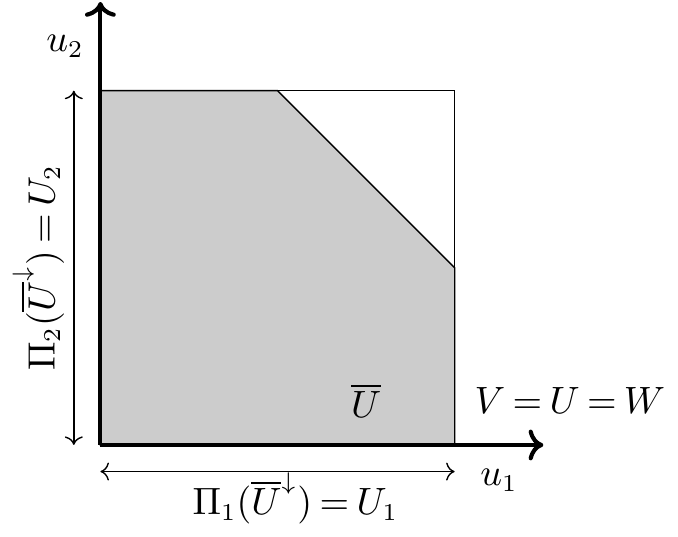}
\caption{\reviewChanges{$u_1+ u_2 \leq \eta = 3/2$.}}
\label{fig:noimp}
\end{subfigure}
\centering
\begin{subfigure}{0.49\textwidth}
\includegraphics[width=0.7\textwidth]{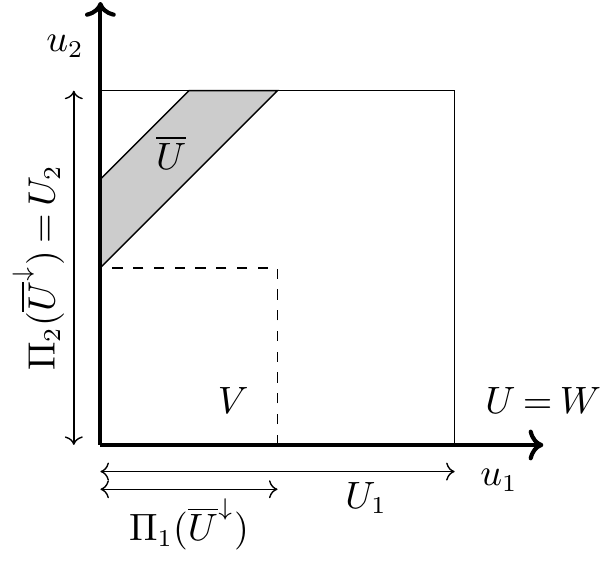}
\caption{\reviewChanges{$\alpha = 1/2 \leq u_2 - u_1 \leq 3/4 = \beta$}.}
\label{fig:imp}
\end{subfigure}
\caption{Illustration of shrinkage factors for the static supply chain example.}
\label{fig:ualphaproj}
\end{figure}

The projection of the uncertainty set is what affects the problem.
\reviewChanges{When the projections $\Pi_i\left(\cp^\downarrow\right)$ are equivalent to $U_i^\downarrow$ across all constraints $i$, $\ri=1$ and the coupled problem is equivalent to the constraint-wise formulation.}
However, when the projections shrink after coupling, $\ri<1$, and coupling might reduce the conservatism of the robust solution.
When $\gi<1$, an improvement through coupling is guaranteed, whereas $\gi=1$ does not provide additional information on the improvement.
Further results are stated in the following theorem.

\begin{theorem}
\label{thm:main}
Consider the robust problem~\eqref{eqn:ro} under \acrlong{cw} uncertainty and problem~\eqref{eqn:cp} under coupled uncertainty.
\reviewChanges{$U \subset \reals_+^m$ is convex}, and $0 < \zcp \leq \zro < \infty$.
Then the objective value is bounded as:
\begin{equation*}
\ri \leq \frac{\zcp}{\zro} \leq \gi.
\end{equation*}
\end{theorem}

\ifpreprint \begin{proof}[Proof]
\else
\proof{Proof.}
\fi
We first show the lower bound of the improvement.
Let $(\xcp,\ycp)$ be an optimal solution to~\eqref{eqn:cp}, \reviewChanges{which is then also an optimal solution to~\eqref{eqn:cp_proj}}. Since $\reviewChanges{\ri \, U^\downarrow \subseteq \Pi\left(\cp^\downarrow\right)}$, it is also feasible to problem
\begin{equation*}
{
\begin{array}{llclcl}
& \text{minimize} && {\vec{c}^T\vec{x} + \vec{d}^T\vec{y}} \\
& \text{subject to} &&
 \vec{a}_i^T\vec{x} + \vec{g}_i^T\vec{y}\geq u_i, \quad \forall \vec{u} \in \ri \, \reviewChanges{U^\downarrow}, \quad \forall i \in [m].
 \end{array}}
\end{equation*}
\reviewChanges{This problem is equivalent to}
\begin{equation*}
{
\begin{array}{llclcl}
& \text{minimize} && {\vec{c}^T\vec{x} + \vec{d}^T\vec{y}} \\
& \text{subject to} &&
 \vec{a}_i^T\vec{x} + \vec{g}_i^T\vec{y}\geq \ri u_i, \quad \forall \vec{u} \in \reviewChanges{U^\downarrow}, \quad \forall i \in [m].
 \end{array}}
\end{equation*}
Then $({\xcp}/{\ri},{\ycp}/{\ri})$ is feasible to~\eqref{eqn:ro} since
$$\vec{a}_i^T\xcp/\ri + \vec{g}_i^T\ycp / \ri \geq u_i, \quad \forall \vec{u} \in \reviewChanges{U^\downarrow}, \quad \forall i \in [m]. $$
Thus
\begin{equation*}
\zro \leq \zcp/\ri.
\end{equation*}

We next show the upper bound of the improvement in a similar approach. Consider the problem
\begin{equation*}
{
\begin{array}{llclcl}
& \text{minimize} && {\vec{c}^T\vec{x} + \vec{d}^T\vec{y}} \\
& \text{subject to} &&
 \vec{a}_i^T\vec{x} + \vec{g}_i^T\vec{y}\geq u_i, \quad \forall \vec{u} \in W, \quad \forall i \in [m].
\end{array}}
\end{equation*}
Since $W = \gi \, \reviewChanges{U^\downarrow}$, the problem is equivalent to
\begin{equation*}
{
\begin{array}{llclcl}
\displaystyle \zsi =
& \text{minimize} && {\vec{c}^T\vec{x} + \vec{d}^T\vec{y}} \\
& \text{subject to} &&
 \vec{a}_i^T\vec{x} + \vec{g}_i^T\vec{y}\geq \gi u_i, \quad \forall \vec{u} \in \reviewChanges{U^\downarrow}, \quad \forall i \in [m].
\end{array}}
\end{equation*}
Let $(\xro,\yro)$ be an optimal solution to problem~\eqref{eqn:ro}.
Then ${\gi} (\xro,\yro)$ is a feasible solution to this rescaled problem since
$$\vec{a}_i^T\gi \xro + \vec{g}_i^T\gi\yro \geq \gi u_i, \quad \forall \vec{u} \in \reviewChanges{U^\downarrow}, \quad \forall i \in [m]. $$
As $\reviewChanges{\Pi\left(\cp^\downarrow\right)} \subseteq W$, thus
\begin{equation*}
\zcp \leq \zsi \leq{\gi} \zro,
\end{equation*}
which completes the proof.
\ifpreprint \end{proof} \else
\Halmos
\endproof
\fi

\begin{corollary}
\label{cor:main}
The upper and lower bounds constructed in Theorem~\ref{thm:main} are tight.
\end{corollary}

\ifpreprint \begin{proof} \else
\proof{Proof.}
\fi
We consider scenario (\textit{b}) of the supply chain problem~\eqref{eq:scex} with $\alpha=1/2, \beta=3/4$.
We show in Example~\ref{ex:scrho} that $\ri=1/2, \gi=1$.
Thus according to Theorem~\ref{thm:main}, $1/2 \leq \zcp/\zro \leq 1$.
If $c_{11}=s_{11}=c_{22}=0, s_{22}=s_{12}=1,t=p=1$, then $\zro=1,\zcp=1$, $\zcp/\zro= 1 = \gi$.
If $c_{11}=s_{11}=1,c_{22}=s_{22}=s_{12}=0,t=p=1$, then $\zro=2,\zcp=1$, $\zcp/\zro =1/2 =\ri$.
Therefore both bounds are tight.
\ifpreprint \end{proof} \else
\Halmos
\endproof
\fi

\paragraph{\reviewChanges{A remark on the attainability of the bounds.}}
\reviewChanges{We established above that the constructed bounds are tight. A natural follow-up to seek is then the {\it conditions} under which the bounds become tight, which are intrinsically tied to the {\it coefficients} of the problem. We give sufficient conditions here.}

\begin{corollary}
\label{cor:attain_bounds_lower}
Let $(\xcp,\ycp)$ be the optimal solution to the coupled problem~\eqref{eqn:cp}, and consider $(\xcp/\ri,\ycp/\ri)$, a feasible solution to the robust problem~\eqref{eqn:ro}, constructed in the proof 
of Theorem~\ref{thm:main}. 
If the following holds for the robust problem~\eqref{eqn:ro}:
\begin{enumerate}[label=(\roman*)]
    \item for any constraint $i \in [m]$ that is tight at the worst case realization of $\vec{u}$, \ie,
$$\vec{a}_i^T\xcp/\ri + \vec{g}_i^T\ycp/\ri = \underset{\vec{u} \in U^\downarrow}{\max}~u_i,$$
the non-zero positions of $\vec{a}_i$ and $\vec{g}_i$ also take non-zero values in the respective objective coefficient vectors $\vec{c}$ and $\vec{d}$, and
\item for any constraint that is not tight, \ie,
$$\vec{a}_i^T\xcp/\ri + \vec{g}_i^T\ycp/\ri > \underset{\vec{u} \in U^\downarrow}{\max}~u_i,$$
the non-zero positions of $\vec{a}_i$ and $\vec{g}_i$ take values of zero in the respective objective coefficient vectors $\vec{c}$ and $\vec{d}$, if these objective coefficients are not set to be non-zero due to a tight constraint,
\end{enumerate}
then the lower bound of Theorem~\ref{thm:main} is tight.
\end{corollary}

\ifpreprint \begin{proof} \else
\proof{Proof.}
\fi
If the above conditions are satisfied, then $(\xcp/\ri,\ycp/\ri)$ is an optimal solution to the robust problem~\eqref{eqn:ro}. Therefore,  $\zro = \zcp/\ri$, and the lower bound is tight. We note that the example given in Corollary~\ref{cor:main} for the lower bound satisfies such conditions.
\ifpreprint \end{proof} \else
\Halmos
\endproof
\fi

\begin{corollary}
\label{cor:attain_bounds_upper}
Let $(\xro,\yro)$ be the optimal solution to the robust problem~\eqref{eqn:ro}, and consider the modified solution $(\gi\xro,\gi\yro)$. If  $(\gi\xro,\gi\yro)$ is a feasible solution to the coupled problem~\eqref{eqn:cp}, and the following holds for the coupled problem:
\begin{enumerate}[label=(\roman*)]
    \item for any constraint $i \in [m]$ that is tight at the worst case realization of $\vec{u}$, \ie,
$$\vec{a}_i^T\gi \xro + \vec{g}_i^T\gi\yro = \underset{\vec{u} \in \Pi\left(\cp^\downarrow\right)}{\max}~u_i,$$
the non-zero positions of $\vec{a}_i$ and $\vec{g}_i$ also take non-zero values in the respective objective coefficient vectors $\vec{c}$ and $\vec{d}$, and
\item for any constraint that is not tight, the non-zero positions of $\vec{a}_i$ and $\vec{g}_i$ take values of zero in the respective objective coefficient vectors $\vec{c}$ and $\vec{d}$, if these objective coefficients are not set to be non-zero due to a tight constraint,
\end{enumerate}
then the upper bound of Theorem~\ref{thm:main} is tight.
\end{corollary}

\ifpreprint \begin{proof} \else
\proof{Proof.}
\fi
If the above conditions are satisfied, then $(\gi\xro,\gi\yro)$ is an optimal solution to the coupled problem~\eqref{eqn:cp}. Therefore,  $\zcp = \gi\zro$, and the upper bound is tight. We note that the example given in Corollary~\ref{cor:main} for the upper bound satisfies such conditions.
\ifpreprint \end{proof} \else
\Halmos
\endproof
\fi

\reviewChanges{The sufficient conditions given above can be summarized as such: for the upper or lower bound to be tight, all variables with non-zero objective coefficients must be subject to the same scaling in the constructed optimal solution. In other words, the shrinkage factors of the uncertainty set along all relevant dimensions (dimensions $i$ that are associated with variables with non-zero objective coefficients) must be equivalent. 
In practice, the shrinkage factors along different dimensions are not equivalent, and thus optimization problems can be constructed to attain the full range of the bounds, by manipulating the comparative magnitudes of the coefficients values. 
As the optimization problem is invariant to scaling in the objective, an infinite number of coefficients can be constructed for the attainment of any value of $\zcp/\zro$ in the range $[\ri,\gi]$.}

\reviewChanges{In light of this observation, given an optimization problem, one may examine the shrinkage factors of the uncertainty sets along all dimensions, together with the comparative magnitudes of the corresponding constraint and objective coefficient values, to determine the possible extent of the coupling effect. 
In fact, by examining only the coefficients, it may be possible to identify dimensions along which an effective shrinkage factor may greatly impact the objective value. 
The decision-maker would then benefit from conducting a study into potential coupling effects in that dimension.
For instance, in the second example of Corollary~\ref{cor:main}, we note by the status of the coefficients that the uncertainty of significance is $u_1$, in which case a coupling set that restricts $u_1$ would bring about a significant reduction in cost.
}

\subsection{Adaptive robust optimization}
\label{sec:adapt}
We next study the effect of coupling on adaptive problems.
Whereas the projection of the uncertainty set is what affects the static problem, the uncertainty set itself affects the adaptive problem.
We define shrinkage factors for the adaptive case and derive theoretical bounds on the objectives.

\begin{definition}
\label{rhogamma}
We define
\reviewChanges{$\ari = \rho\left(U^\downarrow, \cp^\downarrow\right)$ and $\agi = \gamma\left(U^\downarrow, \cp^\downarrow\right)$} using functions from Definition~\ref{def:rhogam_func}.
We then construct sets
$V = \ari \, \reviewChanges{U^\downarrow}$ and
$W = \agi \, \reviewChanges{U^\downarrow}.$
\end{definition}

\reviewChanges{
\begin{lemma} 
Assume $U$ is a convex set.
Analogous to the static case, $\ari$ and $\agi$ exist, and
\begin{equation*}
   0 \leq \ari \leq \agi \leq 1.
\end{equation*}
\end{lemma}
}
\paragraph{Obtaining shrinkage factors.}
\reviewChanges{We obtain $\ari$ by solving the robust program 
\begin{equation*}
{
\begin{array}{llclcl}
\displaystyle \ari = & \text{maximize} & \multicolumn{3}{l}{\rho} \\
& \text{subject to} &
\rho \, \vec{u} \in \cp^\downarrow, && \forall \vec{u}\in U^\downarrow,
\end{array}}
\end{equation*}
which is a convex program for convex $U$ and $\cp$.
On the other hand, $\agi$, the maximum amount we can shrink $U^\downarrow$ to still contain $\cp^\downarrow$, can be computed by examining the constraint-wise projections; it is limited by the dimension with the least amount of shrinkage. Therefore,
$$\agi = \gi,$$
where $\gi$ is obtained by equation~\eqref{eq:gamma_static}.}
\begin{example}
\label{ex:scrhoadapt}
We illustrate the shrinkage factors with the supply chain example.
For scenario (\textit{a}) in Figure~\ref{fig:noimpadapt}, although coupling does not shrink the projections ($\ri = \gi = 1$ from the static case), the uncertainty set itself shrinks with
\reviewChanges{$\ari = \eta/2 = 3/4, \agi = 1$ in the adaptive case.
For scenario (\textit{b}) in Figure~\ref{fig:impadapt}, $\ari = 1-\alpha = 1/2 < 1$}, $\agi = 1$, same as the static case.
\end{example}

\begin{figure}[tb]
\centering
\begin{subfigure}{0.49\textwidth}
\centering
\includegraphics[width=0.7\textwidth]{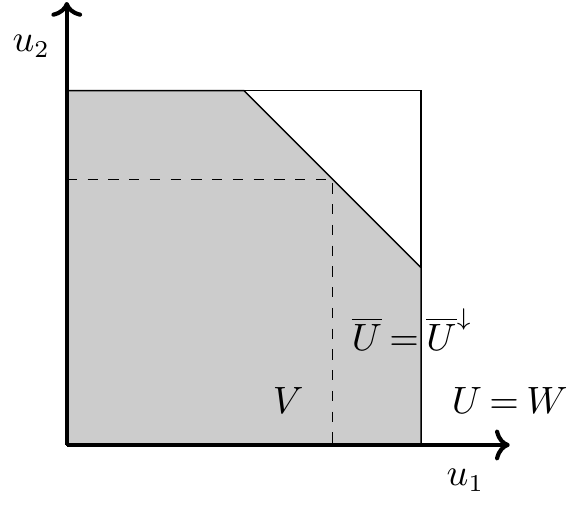}
\caption{\reviewChanges{$u_1+ u_2 \leq \eta = 3/2$}.}
\label{fig:noimpadapt}
\end{subfigure}
\begin{subfigure}{0.49\textwidth}
\centering
\includegraphics[width=0.7\textwidth]{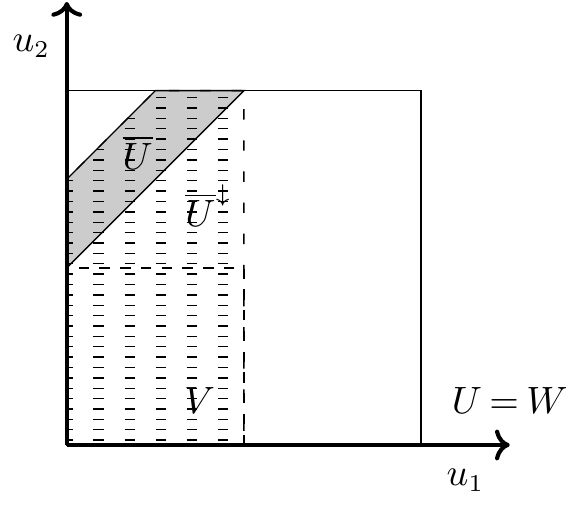}
\caption{\reviewChanges{$\alpha = 1/2 \leq u_2 - u_1 \leq 3/4 =\beta$.}}
\label{fig:impadapt}
\end{subfigure}
\caption{Illustration of shrinkage factors for the adaptive supply chain example.}
\end{figure}

\begin{theorem}
\label{thm:mainadapt}
Consider the adaptive robust problem~\eqref{eqn:aro} under \acrlong{cw} uncertainty and problem~\eqref{eqn:acp} under coupled uncertainty.
\reviewChanges{Assume $U \subset \reals_+^m$, $U$ is convex,} and $0 < \zacp \leq \zaro < \infty$. Then the objective value is bounded as:
\begin{equation*}
\ari \leq \frac{\zacp}{\zaro} \leq \agi.
\end{equation*}
\end{theorem}

\ifpreprint \begin{proof} \else
\proof{Proof.}
\fi
First to show the lower bound,
let $\left(\xacp,\yacp(\vec{u})\right)$ be an optimal solution to~\eqref{eqn:acp}.
\reviewChanges{Since $V \subseteq \cp^\downarrow$, it is also feasible for the problem under uncertainty set $V$, satisfying the constraints 
$$\vec{a}_i^T\xacp + \vec{g}_i^T \yacp\left(\vec{u}\right) \geq u_i, \quad \forall \vec{u} \in \ari U^\downarrow, \quad \forall i \in [m]. $$
Then, with a change of variables and dividing both sides by $\ari$, we have
$$\vec{a}_i^T\xacp/\ari + \vec{g}_i^T \yacp\left(\ari \vec{u}\right) / \ari \geq u_i, \quad \forall \vec{u} \in U^\downarrow, \quad \forall i \in [m]. $$
Therefore, $\left({\xacp}/{\ari},{\yacp\left(\ari \vec{u}\right)}/{\ari}\right)$ is feasible to~\eqref{eqn:aro}. As $\ari U^\downarrow \subseteq \cp^\downarrow$, we have
$$\max\limits_{\vec{u}\in\cp^\downarrow} \vec{d}^T \yacp(\vec{u}) \geq  \max\limits_{\vec{u}\in \ari U^\downarrow} \vec{d}^T \yacp(\vec{u}) = \max\limits_{\ari \vec{u}'\in \ari U^\downarrow} \vec{d}^T \yacp(\ari \vec{u}')= \max\limits_{\vec{u}\in U^\downarrow} \vec{d}^T \yacp(\ari \vec{u}).$$
Thus 
$$\zacp/\ari = \vec{c}^T \xacp/\ari + \max\limits_{\vec{u}\in\cp^\downarrow} \vec{d}^T \yacp(\vec{u})/\ari \geq  \vec{c}^T \xacp/\ari + \max\limits_{\vec{u}\in U^\downarrow} \vec{d}^T \yacp(\ari \vec{u}) \geq \zaro. $$
}
Next to show the upper bound,
suppose $\left(\xaro,\yaro(\vec{u})\right)$ is an optimal solution to~\eqref{eqn:aro}, \reviewChanges{the problem under $U^\downarrow$.
We must have 
$$\vec{a}_i^T\xaro + \vec{g}_i^T \yaro\left(\vec{u}\right)  \geq u_i, \quad \forall \vec{u} \in U^\downarrow, \quad \forall i \in [m],$$
which implies 
$$\agi \vec{a}_i^T\xaro + \agi \vec{g}_i^T \yaro\left(\vec{u}\right)  \geq \agi u_i, \quad \forall \vec{u} \in U^\downarrow, \quad \forall i \in [m].$$
Then, by a change of variables,
$$\agi \vec{a}_i^T\xaro + \agi \vec{g}_i^T \yaro\left(\vec{u}/\agi\right)  \geq u_i, \quad \forall \vec{u} \in \agi U^\downarrow, \quad \forall i \in [m].$$
Since $W =\agi U^\downarrow $, the solution $\left(\agi \xaro,\agi \yaro(\vec{u}/\agi)\right)$ is then feasible for the problem under $W$. Furthermore, as $\cp^\downarrow\subseteq W$, this solution is also feasible to~\eqref{eqn:acp}. Now, observe that 
$$ \max\limits_{\vec{u}\in U^\downarrow} \vec{d}^T \yaro(\vec{u}) = \max\limits_{\vec{u}\in \agi U^\downarrow} \vec{d}^T \yaro(\vec{u}/\agi) \geq \max\limits_{\vec{u}\in \cp^\downarrow} \vec{d}^T \yaro(\vec{u}/\agi),$$
and let the objective value of the problem under $W$ be $z_W$.
We then have
$$\agi \zaro = \agi \vec{c}^T \xaro + \agi \max\limits_{\vec{u}\in U^\downarrow} \vec{d}^T \yaro(\vec{u}) \geq z_W \geq \agi \vec{c}^T \xaro + \agi \max\limits_{\vec{u}\in \cp^\downarrow} \vec{d}^T \yaro(\vec{u}/\agi) \geq \zacp. $$
This completes the proof.}
\ifpreprint \end{proof} \else
\Halmos
\endproof
\fi

\begin{corollary}
\label{lem:tightadaptstatic_tight}
The upper and lower bounds constructed in Theorem~\ref{thm:mainadapt}
are tight.
\end{corollary}

\ifpreprint \begin{proof} \else
\proof{Proof.}
\fi
We consider scenario (\textit{a}) of the supply chain problem~\eqref{eq:scexadapt} with $\eta=3/2$.
We show in Example~\ref{ex:scrhoadapt} that $\ari=3/4, \agi=1$.
Thus according to Theorem~\ref{thm:mainadapt}, $3/4 \leq \zacp/\zaro \leq 1$.
If $c_{11}=s_{11}=c_{22}=s_{22}=s_{12}=1,t=p=1$, then $\zaro=4,\zacp=3$, $\zacp/\zaro= 3/4 = \ari$.
If $c_{11}=s_{11}=c_{22}=s_{22}=1,s_{12}=100,t=p=1$, then $\zaro=4,\zacp=4$, $\zacp/\zaro =1 =\agi$.
Therefore both bounds are tight.
\ifpreprint \end{proof} \else
\Halmos
\endproof
\fi

\paragraph{\reviewChanges{A remark on the attainability of the bounds.}}
\reviewChanges{
Analogous to the static case, we can establish sufficient conditions under which the bounds become tight. 
The conditions once again depend on the scaling of the uncertainty set along all dimensions. 
For simplicity, we give only the lower bound; the upper bound follows similarly.}

\begin{corollary}
\label{cor:attain_bounds_adapt}
Let $(\xacp,\yacp(\vec{u}))$ be the optimal solution to~\eqref{eqn:acp}, and consider $\left({\xacp}/{\ari},{\yacp\left(\ari \vec{u}\right)}/{\ari}\right)$, a feasible solution to~\eqref{eqn:aro}. If the following holds for the robust problem~\eqref{eqn:aro}:
\begin{enumerate}[label=(\roman*)]
    \item for any constraint $i \in [m]$ that is tight at the worst case realization of $\vec{u}$, \ie,
$$\underset{\vec{u} \in U^\downarrow}{\max}~\vec{a}_i^T\xacp/\ari + \vec{g}_i^T \yacp\left(\ari\vec{u}/\ari\right) = u_i,$$
the non-zero positions of $\vec{a}_i$ and $\vec{g}_i$ also take non-zero values in the respective objective coefficient vectors $\vec{c}$ and $\vec{d}$, and
\item for any constraint that is not tight (in which case the shrinkage factor for the dimension is different from $\ari$), the non-zero positions of $\vec{a}_i$ and $\vec{g}_i$ take values of zero in the respective objective coefficient vectors $\vec{c}$ and $\vec{d}$, if these objective coefficients are not set to be non-zero due to a tight constraint, and 
\item the solution satisfies $$\max\limits_{\vec{u}\in\cp^\downarrow} \vec{d}^T \yacp(\vec{u}) =  \max\limits_{\vec{u}\in \ari U^\downarrow} \vec{d}^T \yacp(\vec{u}),$$
\end{enumerate}
then the lower bound of Theorem~\ref{thm:mainadapt} is tight.
\end{corollary}

\ifpreprint \begin{proof} \else
\proof{Proof.}
\fi
If the above conditions are satisfied, then $\left({\xacp}/{\ari},{\yacp\left(\ari \vec{u}\right)}/{\ari}\right)$ is an optimal solution to the robust problem~\eqref{eqn:aro}. Therefore,  $\zaro = \zacp\ari$, and the lower bound is tight. We note that the example given in Corollary~\ref{lem:tightadaptstatic_tight} for the lower bound satisfies such conditions.
\ifpreprint \end{proof} \else
\Halmos
\endproof
\fi


\reviewChanges{Similar to the static case, the full range of the bounds can be attained, with the effect of coupling dependent on the shrinkage factors along each dimension of the uncertainty set, as well as the comparative magnitudes of the associated objective coefficients.
Therefore, an a priori analysis of the problem coefficients and shrinkage factors of potential coupled sets may be informative for the decision-maker, before solving the optimization problem.}

\reviewChanges{The same results can be extended to the upcoming sections, for more general objective and constraint functions, which we omit for brevity.} 

\subsection{Improvement of adaptive over static problems under coupled uncertainty}
\label{sec:adap_static}
In previous subsections, we have characterized the improvement of coupling for RO and ARO.
We compare the bounds in static and adaptive cases in the next corollary and then characterize the improvement of adaptability over static problems.
\begin{corollary}
\label{lem:comp}
Let $\ri, \ari, \gi, \agi$ be defined as before for adaptive problems~\eqref{eqn:aro},~\eqref{eqn:acp} and their corresponding static problems.
Then $\ri \geq \ari$ and $\gi = \agi$.
\end{corollary}

\ifpreprint \begin{proof} \else
\proof{Proof.}
\fi
Since \reviewChanges{$\cp^\downarrow \subseteq \Pi\left(\cp^\downarrow\right)$, we have $\ri \geq \ari$ by definition. In addition, from Section~\ref{sec:adapt}, we have $\gi = \agi$.}
\ifpreprint \end{proof} \else
\Halmos
\endproof
\fi
The corollary says that the lower bound of improvement in the adaptive case in Theorem~\ref{thm:mainadapt} is lower or equal to the static case in Theorem~\ref{thm:main} and the upper bounds in the two cases are the same.
A comparison of Example~\ref{ex:scrho} and Example~\ref{ex:scrhoadapt} serves as an illustration.
Therefore, coupling brings a larger potential improvement in the adaptive case than in the static case.
Furthermore, given a coupled uncertainty set, we can utilize the results to characterize the improvement of the adaptive robust objective over the static objective.

\begin{definition}
\label{def:adapt_rho}
$\reviewChanges{\arho = \rho\left(\Pi\left(\cp^\downarrow\right), \cp^\downarrow\right)}$ using the function from Definition~\ref{def:rhogam_func}.
\end{definition}

\reviewChanges{
We can use the previous definitions to directly define a bound on $\arho$.}
\reviewChanges{
\begin{lemma}
\label{lem:aro_arho}
     $\arho \geq \ari/\gi$.
\end{lemma}}

\reviewChanges{
\ifpreprint \begin{proof} \else
\proof{Proof.}
\fi
 Recall from Definition~\ref{def:rhogam} that $\Pi\left(\cp^\downarrow\right) \subseteq \gi U^\downarrow$. 
 Together with Definition~\ref{rhogamma}, 
we have
$$\frac{\ari}{\gi} \Pi\left(\cp^\downarrow\right) \subseteq   \ari U^\downarrow \subseteq \cp^\downarrow.$$
Then, the result follows from  Definition~\ref{def:adapt_rho}.
\ifpreprint \end{proof} \else
\Halmos
\endproof
\fi
}

\reviewChanges{In order to relate $\arho$ to the adaptive and static objectives, We require an assumption on the space for the adaptive solutions $\vec{y}(\vec{u})$. This is to establish the equivalence of the adaptive and static objectives under certain conditions.}
\begin{assumption}
\label{ass:compact}
Given an adaptive problem under uncertainty set $U$, assume that whenever $\vec{x}$ is feasible, there exists a compact set $Y_{\vec{x}} \subseteq Y$, \reviewChanges{where $Y$ is the space for $\vec{y}$}, such that for every $\vec{u} \in U$, satisfaction of the robust constraints implies $\vec{y}\in Y_{\vec{x}}$~\citep{ben2004adjustable}.
\end{assumption}

\reviewChanges{\begin{theorem}
\label{cor:mainadapt_rhs}
Consider the adaptive robust problem~\eqref{eqn:acp} with objective $\zacp$ and the corresponding static robust problem~\eqref{eqn:cp} with objective $\zcp$.
Assume $U \subset \reals_+^m$, $U$, $\cp$ are convex compact sets, and Assumption~\ref{ass:compact} holds for the adaptive problem under $\Pi\left(\cp^\downarrow\right)$. Suppose $0 < \zacp \leq \zcp < \infty$.
The objective value is bounded as:
\begin{equation*}
\frac{\zacp}{\zcp} \geq  {\arho}.
\end{equation*}
\end{theorem}
\ifpreprint \begin{proof} \else
\proof{Proof.}
\fi
The first inequality is established using the proof technique of Theorem~\ref{thm:mainadapt}. Starting from an optimal solution $\left(\xacp,\yacp(\vec{u})\right)$ to problem~\eqref{eqn:acp}, we construct a solution
 $\left({\xacp}/{\arho},{\yacp\left(\arho \vec{u}\right)}/{\arho}\right)$ feasible for the constraint-wise adaptive robust problem under $\Pi\left(\cp^\downarrow\right)$. Let the optimal value of this problem be $\hat{z}_{\text{acp}}$.
It follows that $\zacp/\arho  \geq \hat{z}_{\text{acp}}$.
 Since $\Pi\left(\cp^\downarrow\right)$ is constraint-wise, compact, and the adaptive problem under $\Pi\left(\cp^\downarrow\right)$ satisfies Assumption~\ref{ass:compact}, then
$\zcp = \hat{z}_{\text{acp}}$~\citep[Theorem 2.1]{ben2004adjustable}.
Therefore,
$$\zacp/\arho  \geq \hat{z}_{\text{acp}} = \zcp.$$
\ifpreprint \end{proof} \else
\Halmos
\endproof
\fi}

\reviewChanges{\begin{corollary}
\label{lem:tightadaptmain}
The bounds constructed in Theorem~\ref{cor:mainadapt_rhs} and Lemma~\ref{lem:aro_arho} are tight .
\end{corollary}
\ifpreprint \begin{proof} \else
\proof{Proof.}
\fi
Let $\zro, \zaro$ denote the objectives of the static and adaptive robust problems~\eqref{eqn:ro} and~\eqref{eqn:aro} under the constraint-wise uncertainty set $U^\downarrow$, and consider the case where $\Pi\left(\cp^\downarrow\right) = U^\downarrow$. Then, by Definition~\ref{def:rhogam},  Theorem~\ref{thm:main}, and the proof of Theorem~\ref{cor:mainadapt_rhs}, $\hat{z}_{\text{acp}} = \zaro = \zro = \zcp$, and $\gi = 1$. By Definitions~\ref{rhogamma} and~\ref{def:adapt_rho}, $\arho = \ari$. Thus
\begin{equation*}
    \frac{\zacp}{\zcp} = \frac{\zacp}{\zro} = \frac{\zacp}{\zaro} \geq \ari = \arho,
\end{equation*}
where the inequality follows from Theorem~\ref{thm:mainadapt}, and is proven to be tight in Corollary~\ref{lem:tightadaptmain}.
\ifpreprint \end{proof} \else
\Halmos
\endproof
\fi
}

\begin{lemma}
\label{lem:1/m}
\reviewChanges{For any convex set $\cp \subset \reals^m_+$},
$$\arho \geq \frac{1}{m}.$$
\end{lemma}

\ifpreprint \begin{proof} \else
\proof{Proof.}
\fi
\reviewChanges{Let $\Pi\left(\cp^\downarrow\right) = [0,\overline{d}_1] \times \dots \times [0,\overline{d}_m]$.
By definition, $\left(\overline{d}_1,0,\dots,0\right), \dots, \left(0,\dots,0,\overline{d}_m\right) \in \cp^\downarrow$.
Since $\cp^\downarrow$ is convex, we have $(1/m)(\overline{d}_1,\dots, \overline{d}_m) = (1/m) \left(\left(\overline{d}_1,0,\dots,0\right)+ \dots+ \left(0,\dots,0,\overline{d}_m\right)\right) \in \cp^\downarrow$.
Then $(1/m) \Pi\left(\cp^\downarrow\right) \subseteq \cp^\downarrow$.
Therefore, by Definition~\ref{def:adapt_rho}}, $\arho \geq 1/m$.
\ifpreprint \end{proof} \else
\Halmos
\endproof
\fi

The next example illustrates that there exist cases where ${\zacp}/{\zcp} = \arho = 1/m$.
\begin{example}
Consider scenario (\textit{a}) of the supply chain example with $\eta=1$, $\arho = \eta/2 = 1/2 = 1/m$.
If $c_{11}=s_{11}=c_{22}=s_{22}=s_{12}=t=p=1$, then $\zro=\zcp=4,\zacp=2$, and $\zacp/\zcp= \arho = 1/m$.
\end{example}
\paragraph{Obtaining the shrinkage factor.}
\reviewChanges{
We can also directly calculate $\arho$ using the robust program
\begin{equation}
\label{eq:calc_rhoadapt}
{
\begin{array}{llclcl}
\displaystyle \arho = & \text{maximize} & \multicolumn{3}{l}{\rho} \\
& \text{subject to} &
\rho \, \vec{u} \in \cp^\downarrow, && \forall \vec{u}\in \Pi\left(\cp^\downarrow\right),
\end{array}}
\end{equation}
which is a convex program for convex $\cp$.
}

\begin{example}
\label{ex:norm}
Consider the intersection of a constraint-wise set and a $\q$-norm based set
 $$\cp =  \left\{0 \leq \vec{u} \leq \alpha,\quad \|\vec{u}\|_\q \leq \beta\right\},$$
where $\alpha \leq \beta \leq \alpha m^{1/\q}$.
Then $\reviewChanges{\Pi\left(\cp^\downarrow\right)} = \left\{0 \leq \vec{u} \leq \alpha\right\}$.
Since $\arho$ is the maximum number such that
$V = \left\{\vec{u} \mid \|\vec{u}\|_\infty \leq \arho \alpha \right\} \subseteq \reviewChanges{\cp^\downarrow}$,
or equivalently $\|(\arho \alpha,\dots,\arho \alpha))\|_\q \leq \beta$, i.e.
$m (\arho \alpha)^\q \leq \beta^\q$.
Therefore $$\frac{\zacp}{\zcp} \geq \arho = \frac{\beta}{\alpha m^{1/\q}},$$
where the bound is tight.
If $\alpha=\beta$, then $$\frac{\zacp}{\zcp} \geq \arho = \frac{1}{m^{1/\q}}.$$
\end{example}

\paragraph{Comparison with other bounds.}
We compare our bound to other bounds from the literature that characterize the improvement of an adaptive problem over static problem.
Example~\ref{ex:norm} with $\alpha=\beta$ gives a family of uncertainty sets where our bound is tighter than the symmetry bound, $1/(1+m^{1/\q})$, in~\citep{bertsimas2011geometric,bertsimas2013approximability} and as tight as the approximation bound related to a measure of non-convexity of a transformation of the uncertainty set in~\cite{bertsimas2015tight}.
Although the approximation bound in~\cite{bertsimas2015tight} gives a tight characterization, it is not necessarily tractable to compute the measure of non-convexity for an arbitrary convex compact set.
Given the uncertainty set $\cp$ and $\vec{h}>0$, they define the transformation
\begin{equation*}
    T\left(\cp,\vec{h}\right) = \left\{\vec{B}^T \vec{\mu} \mid \vec{h}^T\vec{\mu} =1,\quad \vec{B}\in \cp,\quad \vec{\mu} \geq 0\right\}.
\end{equation*}
The measure of non-convexity for the transformation $T(\cp,\vec{h})$ is obtained by solving
\begin{equation*}
{
\begin{array}{clclcl}
\displaystyle  & \text{minimize} & \multicolumn{2}{l}{\alpha} &&   \\
&\text{subject to} &
\text{conv}\left(T\left(\cp,\vec{h}\right)\right) \subseteq \alpha T\left(\cp,\vec{h}\right), &&  \\
\end{array}}
\end{equation*}
where $T\left(\cp,\vec{h}\right)$ is nonconvex for some $\vec{h}>0$ if adaptability can improve the problem.
In comparison, our bound can be obtained \reviewChanges{by solving the convex program~\eqref{eq:calc_rhoadapt}.}

\reviewChanges{If the upper bounds of $\Pi\left(\cp^\downarrow\right) = [0,\overline{d}_1]\times\dots\times[0,\overline{d}_m]$ are unknown, we can first obtain $\overline{d}_i$ by solving the simple convex programs
\begin{equation}
\label{eq:maxu}
{
\begin{array}{llclcl}
\displaystyle  & \overline{d}_i =  & \text{maximize} & \multicolumn{2}{l}{u_i} \\
& &\text{subject to} &
\vec{u} \in \cp^\downarrow.
\end{array}}
\end{equation}}
Thus our bound is in general easier to compute.
Moreover, instead of constructing a transformation of the uncertainty set in~\cite{bertsimas2015tight}, our bound provides easy geometric interpretation directly on the uncertainty set itself and associates the improvement with the comparison with \acrlong{cw} sets.

\section{Theoretical improvement under constraint coefficient uncertainty}
\label{sec:theorycoeff}
In Section~\ref{sec:theory}, we consider problems where the uncertainty appears on the right hand side of the constraints.
We extend the result to problems with uncertain coefficients in the constraints.
Proof of results in this section can be found in Appendix~\ref{append:proofcoeff}.

\reviewChanges{In this section, we can no longer ensure that solving the robust optimization problem with respect to an uncertainty set $U$ is equivalent to solving it under its down-hull $U^\downarrow$. We thus redefine the shrinkage factors as follows.}
\reviewChanges{
\begin{definition}
    \label{def:shrinkage_new}
    Let $\ri = \rho\left(U, \Pi\left(\cp\right)\right)$,$\gi = \gamma\left(U, \Pi\left(\cp\right)\right)$,
$\ari = \rho\left(U, \cp\right)$ and $\agi = \gamma\left(U, \cp\right)$. In addition, let $\arho = \rho(\Pi\left(\cp\right),\cp)$.
\end{definition}
}
\reviewChanges{Now, to ensure that the shrinkage factors exist, $\ie$, $\ari \cp \subseteq \cp \subseteq \agi U \subseteq U$, we require that $\{0\} \subseteq \cp \subseteq U$. Given any convex sets $U',\cp'$, we can construct such sets $U,\cp$ through the use of a {\it translation factor}. To this end, we define the {\it symmetry} of an uncertainty set.}
\reviewChanges{\begin{definition}
    \label{def:sym}{Symmetry of a convex set~\citep{bertsimas2013approximability}.}
    Given a nonempty compact convex set $U \subseteq \reals^m$ and a
point $\vec{u} \in U$, we define the symmetry of $\vec{u}$ with respect to $U$ as follows:
$$ \textup{\textbf{sym}}(\vec{u},U) \define \max\{\alpha\geq 0 : \vec{u}+ \alpha(\vec{u}-\vec{u'}) \in U,~\forall \vec{u'}\in U\}.$$
\end{definition}}
\reviewChanges{
The \emph{point of symmetry} of a convex coupled set $\cp'$ is then $$ \vec{u}^s = \underset{\vec{u} \in \cp' }{\arg\max}~\textup{\textbf{sym}}(\vec{u},\cp').$$}
\reviewChanges{Setting the translation factor to $\vec{u}^s$, we note that as $\{\vec{u}^s\}\subseteq \cp' \subseteq U'$, we have
$$ \{0\} \subseteq U \define U'-\vec{u}^s,\quad \{0\} \subseteq \cp \define \cp'-\vec{u}^s,\quad \{0\} \subseteq \Pi\left(\cp\right) \define \Pi\left(\cp'\right)-\vec{u}^s.$$ }
\reviewChanges{
We can then reformulate the robust problems with respect to the newly defined sets, where all functions $f(\vec{u})$,~$\forall\vec{u} \in U'$ is expressed as $ f(\vec{u}^s+ \vec{u})$,~$\forall\vec{u} \in U$.}
\reviewChanges{Therefore, from here onwards we can assume without loss of generality that $\{0\} \subseteq \cp$.}

\reviewChanges{As a remark, the above translation can be done with any point in $\cp'$, not just the point of symmetry. The point of symmetry is chosen such that, for most cases, the shrinkage factors can be maximized. However, for sets $\cp'$ that could be translated to satisfy $\{0\} \subseteq \cp \subseteq U \subset \reals^m_+$, this translation should be used instead of the point of symmetry, to maximize the potential shrinkage factors.}

\subsection{Static robust optimization}
Consider robust problems where the coefficients in the constraints are uncertain
\begin{equation}
\label{eqn:ro2}
{
\begin{array}{llclcl}
\displaystyle \zro =
& \text{maximize} && {\vec{c}^T\vec{x}} \\
& \text{subject to} &&
 \vec{u}_i^T\vec{x}\leq b_i, \quad \forall \vec{u} \in U, \quad \forall i \in [m],
\end{array}}
\end{equation}
where $\vec{x} \in \reals^n$ is the decision variable and $\vec{c} \in \reals^n, b_i \in \reals, \vec{u}_i \in \reals^p, p = n$,
and
\begin{equation}
\label{eqn:cp2}
{
\begin{array}{llclcl}
\displaystyle \zcp =
& \text{maximize} && {\vec{c}^T\vec{x}} \\
& \text{subject to} &&
 \vec{u}_i^T\vec{x}\leq b_i, \quad \forall \vec{u} \in \cp, \quad \forall i \in [m].
\end{array}}
\end{equation}

\begin{theorem}
\label{thm:main2}
Consider the robust problem~\eqref{eqn:ro2} under \acrlong{cw} uncertainty and problem~\eqref{eqn:cp2} under coupled uncertainty.
Assume $\left\{0\right\} \subset \cp$, $U$ is convex,
$\ri > 0$, and $0 < \zro \leq \zcp < \infty$.
Then the objective value is bounded as:
\begin{equation*}
\frac{1}{\gi} \leq \frac{\zcp}{\zro} \leq \frac{1}{\ri}.
\end{equation*}
Furthermore, the bounds are tight.
\end{theorem}

\paragraph{Obtaining shrinkage factors.}
The shrinkage factors can be obtained by solving the following robust programs.
\begin{equation*}
{
\begin{array}{llclcl}
\displaystyle \ri = &  \text{maximize} & \multicolumn{3}{l}{\rho} \\
&\text{subject to} &
\rho \, \vec{u}_i \in \Pi_i\left(\cp\right), && \forall \vec{u}\in U, \quad \forall i \in [m],
\end{array}}
\end{equation*}
\begin{equation*}
{
\begin{array}{llclcl}
\displaystyle  \gi = & \text{minimize} & \multicolumn{3}{l}{\gamma} \\
& \text{subject to} &
\overline{\vec{u}}_i \in \gamma \, U, && \forall \overline{\vec{u}}_i\in \Pi_i\left(\cp\right), \quad \forall i \in [m].
\end{array}}
\end{equation*}
In other words, $$\ri = \min\limits_{i \in [m]} r_i, \quad \gi = \max\limits_{i \in [m]} s_i,$$
where for each $i$,
$r_i \in \reals_+$ is the maximum number such that
$$\uri := r_i \,\reviewChanges{U_i} \subseteq \Pi_i\left(\cp\right),$$
and $s_i \in \reals_+$ is the minimum number such that
$$\Pi_i\left(\cp\right) \subseteq s_i \,\reviewChanges{U_i} =: \usi.$$
These problems are convex since we assume that $U$ and $\cp$ are convex sets.

\begin{example}
\label{ex:l1l2}
Consider uncertainty sets $$U_1 = \left\{\vec{u}_1 \mid \|\vec{u}_1\|_1 \leq \alpha\right\}\subset \reals^p, \quad U_2 = \left\{\vec{u}_2 \mid \|\vec{u}_2\|_2 \leq \beta\right\}\subset \reals^p,$$
where $\beta \leq \alpha \leq \sqrt{p} \beta$, and
$$\cpc = \left\{\vec{u} \mid \vec{u}_1=\vec{u}_2\right\}, \quad \cp = U \cap \cpc = (U_1 \times U_2) \cap \cpc.$$
An illustration for $p=2$ is shown in Figure~\ref{ex1}, where the shaded region denotes
$$\Pi_i\left(\cp\right) = U_1 \cap U_2 = \reviewChanges{\left\{\vec{u}_i \mid \|\vec{u}_i\|_1 \leq \alpha, \quad \|\vec{u}_i\|_2 \leq \beta\right\}} , \quad \forall i \in [2].$$

\begin{figure}[tb]
\centering
\begin{subfigure}{0.49\textwidth}
\centering
\includegraphics[width=0.57\textwidth]{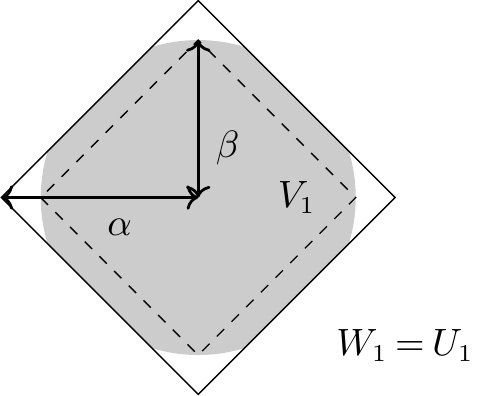}
\end{subfigure}
\begin{subfigure}{0.49\textwidth}
\centering
\includegraphics[width=0.57\textwidth]{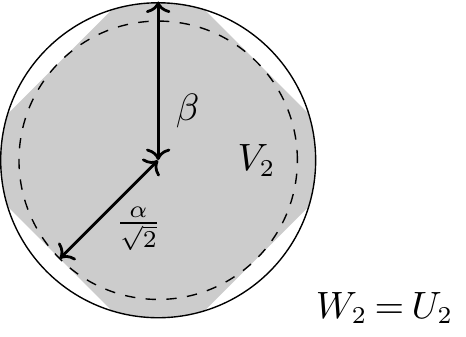}
\end{subfigure}
\caption{Illustration of shrinkage factors for Example~\ref{ex:l1l2} with $p=2$.}
\label{ex1}
\end{figure}

We have
$$\ri = \min \left\{\frac{\beta}{\alpha}, \frac{\alpha}{\sqrt{p}\beta}\right\},\quad \gi = 1, \quad 1 \leq \frac{\zcp}{\zro} \leq \max \left\{\frac{\alpha}{\beta},\frac{\sqrt{p}\beta}{\alpha}\right\}
.$$
Particularly, if $\alpha=\beta$, then $$\ri =\frac{1}{\sqrt{p}},\quad \gi = 1, \quad 1 \leq \frac{\zcp}{\zro} \leq \sqrt{p}.$$
Therefore we can potentially have a significant improvement as $p$ increases.
\end{example}

\subsection{Adaptive robust optimization}
Consider adaptive robust problems with uncertain constraint coefficients
\begin{equation}
\label{eqn:aro2}
{
\begin{array}{llclcl}
\displaystyle  \zaro = & \text{maximize} && \multicolumn{3}{l}{\vec{c}^T\vec{x} + \min\limits_{\vec{u}\in U} \vec{d}^T \vec{y}(\vec{u})}\vspace{3pt} \\
&\text{subject to} &&  \vec{u}_i^T
\begin{bmatrix*}[l] \vec{x}\\ \vec{y}(\vec{u}) \end{bmatrix*}
\leq b_i, && \forall \vec{u} \in U, \quad \forall i \in [m],
\end{array}}
\end{equation}
where $\vec{x} \in \reals^{n_1}$ is the first-stage decision variable and $\vec{y} \in \reals^{n_2}$ is the adaptive variable,
$\vec{c} \in \reals^{n_1}, \vec{d} \in \reals^{n_2}, b_i \in \reals, \vec{u}_i \in \reals^p$, $p=n_1+n_2$,
and
\begin{equation}
\label{eqn:acp2}
{
\begin{array}{llclcl}
\displaystyle  \zacp = & \text{maximize} && \multicolumn{3}{l}{\vec{c}^T\vec{x} + \min\limits_{\vec{u}\in \cp} \vec{d}^T \vec{y}(\vec{u})}\vspace{3pt} \\
&\text{subject to} &&  \vec{u}_i^T
\begin{bmatrix*}[l] \vec{x}\\ \vec{y}(\vec{u}) \end{bmatrix*}
\leq b_i, && \forall \vec{u} \in \cp, \quad \forall i \in [m],
\end{array}}
\end{equation}

\begin{theorem}
\label{thm:mainadapt2}
Consider the robust problem~\eqref{eqn:aro2} under \acrlong{cw} uncertainty and problem~\eqref{eqn:acp2} under coupled uncertainty.
Assume $\left\{0\right\} \subset \cp$, $U$ is convex,
$\ari > 0$, and $0 < \zaro \leq \zacp < \infty$. Then the objective value is bounded as:
\begin{equation*}
\frac{1}{\agi} \leq \frac{\zacp}{\zaro} \leq \frac{1}{\ari}.
\end{equation*}
Furthermore, the bounds are tight.
\end{theorem}

\paragraph{Obtaining shrinkage factors.}
The shrinkage factors can be obtained by solving the robust programs
\begin{equation*}
{
\begin{array}{llclcl}
\displaystyle \ari = & \text{maximize} & \multicolumn{3}{l}{\rho} \\
& \text{subject to} &
\rho \, \vec{u} \in \cp, && \forall \vec{u}\in U,
\end{array}}
\end{equation*}
\begin{equation*}
{
\begin{array}{llclcl}
\displaystyle \agi = & \text{minimize} & \multicolumn{3}{l}{\gamma} \\
&\text{subject to} &
\overline{\vec{u}} \in \gamma \, U, && \forall \overline{\vec{u}} \in \cp,
\end{array}}
\end{equation*}
which are convex programs for convex $U$ and $\cp$.

We illustrate and compare the bounds in the static and adaptive cases by the following example.
\begin{example}
\label{ex:boundl2}
Consider uncertainty sets
$$U_i = \left\{\vec{u}_i \mid \|\vec{u}_i\|_\infty \leq \alpha\right\} \subset \reals^p, \quad \forall i \in [m],\quad \cpc = \left\{\vec{u} \mid \|\vec{u}\|_2 \leq \beta\right\},\quad \cp = U \cap \cpc,$$
with an illustration shown in Figure~\ref{fig:ex2}.
\begin{figure}[tb]
\centering
\begin{subfigure}{0.49\textwidth}
\centering
\includegraphics[width=0.7\textwidth]{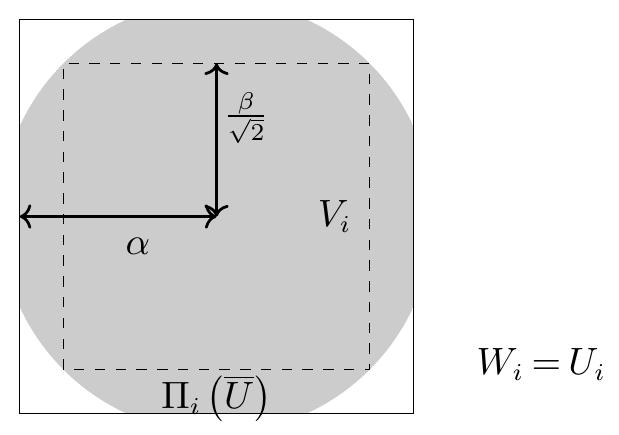}
\caption{Static case with $p=2$}
\label{ex2ro}
\end{subfigure}
\begin{subfigure}{0.49\textwidth}
\centering
\includegraphics[width=0.7\textwidth]{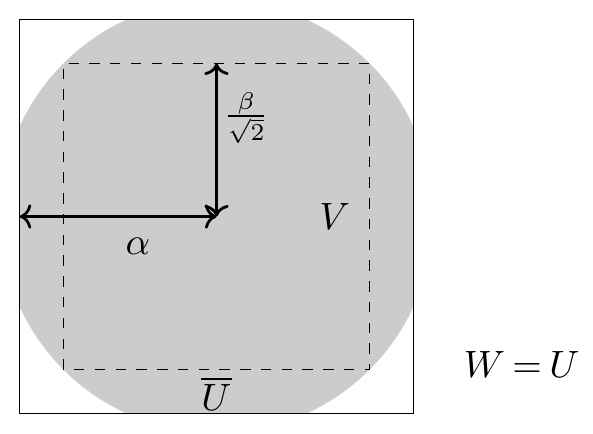}
\caption{Adaptive case with $p=1,m=2$}
\label{ex2aro}
\end{subfigure}
\caption{Illustration of shrinkage factors for Example~\ref{ex:boundl2}.}
\label{fig:ex2}
\end{figure}
For the static case, \reviewChanges{the projections do not depend on $m$, therefore the bounds are as follows,}
$$\max\left\{\frac{\alpha}{\beta},1\right\}  \leq \frac{\zcp}{\zro} \leq \max\left\{\frac{\sqrt{p}\alpha}{\beta},1\right\}.$$
For the adaptive case,
$$\max\left\{\frac{\alpha}{\beta},1\right\}  \leq \frac{\zacp}{\zaro} \leq \max\left\{\frac{\sqrt{mp}\alpha}{\beta},1\right\}.$$
The bounds on the improvement with varying parameters and sizes are illustrated in Figure~\ref{fig:bound}, with the static objective range shown in the brown shaded region and the additional range for the adaptive case shown in the red region.

\begin{figure}[tb]
\centering
\begin{subfigure}{0.31\textwidth}
\includegraphics[width=\textwidth]{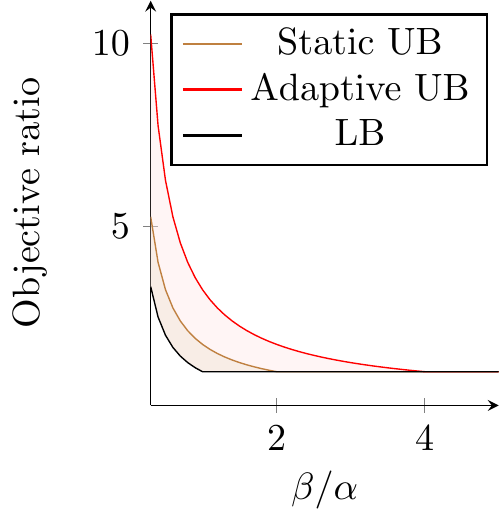}
    \caption{With $p=4,m=4$.}
    \label{fig:boundbeta}
\end{subfigure}
\begin{subfigure}{0.31\textwidth}
\includegraphics[width=\textwidth]{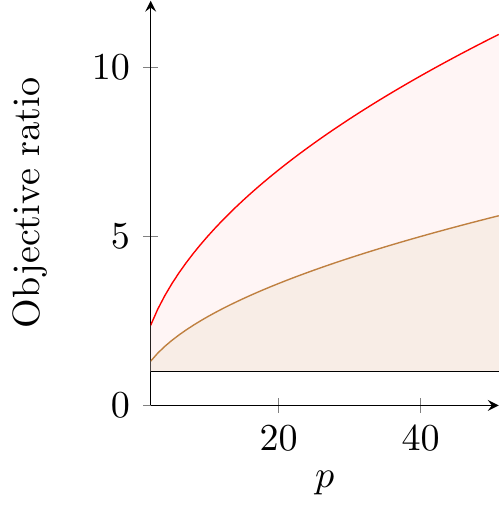}
    \caption{With $m=4,\beta/\alpha = 1$.}
     \label{fig:boundp}
\end{subfigure}
\begin{subfigure}{0.31\textwidth}
\includegraphics[width=\textwidth]{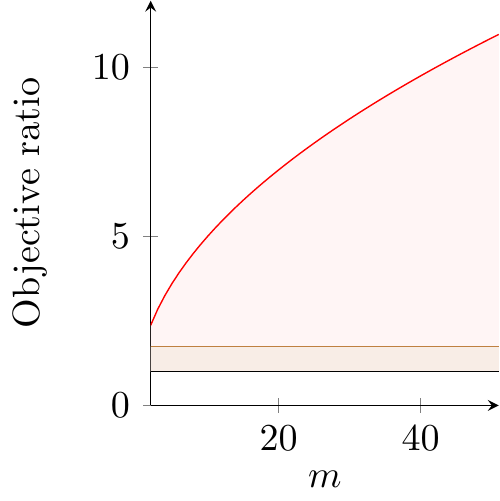}
    \caption{With $p=4,\beta/\alpha = 1$.}
     \label{fig:boundm}
\end{subfigure}
\caption{Upper and lower bounds vs coupling parameters for Example~\ref{ex:boundl2}.}
\label{fig:bound}
\end{figure}
\end{example}

\subsection{Improvement of adaptive over static problems under coupled uncertainty}
\label{sec:adap_static2}
Similar to the case when the uncertainty occurs on the right hand side, we characterize the benefit of adaptability of static solutions for coupled problems.

\begin{theorem}
\label{cor:mainadapt}
Consider the adaptive robust problem~\eqref{eqn:acp2} and the corresponding static robust problem with objective $\zcp$. 
Assume $\left\{0\right\} \subset \cp$, $\Pi\left(\cp\right)$ is compact and convex, and Assumption~\ref{ass:compact} holds for the adaptive problem under $\Pi\left(\cp\right)$.
The objective value is bounded as:
\begin{equation*}
\frac{\zacp}{\zcp} \leq \frac{1}{\arho} \reviewChanges{\leq \frac{\gi}{\ari}}.
\end{equation*}
Furthermore, the bounds are tight.
\end{theorem}

\begin{lemma}
\label{lem:mp}
\reviewChanges{For any convex set satisfying~$\{0\} \subset \cp \subset \reals_+^{mp}$},
$$\frac{1}{\arho}\leq {mp}.$$
\end{lemma}

Moreover, there exist cases where ${\zacp}/{\zcp} = \frac{1}\arho = mp$.

\paragraph{Obtaining the shrinkage factor.}
\reviewChanges{
The shrinkage factor $\arho$ can be obtained by solving the robust program
\begin{equation*}
{
\begin{array}{llclcl}
\displaystyle \ari = & \text{maximize} & \multicolumn{3}{l}{\rho} \\
& \text{subject to} &
\rho \, \vec{u} \in \cp, && \forall \vec{u}\in \Pi\left(\cp\right).
\end{array}}
\end{equation*}}

\begin{example}
\label{ex:normp}
Consider the intersection of a constraint-wise box set and a $\q$-norm based set
 $$\cp =  \left\{\|\vec{u}\|_\infty \leq \alpha, \|\vec{u}\|_\q \leq \beta\right\},$$
where $\alpha \leq \beta \leq \alpha (mp)^{1/\q}$.
We construct a \acrlong{cw} uncertainty set $\Pi\left(\cp\right) = \left\{\|\vec{u}\|_\infty \leq \alpha\right\}$.
Since $\arho$ is the maximum number such that
$V = \left\{\vec{u} \mid \|\vec{u}\|_\infty \leq \arho \alpha \right\} \subseteq \cp$,
or equivalently, $mp (\arho \alpha)^\q \leq \beta^\q$,
thus $$\frac{\zacp}{\zaro} \leq \frac{1}\arho = \frac{\alpha }{\beta}(mp)^{1/\q},$$
where the bound is tight.
If $\alpha=\beta$, then $$\frac{\zacp}{\zaro} \leq \frac{1}\arho = {(mp)^{1/\q}}.$$
\end{example}

\paragraph{Comparison with results under right hand side uncertainty.}
We discuss similarities and differences in the results when the uncertainty occurs on the right hand side (Section~\ref{sec:theory}) and on the constraint coefficient (Section~\ref{sec:theorycoeff}).
The proofs of theorems follow the same procedure as the ones in the previous section:
First, we construct auxiliary sets $V$ and $W$ by rescaling the original uncertainty sets with shrinkage factors, which bound the uncertainty set of interest.
Next, we rescale the optimal solution of the original problem with corresponding factors to construct feasible solutions to optimization problems with the rescaled uncertainty sets.
Finally, the objective value bounds follow with respect to the rescaling factors.
Yet, the specific forms of results change due to the different problem structures between the two sections:
As the uncertainty and the decision variable are on the same side of the constraints (as opposed to different sides in previous section), the rescaling uses the inverses of the shrinkage factors in bounds of this section.
In correspondence with the change from minimizing to maximizing the objective function, the directions of inequality signs to bound the ratio between objective values also flip.
Furthermore, since the size of uncertainty increases from $m$ to $mp$, the potential improvement of adaptive over static problems is larger and grows with $p$ in Section~\ref{sec:adap_static2} compared with Section~\ref{sec:adap_static}.

\section{Extension to nonlinear problems}
\label{sec:theorynl}
Our results can be generalized to problems in which the constraints and objectives are not necessarily linear.
We discuss three cases, when the uncertainty affects the right hand side of the constraints, affects the coefficients of the constraints, and affects both objectives and constraints.
Proof of results in this section can be found in Appendix~\ref{append:proofnl}.

\subsection{Uncertain right hand side}
Consider problems of the form
\begin{equation}
\label{eqn:aronl}
{
\begin{array}{llclcl}
\displaystyle  \zaro = & \text{minimize} && {f_1(\vec{x}) + \max\limits_{\vec{u} \in U} f_2(\vec{y}(\vec{u}))} \\
& \text{subject to} &&
 g_i(\vec{x}, \vec{y}(\vec{u}))\geq u_i, \quad \forall \vec{u} \in U, \quad \forall i \in [m].
\end{array}}
\end{equation}
and
\begin{equation}
\label{eqn:acpnl}
{
\begin{array}{llclcl}
\displaystyle  \zacp = & \text{minimize} && {f_1(\vec{x}) + \max\limits_{\vec{u} \in \cp} f_2(\vec{y}(\vec{u}))} \\
& \text{subject to} &&
 g_i(\vec{x}, \vec{y}(\vec{u}))\geq u_i, \quad \forall \vec{u} \in \cp, \quad \forall i \in [m].
\end{array}}
\end{equation}
\reviewChanges{As these are problems with right hand side uncertainty, we see in Appendix~\ref{append:down} that we can replace the uncertainty sets $\cp, U \in \reals^m_+$ with their down-hulls, $\cp^\downarrow$, $U^\downarrow$. In this section, we then use the corresponding definitions of shrinkage factors.}

\begin{assumption}
\label{ass:concave}
Assume the case that $f_1,f_2$ are concave and satisfy $f_1(0) = 0,f_2(0) = 0$;
for all $i\in [m]$, $g_i$ is convex in $\vec{x}$ and $\vec{y}(\vec{u})$ and satisfies $g_i(0,0)=0$.
\end{assumption}

\begin{theorem}
\label{thm:mainnl}
Consider adaptive problems~\eqref{eqn:aronl}, \eqref{eqn:acpnl}, and their static problems with objectives $\zro$ and $\zcp$.
\reviewChanges{Assume $U \subset \reals^m_+$ is convex},
$\zacp>0$ and $\zro < \infty$.
If Assumption~\ref{ass:concave} holds,
\begin{equation*}
\frac{\zcp}{\zro} \geq \ri,\quad \frac{\zacp}{\zaro}\geq \ari.
\end{equation*}
Furthermore, the bounds are tight.
\end{theorem}

\begin{theorem}
\label{cor:mainadaptnl}
Consider the adaptive robust problem~\eqref{eqn:acpnl} and the corresponding static robust problem with objective $\zcp$.
Under the same assumptions in Theorem~\ref{thm:mainnl} and assume that $\Pi\left(\cp\right)$ is compact, the objective value is bounded as:
\begin{equation*}
\frac{\zacp}{\zcp} \geq  {\arho} \reviewChanges{ \geq \ari}.
\end{equation*}
Furthermore, the bounds are tight.
\end{theorem}

\begin{assumption}
\label{ass:convex}
Assume the case that $f_1,f_2$ are convex and satisfy $f_1(0) = 0,f_2(0) = 0$;
for all $i\in [m]$, $g_i$ is concave in $\vec{x}$ and $\vec{y}(\vec{u})$ and satisfies $g_i(0,0)\geq0$.
\end{assumption}

\begin{theorem}
\label{thm:mainnlconv}
Consider adaptive problems~\eqref{eqn:aronl}, \eqref{eqn:acpnl}, and their static problems with objectives $\zro$ and $\zcp$.
\reviewChanges{Assume $U \subset \reals_+^m$ is convex}, $\zacp>0$ and $\zro < \infty$.
If Assumption~\ref{ass:convex} holds,
\begin{equation*}
\frac{\zcp}{\zro}\leq \gi , \quad \frac{\zacp}{\zaro}\leq \agi.
\end{equation*}
Furthermore, the bounds are tight.
\end{theorem}

\subsection{Uncertain constraint coefficients}
Consider the following adaptive problems and their corresponding static problems, where the uncertainty affects the constraint coefficients.
\begin{equation}
\label{eqn:aronl2}
{
\begin{array}{llclcl}
\displaystyle  \zaro = & \text{maximize} && {f_1(\vec{x}) + \min\limits_{\vec{u} \in U} f_2(\vec{y}(\vec{u}))} \\
& \text{subject to} &&
 g_i(\vec{x}, \vec{y}(\vec{u}),\vec{u}_i)\leq b_i, \quad \forall \vec{u} \in U, \quad \forall i \in [m].
\end{array}}
\end{equation}
\begin{equation}
\label{eqn:acpnl2}
{
\begin{array}{llclcl}
\displaystyle  \zacp = & \text{maximize} && {f_1(\vec{x}) + \min\limits_{\vec{u} \in \cp} f_2(\vec{y}(\vec{u}))} \\
& \text{subject to} &&
 g_i(\vec{x}, \vec{y}(\vec{u}),\vec{u}_i)\leq b_i, \quad \forall \vec{u} \in \cp, \quad \forall i \in [m].
\end{array}}
\end{equation}

\begin{assumption}
\label{ass:concave2}
Assume the case that $f_1,f_2$ are concave and satisfy $f_1(0) = 0,f_2(0) = 0$;
for all $i\in [m]$, $g_i$ is convex in $\vec{x}$ and $\vec{y}(\vec{u})$, concave in $\vec{u}_i$, satisfies $g_i(0,0,\vec{u}_i)\leq 0,\forall \vec{u}\in U$, and $g_i(\vec{x},\vec{y}(\vec{u}),0)\geq 0, \forall \vec{x},\vec{y}(\vec{u})$ feasible.
\end{assumption}

\begin{theorem}
\label{thm:mainnl2}
Consider adaptive problems~\eqref{eqn:aronl2}, \eqref{eqn:acpnl2}, and their static problems with objectives $\zro$ and $\zcp$.
Assume $0 \in \cp$, $U$ is convex, $\ari > 0$, $\zacp>0$ and $\zro < \infty$.
If Assumption~\ref{ass:concave2} holds,
\begin{equation*}
\frac{\zcp}{\zro} \leq \frac{1}{\ri},\quad \frac{\zacp}{\zaro}\leq \frac{1}{\ari}.
\end{equation*}
Furthermore, the bounds are tight.
\end{theorem}

\begin{theorem}
\label{cor:mainadaptnl2}
Consider the adaptive robust problem~\eqref{eqn:acpnl2} and the corresponding static robust problem with objective $\zcp$.
Assume $\Pi\left(\cp\right)$ is compact and convex.
Under the same assumptions in Theorem~\ref{thm:mainnl2}, the objective value is bounded as:
\begin{equation*}
\reviewChanges{
\frac{\zacp}{\zcp} \leq  \frac{1}{\arho} \leq   \frac{1}{\ari}.}
\end{equation*}
Furthermore, the bounds are tight.
\end{theorem}

\begin{assumption}
\label{ass:convex2}
Assume the case that $f_1,f_2$ are convex and satisfy $f_1(0) = 0,f_2(0) = 0$;
for all $i\in [m]$, $g_i$ is concave in $\vec{x}$ and $\vec{y}(\vec{u})$, convex in $\vec{u}_i$, satisfies $g_i(0,0,\vec{u}_i)\geq 0,\forall \vec{u}\in \cp$, and $g_i(\vec{x},\vec{y}(\vec{u}),0)\leq 0, \forall \vec{x},\vec{y}(\vec{u})$ feasible.
\end{assumption}

\begin{theorem}
\label{thm:mainnlconv2}
Consider adaptive problems~\eqref{eqn:aronl2}, \eqref{eqn:acpnl2}, and their static problems with objectives $\zro$ and $\zcp$.
Assume $0 \in \cp$, $U$ is a convex set, $\zacp > 0$ and $\zro < \infty$.
If Assumption~\ref{ass:convex2} holds,
\begin{equation*}
\frac{\zcp}{\zro} \geq \frac{1}{\gi},\quad \frac{\zacp}{\zaro}\geq \frac{1}{\agi}.
\end{equation*}
Furthermore, the bounds are tight.
\end{theorem}

\subsection{Uncertain constraint and objective}
\label{sec:cons_obj}
We now bound the effect of coupling in case of uncertain constraint coefficients and objective.
Consider the following adaptive problems and their corresponding static problems, where the uncertainty affects both the constraint coefficients and the objective.
\begin{equation}
\label{eqn:aronl3}
{
\begin{array}{llclcl}
\displaystyle  \zaro = & \text{maximize} && \min\limits_{\vec{u} \in U} ({f_1(\vec{x},\vec{u}) +  f_2(\vec{y}(\vec{u}),\vec{u})}) \\
& \text{subject to} &&
 g_i(\vec{x}, \vec{y}(\vec{u}),\vec{u}_i)\leq b_i, \quad \forall \vec{u} \in U, \quad \forall i \in [m].
\end{array}}
\end{equation}
\begin{equation}
\label{eqn:acpnl3}
{
\begin{array}{llclcl}
\displaystyle  \zacp = & \text{maximize} && \min\limits_{\vec{u} \in \cp} ({f_1(\vec{x},\vec{u}) +  f_2(\vec{y}(\vec{u}),\vec{u})}) \\
& \text{subject to} &&
 g_i(\vec{x}, \vec{y}(\vec{u}),\vec{u}_i)\leq b_i, \quad \forall \vec{u} \in \cp, \quad \forall i \in [m].
\end{array}}
\end{equation}

\begin{assumption}
\label{ass:concave3}
Assume the case that $f_1$ is concave in $\vec{x}$, convex in $\vec{u}$, $f_2$ is concave in $\vec{y}(\vec{u})$, convex in $\vec{u}$, and for all $\vec{u}\in U$, satisfy $f_1(0,\vec{u}) = 0,f_2(0,\vec{u}) = 0$;
for all $\vec{u}\in U$, $\vec{x},\vec{y}(\vec{u})$ feasible and $0 \leq \alpha \leq 1$, $f_1(\vec{x},\vec{u}) +  f_2(\vec{y}(\vec{u}),\vec{u}) \geq \alpha (f_1(\vec{x},\alpha\vec{u}) +  f_2(\vec{y}(\vec{u}),\alpha\vec{u}))
$;
for all $i\in [m]$, $g_i$ is convex in $\vec{x}$ and $\vec{y}(\vec{u})$, concave in $\vec{u}_i$, satisfies $g_i(0,0,\vec{u}_i)\leq 0,\forall \vec{u}\in U$, and $g_i(\vec{x},\vec{y}(\vec{u}),0)\geq 0, \forall \vec{x},\vec{y}(\vec{u})$ feasible.
\end{assumption}

\begin{theorem}
\label{thm:mainnl3}
Consider adaptive problems~\eqref{eqn:aronl3}, \eqref{eqn:acpnl3}, and their static problems with objectives $\zro$ and $\zcp$.
Assume $0 \in \cp$, $U$ is convex, $\ari > 0$, $\zacp>0$ and $\zro < \infty$.
If Assumption~\ref{ass:concave3} holds,
\begin{equation*}
\frac{\zcp}{\zro} \leq \frac{1}{\ri^2},\quad \frac{\zacp}{\zaro}\leq \frac{1}{\ari^2}.
\end{equation*}
\end{theorem}

\begin{theorem}
\label{cor:mainadaptnl3}
Consider the adaptive robust problem~\eqref{eqn:acpnl3} and the corresponding static robust problem with objective $\zcp$.
Assume $\Pi\left(\cp\right)$ is compact and convex, \reviewChanges{and Assumption~\ref{ass:compact} holds for the adaptive problem under $\Pi\left(\cp\right)$.}
Under the same assumptions in Theorem~\ref{thm:mainnl3}, the objective value is bounded as:
\begin{equation*}
\frac{\zacp}{\zcp} \leq  \frac{1}{\arho^2} \reviewChanges{ \leq \frac{1}{\ari^2} }.
\end{equation*}
\end{theorem}

\begin{assumption}
\label{ass:convex3}
Assume the case that for all $\vec{u}\in U$, $f_1$ is convex in $\vec{x}$, $f_2$ is convex in $\vec{y}(\vec{u})$, and satisfy $f_1(0,\vec{u}) = 0,f_2(0,\vec{u}) = 0$;
for all $\vec{u}\in U$, $\vec{x},\vec{y}(\vec{u})$ feasible, $f_1(\vec{x},\vec{u}) +  f_2(\vec{y}(\vec{u}),\vec{u}) \leq \agi (f_1(\vec{x},\agi\vec{u}) +  f_2(\vec{y}(\vec{u}),\agi\vec{u}))
$;
for all $i\in [m]$, $g_i$ is concave in $\vec{x}$ and $\vec{y}(\vec{u})$, convex in $\vec{u}_i$, satisfies $g_i(0,0,\vec{u}_i)\geq 0,\forall \vec{u}\in \cp$, and $g_i(\vec{x},\vec{y}(\vec{u}),0)\leq 0, \forall \vec{x},\vec{y}(\vec{u})$ feasible.
\end{assumption}

\begin{theorem}
\label{thm:mainnlconv3}
Consider adaptive problems~\eqref{eqn:aronl3}, \eqref{eqn:acpnl3}, and their static problems with objectives $\zro$ and $\zcp$.
Assume $0 \in \cp$, $U$ is a convex set, $\zacp > 0$ and $\zro < \infty$.
If Assumption~\ref{ass:convex3} holds,
\begin{equation*}
\frac{\zcp}{\zro} \geq \frac{1}{\gi^2},\quad \frac{\zacp}{\zaro}\geq \frac{1}{\agi^2}.
\end{equation*}
\end{theorem}

\begin{lemma}
\label{lem:tightadapt2}
The bounds constructed in Theorem~\ref{thm:mainnl3}, Theorem~\ref{cor:mainadaptnl3}, and Theorem~\ref{thm:mainnlconv3} are tight.
\end{lemma}

\paragraph{Comparison with results for linear problems.}
We discuss similarities and differences in the results when the problems are linear (Section~\ref{sec:theory} and Section~\ref{sec:theorycoeff}) versus nonlinear (Section~\ref{sec:theorynl}).
In proofs of this section, some inequalities result from properties of convexity and concavity.
As a result, we make assumptions on two separate cases based on convexity / concavity of functions, where either lower bounds or upper bound apply (not both at the same time).
Since linear functions are both convex and concave, we achieve the lower and upper bounds simultaneously when the problem is linear as a special case.
The mathematical expressions of the bounds remain the same for both linear and nonlinear problems under right hand side uncertainty, and respectively under constraint coefficient uncertainty.
We further extend to problems where uncertainty appears in both the objective and the constraint coefficient in Section~\ref{sec:cons_obj}.
The two occurrences cause a multiplication effect and characterize the bounds with the shrinkage factors squared, resulting in more potential of improvement from coupling.

\section{Computational experiments}
\label{sec:computation}
In this section, we study the benefit of uncertainty coupling in robust and adaptive robust optimization via numerical experiments.
We implement various solution methods to three computational problems under coupled uncertainty in supply chain management, portfolio optimization, and lot sizing in a network.
We demonstrate the numerical effect of the coupling relative to the theoretical bounds, analyze how coupling parameters affect the objectives practically, and compare the runtime of different solution methods.

\paragraph{Solution methods.}
We implement different solution methods to solve linear robust and adaptive robust problems under coupled uncertainty, for both cases of right hand side and coefficient uncertainty.
For static problems, we consider three methods: Constraint-wise formulation with direct projections, Robust counterpart, and Cutting plane algorithm.
For adaptive problems, we consider three methods: Robust counterpart with linear decision rules, Benders decomposition, Finite scenario approach.
We describe each of the solution methods in detail in Section~\ref{sec:solution} of the Appendix.
We execute the numerical tests on Intel Xeon E5-2650 in Python 3.6. The optimization problems are solved using the Gurobi solver~\citep{optimization2016gurobi}.

\paragraph{Performance metrics.}
We evaluate the objective improvement resulting from coupling, the effect of coupling parameters on the objectives, and the performance of different solution methods.
We obtain objectives of coupled RO and ARO problems relative to the objective of \acrlong{cw} problems, namely $\zcp/\zro$ and $\zacp/\zaro$, as well as the running time of different solution methods for problems of different dimensions.
We show the average relative objective and running time with one standard deviation as shaded areas in the plots.

\subsection{Supply chain management}
We consider a two-stage supply chain network, in which products are first shipped from $J$ sources through $K$ centers.
After demands at the stores are realized, products are shipped from the $K$ centers to $M$ stores. We have introduced the problem for a small network in Section~\ref{sub:supply_chain_example}, and we formulate a more general network.
The uncertainty $u_i \in \reals$ denotes the demand of store $i$.
Decision variables are $x_{jk}$, the quantity of products shipped from source $j$ to center $k$, and $y_{ki}(\vec{u})$, shipment from center $k$ to store $i$, which can be adjusted to realizations of the demand.
Let $c_{jk},t_{jk}$ denote the unit cost and capacity of the products shipped from source $j$ to center $k$.
Let $s_{ki},p_{ki}$ denote the unit cost and capacity of the products shipped from center $k$ to store $i$.
We have the following adaptive robust formulation:

\begin{equation*}
{
\begin{array}{llclcl}
\displaystyle \text{minimize} & \multicolumn{3}{l}{ \sum\limits_{j=1}^J \sum\limits_{i=1}^M  c_{jk} x_{jk} + \max\limits_{\vec{u}\in U} \sum\limits_{i=1}^M \sum\limits_{k=1}^K s_{ki} y_{ki}(\vec{u})
} \\
\text{subject to} & \sum\limits_{k=1}^K y_{ki}(\vec{u}) \geq u_i, && \forall \vec{u} \in U, \quad \forall i \in [M] \\
& 0 \leq x_{jk} \leq t_{jk}, && \forall j \in [J], \quad \forall k\in [K] \\
& 0 \leq y_{ki} \leq p_{ki}, && \forall i \in [M], \quad \forall k\in [K] \\
& \sum\limits_{j=1}^J x_{jk}  \geq \sum\limits_{i=1}^M y_{ki}(\vec{u})
, &&  \forall \vec{u} \in U, \quad \forall k\in [K].
\end{array}}
\end{equation*}
The objective is to minimize the total transportation cost. The first constraint requires that each store satisfies the demand under uncertainty. The second and third constraints denote the capacities for the amount of products shipped on each route. The fourth constraint ensures flow conservation such that the outflow from the centers is no more than the inflow to the centers in the network.

We use a box set for the \acrlong{cw} uncertainty set
\begin{equation*}
    U = \left\{0 \leq \vec{u} \leq 1\right\}.
\end{equation*}
The stores are divided into groups of two and three.
Stores within the same group are adjacent to each other, and customers either visit all of them together or only the first one.
As a result, all stores in each group have a common portion of uncertain demand and one of them has an additional uncertain portion of demand from another source.
Let $G$ denote the set of all groups.
For each group $l \in G$, we impose coupling constraints to model the relationship between the stores in the group $l$ with parameter $\alpha_l$:
$$u_l^1 \geq u_l^2+\alpha_l,$$ for each group of two stores and
$$u_l^1 \geq u_l^2+\alpha_l, \quad u_l^1 \geq u_l^3+\alpha_l,$$ for each group of three stores.
Moreover, the total demand in the stores is restricted as
\begin{equation*}
    \|\vec{u}\|_2 \leq \Gamma.
\end{equation*}
We then have a coupled uncertainty set
\begin{equation*}
    \cp = \left\{\vec{u} \mid 0 \leq \vec{u} \leq 1,\quad \|\vec{u}\|_2 \leq \Gamma,\quad u_l^1 \geq u_l^{2}+\alpha_l, \quad u_l^1 \geq u_l^{3}+\alpha_l,\quad \forall l \in G\right\}.
\end{equation*}
We also consider the corresponding static problem, when $\vec{y}(\vec{u})$ does not adapt to uncertainty $\vec{u}$.
The static and adaptive problems under constraint-wise uncertainty set have the same objective.
We study how parameters of the coupling uncertainty set $\alpha$ and $\Gamma$ affect the problem, and evaluate the objectives and running times of the solution methods.

\paragraph{Numerical example.}

For each dimension, we set $J=K=M$.
We obtain 50 problem instances by drawing $s_{ki} \sim U(0,5)$ and setting $t_{jk}=p_{ki}=10$.
We draw $c_{jk}$ from $U(0,5)$ if the path is feasible or from $U(0,5001)$ if the path is not available, each with probability 50\%.
To ensure feasibility to meet the demands, we set at least one available path from some source to each center.
Benders decomposition and cutting plane methods allow 1000 maximum number of iterations and $10^{-3}$ tolerance. Benders decomposition starts with the coupled static optimal solution and 100 random uncertainty starting points.

\paragraph{Coupling effect of parameter $\alpha.$}
When $\Gamma \geq \sqrt{m}$, only the coupling $u_l^1 \geq u_l^{2(3)}+\alpha_l, \forall l \in G$ affects the uncertainty set.
We let $\alpha_l = \alpha, \forall l \in G$.
In Figure~\ref{fig:sc_improv_alpha}, we show the objective of the coupled problem in different dimensions relative to the objective of the constraint-wise problem with different values of $\alpha \in [0,1]$.
From Example~\ref{ex:scrhoadapt},
$$1 - \alpha \leq \frac{\zacp}{\zaro} = \frac{\zcp}{\zro} \leq 1,$$
with the possible range shown in the brown shaded region.
When $0 <\alpha \leq 1,$ coupling reduces groupings of uncertainty set projections and the cost in proportion to $\alpha$ both computationally and theoretically.
Since the relative objective improvement reflects an average uncertainty set shrinkage among the groups of stores, the average improvement is independent of $M$ and the standard deviation decreases with $M$.

\begin{figure}[tb]
\begin{subfigure}{0.49\textwidth}
    \centering
    \includegraphics[width=\textwidth]{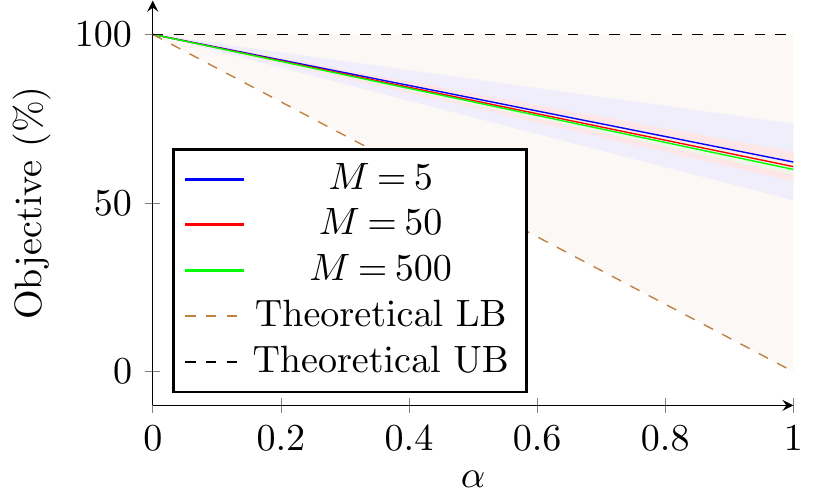}
\caption{Objectives vs $\alpha$ ($\Gamma=\sqrt{M})$.}
	\label{fig:sc_improv_alpha}
\end{subfigure}
\begin{subfigure}{0.49\textwidth}
    \centering
   \includegraphics[width=\textwidth]{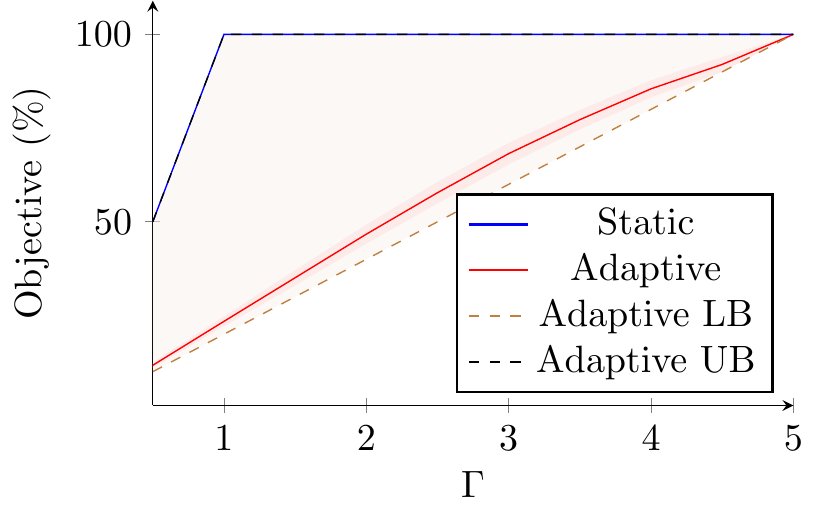}
\caption{Objectives vs $\Gamma$ ($M=25, \alpha=-1$).}
	\label{fig:sc_improv_gamma}
\end{subfigure}
\caption{Relative objectives of coupled problems vs coupling parameters in supply chain problems.}
\end{figure}

\paragraph{Coupling effect of parameter $\Gamma.$}
When $\alpha \leq -1$, \reviewChanges{we note that the constraints $u_l^1 \geq u_l^{2}+\alpha_l$ and $u_l^1 \geq u_l^{3}+\alpha_l$ hold true for all $0 \leq \vec{u} \leq 1$, so} only the coupling $\|\vec{u}\|_2 \leq \Gamma$ affects the uncertainty set.
In Figure~\ref{fig:sc_improv_gamma}, we show the relative objective of the coupled problem with different values of $\Gamma\in [0.5,5]$ for $M=25$.
From Example~\ref{ex:boundl2},
$$\frac{\Gamma}{\sqrt{M}} \leq \frac{\zacp}{\zaro} \leq \min\left\{\Gamma,1\right\} = \frac{\zcp}{\zro}.$$
Coupling improves the static problem when $\Gamma<1$ and possibly improves the adaptive problem when $\Gamma < \sqrt{M}$.
As $\Gamma$ increases, the coupling affects the uncertainty set less, and the relative coupled objective as well as the theoretical lower bound for the adaptive case grow linearly with $\Gamma$.
We see that with some choices of parameters $\alpha$ and $\Gamma$, coupling can alleviate conservatism significantly.

\paragraph{Objectives and running times.}
For the case when both $\alpha$ and $\Gamma$ affect the objectives, we randomly draw problem instances with $\alpha_l \sim U(0.1,0.2)$ and $\Gamma \sim U(\sqrt{M}/2, 3\sqrt{M}/4)$.
In Figure~\ref{fig:sc}, we show the objective values and running times for different methods on problems with $M$ robust constraints, $M^2$ here-and-now variables, and $M^2$ adaptive variables.
Since the choice of $\Gamma$ scales with $\sqrt{m}$,
the relative objectives remain independent of $M$.
For 100\% of the static coupled problems, the robust counterpart and cutting planes give the same results
with the cutting planes taking longer time.
The coupled robust counterpart improves the objective while maintaining similar computational complexity as the constraint-wise robust counterpart.
Benders decomposition solves the coupled adaptive problem with larger objective improvement and additional computational effort.

\begin{figure}[tb]
\begin{subfigure}{0.49\textwidth}
\centering
\includegraphics[width=\textwidth]{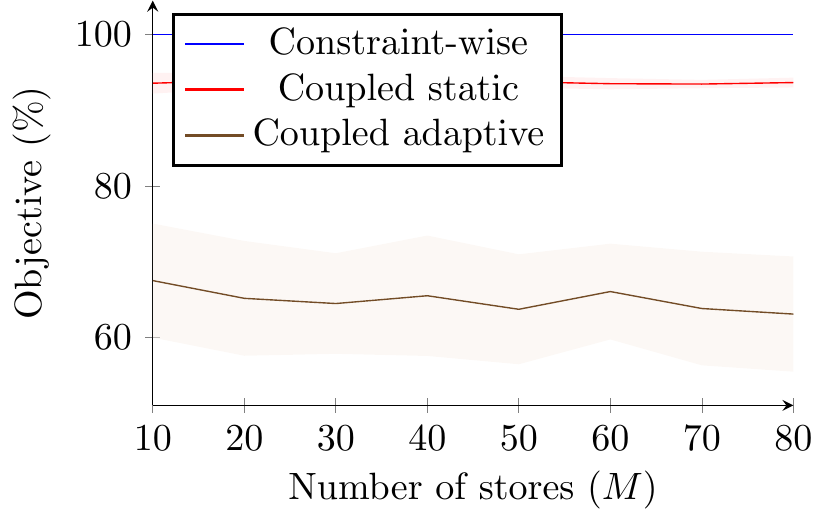}
	\end{subfigure}
\begin{subfigure}{0.49\textwidth}
    \centering
    \includegraphics[width=\textwidth]{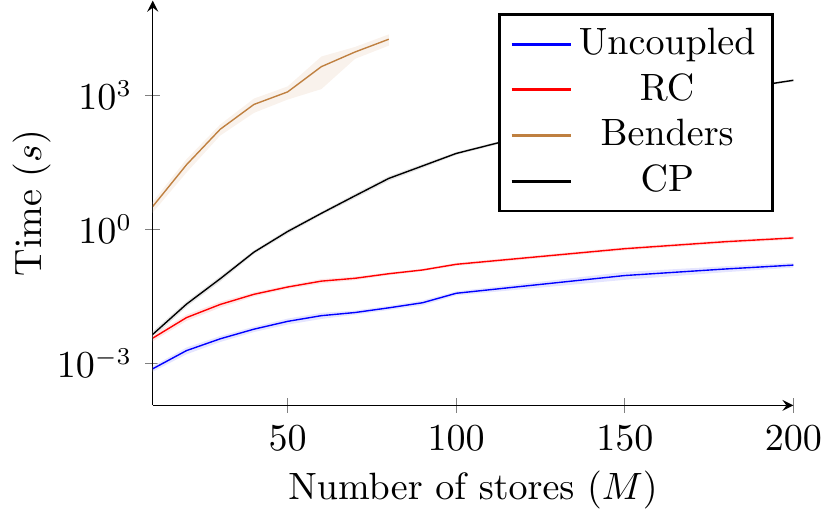}
\end{subfigure}
\caption{Relative objectives and running times for supply chain (minimization) problems.}
	\label{fig:sc}
\end{figure}

\subsection{Portfolio management}
We consider a portfolio optimization problem in which an investor chooses from $m$ different assets to invest in.
For $i \in [m],$ the investor decides to invest fraction $x_i$ of the total investments into asset $i$.
Both returns and prices of the assets are affected by a variety of uncertain sources.
The investor wants to maximize the returns robustly against uncertain scenarios.
For each asset $i$, we have the uncertainty
$\vec{u}_i = (\vec{v}_{i}, \vec{z}_i) \in \reals^8$,
where $\vec{v}_{i} \in \reals^4$ denotes common uncertain sources, including interest rate, S\&P 500 annual return, Commercial Construction Index, and Consumer Price Index,
and $\vec{z}_i$ denotes asset-specific uncertain sources, including supply, demand, dividend rate, and industry performance.
For each asset $i$, the uncertainty affects both the unit return $\overline{c}_i + \vec{Q}_i \vec{u}_i$ and unit price
$\overline{a}_i + \vec{P}_i \vec{u}_i$, where $\Vec{Q}_i \in \reals^8, \Vec{P}_i \in \reals^8, \overline{c}_i \in \reals, \overline{a}_i \in \reals$.
We use auxiliary variables $b_i \in \reals$ to track the amount spent on each asset $i$ and $t_i \in \reals$ to track the return on each asset.
Each asset belongs to one of the sectors $S_1,\dots,S_K$ with $g_k \in \reals$ denoting the limit of the total amount of assets invested in sector $S_k$.
We enforce a total budget limit $b \in \reals$.
The problem is formulated as follows:
\begin{equation*}
\begin{array}{llclcl}
\displaystyle \text{maximize} & \multicolumn{3}{l}{\sum\limits_{i=1}^m t_i} \\
\text{subject to}  &  t_i \leq (\overline{c}_i + \vec{Q}_i \vec{u}_i) x_i, && \forall u_i \in U_i, \quad \forall i \in [m]\\
& (\overline{a}_i + \vec{P}_i \vec{u}_i) x_i \leq b_i,
&& \forall u_i \in U_i, \quad \forall i \in [m]\\
& \sum\limits_{i=1}^m  b_i \leq b \\
& \sum\limits_{i \in S_k} x_i \leq g_k, &&\forall k\in [K]\\
& \sum\limits_{i=1}^m x_i = 1\\
& \vec{x} \geq 0.
\end{array}
\end{equation*}
The objective is to maximize the minimal total return of the assets under all uncertain scenarios, described in the first constraint. The second and third constraints restrict the total amount of money spent on the investment within a budget.
The fourth constraint encourages the investor to diversify their investment over different sectors.
The final constraints represent the decision variable as a fraction of the total investment.

For each asset $i$, we have a polyhedral uncertainty set $U_i$ that captures relationships between uncertain sources:
\begin{equation*}
    U_i = \left\{\vec{M}_i \vec{u}_i \leq \vec{s}_i\right\}, \quad \forall i \in [m],
\end{equation*}
where $\vec{M}_i \in \reals^{q \times 8}$ and $\vec{s}_i \in \reals^q$ with $q \in \reals$ denoting the number of constraints in the \acrlong{cw} uncertainty set $U_i$.

Since common uncertainties denote the same source, they should be the same for each constraint:
\begin{equation*}
    \vec{v}_1 = \dots = \vec{v}_m.
\end{equation*}
Then the coupled uncertainty set becomes
\begin{equation*}
    \cp = \left\{\vec{u} \mid \vec{u}_i = (\vec{v}_i, \vec{z}_i),\quad \vec{M}_i \vec{u}_i \leq \vec{s}_i, \quad \forall i \in [m], \quad \vec{v}_1 = \dots = \vec{v}_m\right\}.
\end{equation*}
Furthermore, idiosyncratic uncertainties are also related among \reviewChanges{similar assets. For example, consumer demand of products is influenced by substitution effects, which affects the performance and valuation of companies, and thereby influences asset returns. 
Companies from the same or related sectors may have similar industry performance, and company performances may be affected by peers and competitors.}
We characterize the couplings among asset-specific uncertainties by a polyhedral set with constraints
$$\vec{D} \vec{z} \leq \vec{d}.$$
The resulting further coupled uncertainty set is
\begin{equation*}
    \cp' = \left\{\vec{u} \mid \vec{u}_i = (\vec{v}_i, \vec{z}_i),\quad \vec{M}_i \vec{u}_i \leq \vec{s}_i, \quad \forall i \in [m], \quad \vec{v}_1 = \dots = \vec{v}_m, \quad \vec{D} \vec{z} \leq \vec{d}\right\}.
\end{equation*}

\paragraph{Numerical example.}

For each dimension $m$, we obtain 100 problem instances by drawing $c_{i} \sim U(80,100)$,  $P_{ij} \sim U(-0.5,0.5)$,  $Q_{ij} \sim U(-2.5,2.5)$, for all $i \in [m], j \in [8]$, and setting $b = 100, q = 30, K=5, g_k = 8/m$ for each sector $S_k$.
We first generate $\overline{\vec{M}} \in \reals^{q \times 8}$, $\vec{w} \in \reals^8$, and $\overline{\vec{s}} \in \reals^q$, with $\overline{{M}}_{jk} \sim U(-2.5,2.5),w_k \sim N(0,1), $ and $\overline{s}_{j} \sim \overline{\vec{M}} \vec{w} + U(0,100)$.
Then for each $i$, we generate $\vec{M}_i \sim \overline{\vec{M}} + U(-0.5,0.5), \vec{s}_i \sim \overline{\vec{s}} + U(-0.5,0.5)$, so that we have a polyhedral set for each $U_i$ which is slightly different from each other. They are not necessarily identical sets since each asset has different ranges of uncertain sources.
For each sector $S_k$, we further impose coupling constraints
$u_i \geq u_j+\alpha_j$ for some asset $i \in S_k$ and all $j \in S_k, j \neq i$ with $\alpha_j \sim U(0,1)$.
We set 1000 maximum number of iterations and $10^{-3}$ tolerance for cutting planes method.

\paragraph{Objectives and running times.}
In Figure~\ref{fig:pt}, we show the objective values and running times of the different methods on problems with $2m$ robust constraints and $m$ decision variables.
Enforcing the common sources $\vec{v}_i$ to be the same can obtain an improvement in the returns.
In addition, coupling between idiosyncratic uncertainties $\vec{z}_i$ can result in additional improvement.
Compared with the \acrlong{cw} RC, the cutting plane approach takes a shorter time, and coupled RC takes slightly longer.
The supply chain and portfolio optimization examples show that robust counterpart and cutting plane methods are tractable in high dimensions and neither one dominates the other in solving times.

\begin{figure}[tb]
\begin{subfigure}{0.49\textwidth}
\includegraphics[width=\textwidth]{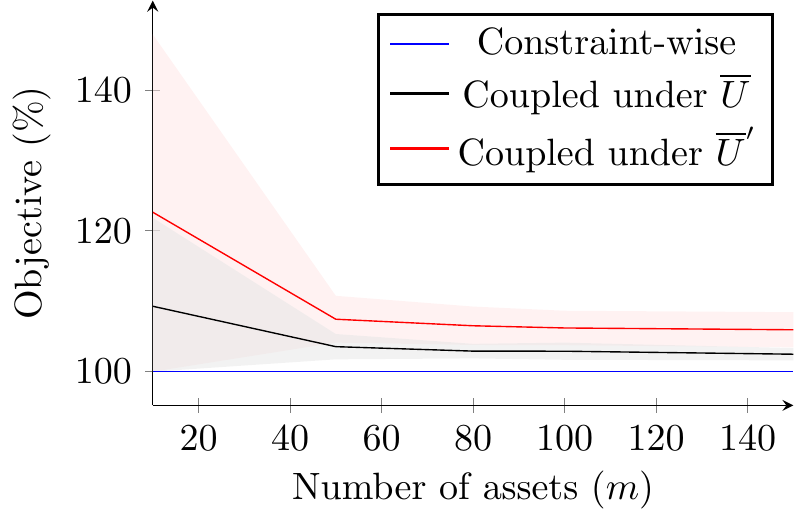}
\end{subfigure}
\begin{subfigure}{0.49\textwidth}
\includegraphics[width=\textwidth]{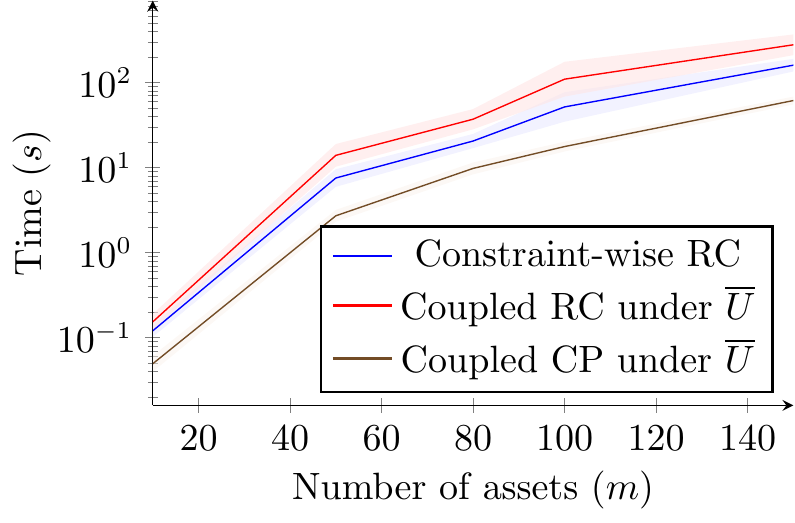}
\end{subfigure}
    \caption{Relative objectives and running times for portfolio optimization (maximization) problems.}
	\label{fig:pt}
\end{figure}

\subsection{Lot sizing on a network}
We consider lot sizing on a network, in which stock is first allocated to stores and then transported between stores to meet all demand.
The uncertainty $u_i \in \reals$ for store $i$ denotes the demand of store $i \in [m]$.
The first-stage decision $x_i$ for $i$ stores denotes stock allocation at each location,
and the second-stage decision $y_{ij}(\vec{u})$ denotes stock transportation from store $i$ to store $j$ after demand realization.
Let $c_i,t_{ij}$ denote the unit storage and shipment cost and let $b$ denote the storage capacity.
We have the following formulation:
\begin{equation*}
{
\begin{array}{llclcl}
\displaystyle \text{minimize} & \multicolumn{3}{l}{\sum\limits_{i \in [m]} c_i x_i + \tau
} \\
\text{subject to} & \sum\limits_{i,j \in [m]} t_{ij} y_{ij}(\vec{u}) \leq \tau , && \forall \vec{u} \in U\\
& \sum\limits_{j \in [m]}y_{ji}(\vec{u}) - \sum\limits_{j \in [m]}y_{ij}(\vec{u}) \geq u_i - x_i, && \forall \vec{u} \in U, \quad \forall i \in [m] \\
& y_{ij}(\vec{u}) \geq 0, && \forall \vec{u} \in U, \quad \forall i,j \in [m] \\
& 0 \leq x_{ij} \leq b, && \forall i,j \in [m].
\end{array}}
\end{equation*}
The objective is to minimize the sum of first-stage storage costs and the worst case second-stage transportation cost.
The first constraint represents the worst case second-stage cost, the second constraint requires the demand to be met for all stores, and the rest of the constraints impose nonnegativity and capacity restrictions.

We use a box set for the constraint-wise uncertainty set
\begin{equation*}
    U = \left\{\vec{u} \mid 0 \leq \vec{u} \leq 20\right\}.
\end{equation*}
We model the restriction of total demands among the stores with
\begin{equation*}
\sum\limits_{i\in [m]} u_i \leq 20 \sqrt{n}.
\end{equation*}
Then the coupled set is the budget set
\begin{equation*}
    \cp = \left\{\vec{u} \mid 0 \leq \vec{u} \leq 20, \quad \sum\limits_{i\in [m]} u_i \leq 20 \sqrt{n}\right\}.
\end{equation*}
From Theorem~\ref{thm:mainadapt} and Example~\ref{ex:norm}, we have the theoretical bounds
$$\frac{1}{\sqrt{m}} \leq \frac{\zacp}{\zaro} = \frac{\zacp}{\zcp} \leq 1.$$

\paragraph{Numerical example.}
For each dimension, we generate 50 problem instances with the numerical settings adopted from~\cite{bertsimas2016duality}.
We pick $m$ store locations uniformly at random from $[0,10]^2$.
Let $c_i=20$ for all $i \in [m]$, $b=20$, and let $t_{ij}$ be the Euclidean distance and $t_{ii}= 0$ for all $i \neq j \in [m]$.
We solve the adaptive problem under coupled uncertainty with different solution methods.
The fully adaptive solutions are obtained by enumerating all vertices of the uncertainty set and including an adaptive variable for each vertex.
Benders decomposition method allows 1000 maximum number of iterations, $10^{-3}$ tolerance, starts with the uncoupled robust optimal solution and 100 random uncertainties drawn from $u_i \sim U(0, 20/\sqrt{m})$ for outer approximation.
Finite scenario approach samples uncertainty values using Hit-and-run sampler in coupled uncertainty set and rescales sampled points to the boundary of the set.

\paragraph{Objectives and running times.}
In Figure~\ref{fig:lottime}, we show the objective values relative to the uncoupled problem and running times of different methods on problems  with $m+m^2+1$ robust constraints, $m$ here-and-now variables, and $m^2$ adaptive decision variables.
Since $\Pi_i(U) = \Pi_i\left(\cp\right)$, in the static case the coupled RO and uncoupled RO have the same objective, which is also the same as the adaptive objective under constraint-wise uncertainty.
In the adaptive case, coupling improves the objective, and the improvement increases with problem sizes.
Vertex enumeration gives fully adaptive solutions for small problem sizes.
Benders decomposition gives a close lower bound to the fully adaptive solution in less than 100 seconds.
Linear decision rule and sampling give upper and lower bounds on the objective with a longer time.
The adaptive objectives are relative to the uncoupled objective and the theoretical lower bound both scale with $1/\sqrt{m}$.
Therefore the improvement of coupling on adaptive solutions is significant when $m$ is large.

\begin{figure}[tb]
\begin{subfigure}{0.49\textwidth}
\includegraphics[width=\textwidth]{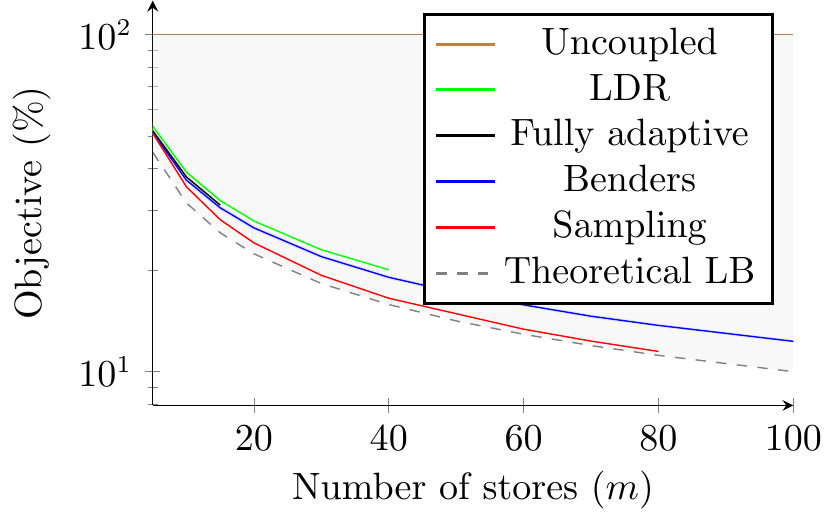}
\end{subfigure}
\begin{subfigure}{0.49\textwidth}
    \includegraphics[width=\textwidth]{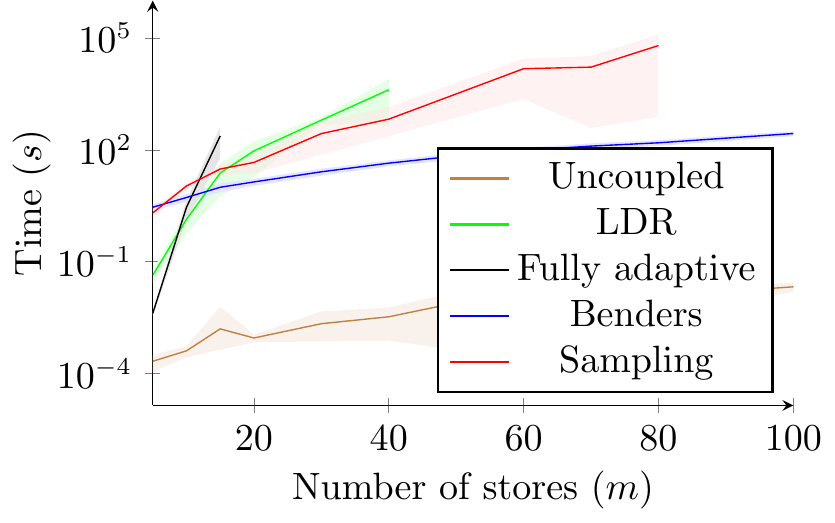}
\end{subfigure}
    \caption{Relative objectives and running times for lot sizing (minimization) problems.}
	\label{fig:lottime}
\end{figure}

\section{Conclusions}
\label{sec:conclusion}
In this paper, we study the benefit of the general uncertainty coupling framework in robust and adaptive robust optimization, where the relationship between uncertain parameters in different constraints is taken into account.
We show that the incorporation of uncertainty coupling can bring benefits from both theoretical and computational approaches.
On the theoretical side, we provide tight and easily computable upper and lower bounds on the improvement of the objective value resulting from coupling in both static and adaptive cases.
We further bound the improvement of adaptive over static solutions under coupled uncertainty.
The theoretical results apply for linear and nonlinear problems in cases when the uncertainties affect the right hand side, affect the coefficient of the constraints, and affect both objectives and constraints.
In particular, we provide closed-form theoretical bounds for special cases of uncertainty sets.
On the computational side, we evaluate the effect of coupling parameters on the improvement of linear RO and ARO problems through numerical examples in supply chain management, portfolio optimization, and lot sizing.
Computational experiments quantify the improvements relative to the theoretical bounds and test different solution methods for the coupled problems.
Computational results demonstrate that coupling can improve solutions for robust and adaptive robust optimization, sometimes significantly, that benefit practical applications.

\bibliography{bibliography.bib} 

\newpage

\ifpreprint
\appendix
\section*{Appendices}
\else
\begin{APPENDICES}
\fi

\section{Down-hull uncertainty sets}
\label{append:down}
\reviewChanges{We show that for robust and adaptive problems~\eqref{eqn:ro},\eqref{eqn:cp},\eqref{eqn:aro},\eqref{eqn:acp},\eqref{eqn:aronl},\eqref{eqn:acpnl}, we can replace $U$ and $\cp$ with their down-hulls without loss of generality.}


\begin{lemma}
\label{lem:down}
Consider robust and adaptive problems~\eqref{eqn:ro},\eqref{eqn:cp},\eqref{eqn:aro},\eqref{eqn:acp},\reviewChanges{\eqref{eqn:aronl},\eqref{eqn:acpnl}}.
The objective values remain the same if the uncertainty sets $U,\cp$ are replaced with their down-hulls $U^\downarrow,\cp^\downarrow$. 
\end{lemma}
\ifpreprint \begin{proof} \else
\proof{Proof.}
\fi
We show the result for\reviewChanges{~\eqref{eqn:acpnl}}.
Let $\zacp,z_{\rm acp}^\downarrow$ denote the objectives of the adaptive problems under $\cp$ and $\cp^\downarrow$.
As $\cp \subseteq \cp^\downarrow$, we have $\zacp \leq z_{\rm acp}^\downarrow$.
To show $\zacp\geq z_{\rm acp}^\downarrow$, let $\left(\xacp,\yacp(\vec{u})\right)$ be an optimal solution to\reviewChanges{~\eqref{eqn:acpnl}.
Then $$\vec{g}_i(\xacp,\yacp(\vec{u}))\geq u_i, \quad \forall \vec{u} \in \cp, \quad \forall i \in [m].$$}
For all $\vec{t}\in\cp^\downarrow$, let $ $
\begin{equation*}
\vec{y}^\downarrow(\vec{t}) =
    \begin{cases}
    \yacp(\vec{t}), & \text{if }\vec{t} \in \cp \\
    \yacp(\vec{s}) \text{ for some $\vec{s} \geq \vec{t}, \vec{s} \in \cp$},  & \text{if }\vec{t} \in \cp^\downarrow \setminus \cp.
    \end{cases}
\end{equation*}
Then for all $\vec{t}\in\cp^\downarrow$, there exists $\vec{s} \in \cp$ s.t.
$$\reviewChanges{\vec{g}_i(\xacp,\vec{y}^\downarrow(\vec{t}))
= \vec{g}_i(\xacp,\yacp(\vec{s}))}
\geq s_i \geq t_i, \quad \forall i \in [m].$$
Then $(\xacp,\vec{y}^\downarrow(\vec{t}))$ is feasible to the problem under $\cp^\downarrow$.
Let $\uacp = \arg\max\limits_{\vec{u}\in\cp} \reviewChanges{f_2( \yacp(\vec{u}))}$, $\vec{t}^\downarrow = \arg\max\limits_{\vec{t}\in\cp^\downarrow} \reviewChanges{f_2(\vec{y}^\downarrow(\vec{t}))}$, and there exists $\vec{s} \in \cp$ s.t.
$$\reviewChanges{\zacp^\downarrow \leq f_1(\xacp) + f_2(\vec{y}^\downarrow(\vec{t^\downarrow}))
= f_1(\xacp) +  f_2(\yacp(\vec{s}))
\leq f_1(\xacp) + f_2(\yacp\left(\uacp\right))} = \zacp.$$
Therefore, $\zacp = z_{\rm acp}^\downarrow$. \reviewChanges{The result for~\eqref{eqn:aronl} can be proved similarly and results for static problems follow as a special case with $\yacp(\vec{u}) = \yacp,\forall \vec{u} \in \cp$. The results for linear functions~\eqref{eqn:ro},\eqref{eqn:cp},\eqref{eqn:aro},\eqref{eqn:acp} also follow.}
\ifpreprint \end{proof} \else
\Halmos
\endproof
\fi

\section{Proof of results in Section~\ref{sec:theorycoeff}}
\label{append:proofcoeff}
\ifpreprint \begin{proof}[Proof of Theorem~\ref{thm:main2}] \else
\proof{Proof of Theorem~\ref{thm:main2}.}
\fi
\reviewChanges{Since $0 \in \cp$ and $U$ is convex, $\ri$ and $\gi$ exist.}
To show the upper bound,
Let $\xcp$ be an optimal solution to~\eqref{eqn:cp2}, since $V \subseteq \Pi\left(\cp\right)$, then it is also feasible to
\begin{equation*}
{
\begin{array}{llclcl}
\displaystyle  &\text{maximize} && {\vec{c}^T\vec{x}} \\
&\text{subject to} &&
\vec{u}_i^T\vec{x}\leq b_i, && \forall \vec{u} \in V, \quad \forall i \in [m].
\end{array}}
\end{equation*}
Then $\ri \xcp$ is feasible to~\eqref{eqn:ro2} since
$$\vec{u}_i^T\ri \xcp = (\ri\vec{u}_i)^T \xcp \leq  b_i, \quad \forall \vec{u} \in U, \quad \forall i \in [m].$$
Thus
\begin{equation*}
 \zro \geq \ri \zcp.
\end{equation*}
To show the lower bound,
let $\xro$ be an optimal solution to~\eqref{eqn:ro2}, then $\frac{1}{\gi} \xro$ is also a feasible solution to
\begin{equation*}
{
\begin{array}{llclcl}
\displaystyle  &\text{maximize} &&{\vec{c}^T\vec{x}} \\
&\text{subject to} & &
 \vec{u}_i^T\vec{x}\leq b_i, && \forall \vec{u} \in W, \quad \forall i \in [m], 
\end{array}}
\end{equation*}
since
$$\vec{u}_i^T \frac{1}{\gi} \xro
= \left(\frac{1}{\gi} \vec{u}_i\right)^T \xro
\leq  \frac{1}{\gi} b_i, \quad \forall \vec{u} \in U, \quad \forall i \in [m].$$
As $\Pi\left(\cp\right) \subseteq W$, thus
\begin{equation*}
\zcp \geq \zsi \geq \frac{1}{\gi} \zro.
\end{equation*}

To show the tightness of the bounds, consider the problem
\begin{equation*}
{
\begin{array}{llclcl}
\displaystyle \max  & c_1 x_1+c_2 x_2+c_3 x_3+c_4 x_4 &\vspace{3pt} \\
{\rm s.t.}
& u_1 x_1 + u_2 x_2 \leq 1 , && \forall (u_1,u_2) \in U_1 \\
& u_3 x_3 + u_4 x_4 \leq 1 , && \forall (u_3,u_4) \in U_2,
\end{array}}
\end{equation*}
under uncertainty sets from Example~\ref{ex:l1l2} with $\alpha = \beta=1$ and $p = 2$.
According to Theorem \ref{thm:main}, since $\ri = 1/\sqrt{2}$ and $\gi = 1$, we have
$1 \leq \zcp/\zro \leq \sqrt{2}$.
If $c_1=c_2=0,c_3=c_4=1$, then $\zro=\sqrt{2},\zcp=2$, $\zcp/\zro= \sqrt{2} = 1/\ri$.
If $c_1=c_2=1,c_3=c_4=0$, then $\zro=2,\zcp=2$, $\zcp/\zro =1 = 1/\gi$.
Therefore both bounds are tight.
\ifpreprint \end{proof} \else
\Halmos
\endproof
\fi

\ifpreprint \begin{proof}[Proof of Theorem~\ref{thm:mainadapt2}] \else
\proof{Proof of Theorem~\ref{thm:mainadapt2}.}
\fi
\reviewChanges{Since $\{0\} \in \cp$ and $U$ is convex, $\ri$ and $\gi$ exist.}
To show the upper bound,
let $\left(\xacp,\yacp(\vec{u})\right)$ be an optimal solution to~\eqref{eqn:acp2}.
Since $V \subseteq \cp$, then it is also feasible to
\begin{equation*}
{
\begin{array}{llclcl}
\displaystyle
& \text{maximize} && {\vec{c}^T\vec{x} + \min\limits_{\vec{u}\in V} \vec{d}^T \vec{y}(\vec{u})} \\
& \text{subject to} &&
\vec{u}_i^T
\begin{bmatrix*}[l] \vec{x}\\ \vec{y}(\vec{u}) \end{bmatrix*}
\leq b_i, && \forall \vec{u} \in V, \quad \forall i \in [m].
\end{array}}
\end{equation*}
\reviewChanges{Then $\ari{\xacp},\ari{\yacp\left(\ari \vec{u}\right)}$ is feasible to~\eqref{eqn:aro2} since
$$\vec{u}_i^T
\begin{bmatrix*}[l] \ari{\xacp}\\ \ari\yacp\left(\ari \vec{u}\right)\end{bmatrix*}
\leq  b_i, \quad \forall \vec{u} \in U, \quad \forall i \in [m]. $$
Thus
\begin{equation*}
\zaro \geq \ari{\vec{c}^T\xacp + \ari\min\limits_{\vec{u}\in U} \vec{d}^T \yacp(\ari \vec{u})} \geq  \ari{\vec{c}^T\xacp + \ari\min\limits_{\vec{u}\in \cp} \vec{d}^T \yacp(\vec{u})} = \ari \zacp.
\end{equation*}}

To show the lower bound, suppose $(\xaro,\yaro(\vec{u}))$ is an optimal solution to~\eqref{eqn:aro2}.
\reviewChanges{Then $\frac{1}{\agi}(\xaro,\yaro(\vec{u}/\agi))$} is a feasible solution to
\begin{equation*}
{
\begin{array}{llclcl}
\displaystyle \zsi =
& \text{maximize} && {\vec{c}^T\vec{x}+ \min\limits_{\vec{u}\in W} \vec{d}^T \vec{y}(\vec{u})} \\
& \text{subject to} &&
\vec{u}_i^T
\begin{bmatrix*}[l] \vec{x}\\ \vec{y}(\vec{u}) \end{bmatrix*}
\leq  b_i, && \forall \vec{u} \in W, \quad \forall i \in [m],
\end{array}}
\end{equation*}
since
\reviewChanges{$$\vec{u}_i^T
\begin{bmatrix*}[l] {\xaro}\\ \yaro(\vec{u}/\agi)\end{bmatrix*}
\leq \agi b_i, \quad \forall \vec{u} \in W, \quad \forall i \in [m]. $$
As $\cp \subseteq W$, this solution is feasible also for~\eqref{eqn:acp2}, thus
\begin{equation*}
\zacp \geq  \frac{1}{\agi}{\vec{c}^T\xaro + \frac{1}{\agi}\min\limits_{\vec{u}\in \cp} \vec{d}^T\yaro (\vec{u}/\agi)} \geq \zsi \geq \frac{1}{\agi} \zaro.
\end{equation*}}
To show the tightness of the bounds, consider the problem
\begin{equation*}
{
\begin{array}{llclcl}
\displaystyle \max  & \min\limits_{\vec{u}\in U} c_1 y_1(\vec{u})+c_2 y_2(\vec{u}) &\vspace{3pt} \\
{\rm s.t.}
& u_1 y_1(\vec{u}) \leq 1 , && \forall \vec{u} \in U \\
& u_2 y_2(\vec{u}) \leq 1 , && \forall \vec{u} \in U,
\end{array}}
\end{equation*}
where
\begin{equation*}
    U = \left\{\vec{u} \mid 0 \leq \vec{u} \leq 1\right\},
\quad
    \overline{U} = \left\{\vec{u} \mid 0 \leq \vec{u} \leq 1, \quad 0 \leq u_1+u_2\leq 1\right\}.
\end{equation*}
According to Theorem \ref{thm:mainadapt}, since $\ari = 1/2$ and $\agi = 1$, we have
$1 \leq \zacp/\zaro \leq 2$.
The optimal solution for the adaptive variable is $y_1(\vec{u})=1/u_1, y_2(\vec{u})=1/u_2.$
If $c_1 = c_2 = 1$, then $\zaro= 2,\zacp=4$, $\zacp/\zaro= 2 = 1/\ari$.
If $c_1 = 1, c_2 = 0$, then $\zaro=\zacp=1$, $\zacp/\zaro =1 = 1/\agi$.
Therefore both bounds are tight.
\ifpreprint \end{proof} \else
\Halmos
\endproof
\fi

\ifpreprint \begin{proof}[Proof of Theorem~\ref{cor:mainadapt}]\else
\proof{Proof of Theorem~\ref{cor:mainadapt}.}
\fi
\reviewChanges{Let $\left(\xacp,\yacp(\vec{u})\right)$ be an optimal solution to~\eqref{eqn:acp2}. We construct a solution $\arho \left(\xacp,\yacp(\arho\vec{u})\right)$ feasible to the adaptive problem under $\Pi\left(\cp\right)$, and let the objective of this problem be $\hat{z}_{\text{acp}}$. 
Since $\Pi\left(\cp\right)$ is a constraint-wise and compact uncertainty set, and the adaptive problem under $\Pi\left(\cp\right)$ satisfies Assumption~\ref{ass:compact}, we have 
$\hat{z}_{\text{acp}} = \zcp$~\citep[Theorem 2.1]{ben2004adjustable}. 
Therefore, 
$$ \arho\zacp \leq \hat{z}_{\text{acp}}= \zcp.$$
Next, the bound on $\arho$ follows from Lemma~\ref{lem:aro_arho}. 
To show the tightness of the bounds, let $\zro, \zaro$ denote the objectives of the static and adaptive robust problems under $U$, and consider the case where $\Pi\left(\cp\right) = U$. By Theorem~\ref{thm:main2} and the above, we then have $\hat{z}_{\text{acp}} = \zaro = \zro = \zcp$, $\arho = \ari$, $\gi = 1$. Thus,
\begin{equation*}
    \frac{\zacp}{\zcp} = \frac{\zacp}{\zro} = \frac{\zacp}{\zaro} \leq \frac{1}{\ari} = \frac{1}{\arho} .
\end{equation*}
The tightness of the bound then follows from Theorem~\ref{thm:mainadapt2}.}
\ifpreprint \end{proof} \else
\Halmos
\endproof
\fi

\ifpreprint \begin{proof}[Proof of Lemma~\ref{lem:mp}] \else
\proof{Proof of Lemma~\ref{lem:mp}.}
\fi
Since \reviewChanges{$\{0\} \subset \cp\subset \reals_+^{mp}$ and $\cp$ is convex,}
let $\Pi\left(\cp\right) = \Pi_1\left(\cp\right) \times \dots \times \dots \Pi_m\left(\cp\right) \subseteq [0,\overline{d}_1] \times \dots \times [0,\overline{d}_{mp}]$ and
$(\overline{d}_1,0,\dots,0), \,\dots, (0,\dots,0,\overline{d}_{mp}) \in \cp$.
By the convexity of $\cp$, $1/mp \, (\overline{d}_1,\dots, \overline{d}_{mp}) = 1/mp \, ((\overline{d}_1,0,\dots,0)+ \dots+ (0,\dots,0,\overline{d}_{mp})) \in \cp$.
Then $1/mp \, \Pi\left(\cp\right) \subseteq 1/mp \, ([0,\overline{d}_1] \times \dots \times [0,\overline{d}_{mp}]) \subseteq \cp$.
Therefore $\arho \geq 1/mp$ by definition.
\ifpreprint \end{proof} \else
\Halmos
\endproof
\fi

\section{Proof of results in Section~\ref{sec:theorynl}}
\label{append:proofnl}

\ifpreprint \begin{proof}[Proof of Theorem~\ref{thm:mainnl}] \else
\proof{Proof of Theorem~\ref{thm:mainnl}.}
\fi
We prove the adaptive case and the static case result follows similarly as a special case. 
\reviewChanges{By Appendix~\ref{append:down}, without loss of generality, we can replace $U$, $\cp$ with their down-hulls in the following proof. }
Suppose $\left(\xacp,\yacp(\vec{u})\right)$ is an optimal solution to~\eqref{eqn:acpnl}.
Then it is also feasible to the problem under $V$, thus
$$g_i(\xacp,\yacp\left(\reviewChanges{\ari\vec{u}}\right)) \geq \ari {u}_i, \quad \forall \vec{u} \in U,\quad \forall i \in [m].$$ 
As $0\leq\ari\leq 1$, then for all $\vec{u}\in U$ and $i \in [m]$:
\reviewChanges{\begin{align*}
\ari g_i\left(\frac{1}{\ari} \xacp,\frac{1}{\ari} \yacp\left(\ari\vec{u}\right)\right)
& = \ari g_i\left(\frac{1}{\ari} \xacp,\frac{1}{\ari} \yacp\left(\ari\vec{u}\right)\right) + (1-\ari) g_i(0,0) \\
& \geq g_i\left(\xacp,\yacp\left(\ari\vec{u}\right)\right)\\
& \geq \ari {u}_i,
\end{align*}}
where the equality follows because $g_i(0,0)= 0$; and
the first inequality follows because $g_i$ is convex in $\vec{x}$ and $\vec{y}(\vec{u})$.
Then $\left(\frac{1}{\ari} \xacp,\frac{1}{\ari} \yacp\left(\reviewChanges{\ari\vec{u}}\right)\right) $ is feasible to~\eqref{eqn:aronl}.
Then
\reviewChanges{
\begin{align*}
\zacp  &= f_1(\xacp) + \max\limits_{\vec{u} \in \cp}f_2(\yacp\left(\vec{u}\right))\\
&\geq \ari f_1(\xacp/\ari) + \max\limits_{\vec{u} \in \cp}\ari f_2(\yacp\left(\vec{u}\right)) + (1-\ari)(f_1(0)+f_2(0)))\\
&\geq \ari f_1(\xacp/\ari) + \max\limits_{\vec{u} \in U}\ari f_2(\yacp\left(\ari\vec{u}\right))\\
&\geq \ari \zaro
\end{align*}
}
\reviewChanges{
The first inequality follows because $f_1,f_2$ are concave and $0 < \ari \leq 1$; the second inequality follows because $f_1(0) = 0,f_2(0) = 0$, and $\ari U \subseteq \cp$.}
\ifpreprint \end{proof} \else
\Halmos
\endproof
\fi

\ifpreprint \begin{proof}[Proof of Theorem~\ref{cor:mainadaptnl}] \else
\proof{Proof of Theorem~\ref{cor:mainadaptnl}.}
\fi
\reviewChanges{Let $\left(\xacp,\yacp(\vec{u})\right)$ be an optimal solution to~\eqref{eqn:acpnl}. We construct a solution $\left(\xacp,\yacp(\arho\vec{u})\right)/\arho$ feasible to the adaptive problem under $\Pi\left(\cp\right)$, and let the objective of this problem be $\hat{z}_{\text{acp}}$. 
Since $\Pi\left(\cp\right)$ is constraint-wise and compact, and the adaptive problem under $\Pi\left(\cp\right)$ has fixed recourse, we have 
$\hat{z}_{\text{acp}} = \zcp$~\citep[Remark 1]{marandi2018static}.
Therefore, 
$$ \zacp/\arho \geq \hat{z}_{\text{acp}}= \zcp.$$
Next, the bound on $\arho$ follows from Lemma~\ref{lem:aro_arho}. 
To show the tightness of the bounds, let $\zro, \zaro$ denote the objectives of the static and adaptive robust problems under $U$, and consider the case where $\Pi\left(\cp\right) = U$. Then, by Theorem~\ref{thm:mainnl} and the above, we have $\hat{z}_{\text{acp}}= \zaro=\zro=\zcp$, $\ari=\arho$, and thus
\begin{equation*}
    \frac{\zacp}{\zcp} = \frac{\zacp}{\zro} = \frac{\zacp}{\zaro} \geq {\ari} = {\arho} .
\end{equation*}
The tightness of the bound then follows from Theorem~\ref{thm:mainnl}.}
\ifpreprint \end{proof} \else
\Halmos
\endproof
\fi

\ifpreprint \begin{proof}[Proof of Theorem~\ref{thm:mainnlconv}] \else
\proof{Proof of Theorem~\ref{thm:mainnlconv}.}
\fi
We prove the adaptive case and the static case result follows similarly as a special case.
Suppose \reviewChanges{$(\xaro,\yaro(\vec{u}))$} is an optimal solution to~\eqref{eqn:aronl}.
Thus \reviewChanges{$$g_i(\xaro,\yaro(\vec{u}/\agi)) \geq u_i/\agi, \quad \forall \vec{u}\in \agi U,\quad \forall i \in [m].$$}
As $0 < \agi \leq 1$, then for all $\vec{u}\in U$ and $i \in [m]$:
\begin{align*}
g_i(\agi \xaro,\agi \yaro(\reviewChanges{\vec{u}/\agi}))
& \geq \agi g_i(\xaro,\yaro(\reviewChanges{\vec{u}/\agi})) + (1-\agi) g_i(0,0) \\
& \geq \agi g_i(\xaro,\yaro(\reviewChanges{\vec{u}/\agi}))\\
& \geq u_i,
\end{align*}
where the first inequality follows because $g_i$ is concave in $\vec{x}$ and $\vec{y}$; and
the second inequality follows because $g_i(0,0)\geq 0,\forall \vec{u}\in U$;

Then $(\agi\xaro,\agi\yaro(\reviewChanges{\vec{u}/\agi}))$ is feasible to the problem under $W$, thus also feasible to~\eqref{eqn:acpnl} as $\cp \subseteq W$.
Then
\reviewChanges{
\begin{align*}
\zacp &\leq f_1(\agi \xaro) + \max\limits_{\vec{u}\in \cp} f_2(\agi\yacp(\vec{u}/\agi))\\
&\leq \agi f_1( \xaro) + \max\limits_{\vec{u} \in \cp }f_2( \yacp(\vec{u}/\agi)) + (1-\agi)(f_1(0)+f_2(0))\\
&\leq \agi f_1( \xaro) + \max\limits_{\vec{u} \in \agi U }f_2( \yacp(\vec{u}/\agi))\\
&= \agi f_1( \xaro) + \max\limits_{\vec{u} \in U }f_2( \yacp(\vec{u}))\\
&= \agi \zaro.
\end{align*}
}
\reviewChanges{The second inequality follows because $f_1,f_2$ are convex and $0 < \agi \leq 1$;
the third inequality follows because $f_1(0) = 0,f_2(0) = 0$ and $\agi U \subseteq \cp$.}
\ifpreprint \end{proof} \else
\Halmos
\endproof
\fi

\ifpreprint \begin{proof}[Proof of Theorem~\ref{thm:mainnl2}.] \else
\proof{Proof of Theorem~\ref{thm:mainnl2}.}
\fi
We prove the adaptive case and the static case result follows similarly as a special case.
Suppose $\left(\xacp,\yacp(\vec{u})\right)$ is an optimal solution to~\eqref{eqn:acpnl2}.
Then it is also feasible to the problem under $V$, thus
$$g_i\left(\xacp,\yacp\left(\reviewChanges{\ari \vec{u}}\right), \ari \vec{u}_i\right) \leq b_i, \quad \forall \vec{u} \in U,\quad \forall i \in [m].$$
As $0 < \ari \leq 1$, then for all $\vec{u}\in U$ and $i \in [m]$:
\begin{align*}
g_i\left(\ari \xacp,\ari \yacp\left(\reviewChanges{\ari \vec{u}}\right), \vec{u}_i\right)
& \leq \ari g_i\left(\xacp,\yacp\left(\reviewChanges{\ari \vec{u}}\right), \vec{u}_i\right) + (1-\ari) g_i(0,0,\vec{u}_i) \\
& \leq \ari g_i\left(\xacp,\yacp\left(\reviewChanges{\ari \vec{u}}\right), \vec{u}_i\right)\\
& \leq (1-\ari)g_i\left(\xacp,\yacp\left(\reviewChanges{\ari \vec{u}}\right),0\right)+ \ari g_i\left(\xacp,\yacp\left(\reviewChanges{\ari \vec{u}}\right),  \vec{u}_i\right)\\
& \leq g_i\left((1-\ari)(\xacp,\yacp\left(\reviewChanges{\ari \vec{u}}\right),0\right)+ \ari \left(\xacp,\yacp\left(\reviewChanges{\ari \vec{u}}\right), \vec{u}_i)\right)\\
& = g_i\left(\xacp,\yacp\left(\reviewChanges{\ari \vec{u}}\right),\ari \vec{u}_i\right)\\
& \leq b_i,
\end{align*}
where the first inequality follows because $g_i$ is convex in $\vec{x}$ and $\vec{y}(\vec{u})$;
the second inequality follows because $g_i(0,0,\vec{u}_i)\leq 0,\forall \vec{u}\in U$;
the third because $g_i(\vec{x},\vec{y(\vec{u})},0)\geq 0, \forall \vec{x},\vec{y}(\vec{u})$ feasible;
and the fourth because $g_i$ is concave in $\vec{u}_i$.

Then $\left(\ari\xacp,\ari\yacp\left(\reviewChanges{\ari \vec{u}}\right)\right)$ is feasible to~\eqref{eqn:aronl2}.
Then
\reviewChanges{
\begin{align*}
\zaro & \geq  f_1\left(\ari \xacp\right) + \min\limits_{\vec{u} \in U } f_2\left(\ari \yacp\left(\ari \vec{u}\right)\right)   \\
&\geq \ari \left(f_1\left(\xacp\right) + \min\limits_{\vec{u} \in U }f_2\left(\yacp\left(\ari \vec{u}\right)\right)\right) + (1-\ari)(f_1(0)+f_2(0)) \\
&\geq \ari \left(f_1\left(\xacp\right) + \min\limits_{\vec{u} \in \cp }f_2\left(\yacp\left(\vec{u}\right)\right)\right)\\
&= \ari \zacp,
\end{align*}}
where the second inequality follows because $f_1,f_2$ are concave and $0 < \ari \leq 1$;
\reviewChanges{and the third inequality follows because $f_1(0) = 0,f_2(0) = 0$, and $\ari U \subseteq \cp$.}
\ifpreprint \end{proof} \else
\Halmos
\endproof
\fi

\ifpreprint \begin{proof}[Proof of Theorem~\ref{cor:mainadaptnl2}] \else
\proof{Proof of Theorem~\ref{cor:mainadaptnl2}.}
\fi
Consider the constraint-wise adaptive problem
\begin{equation*}
{
\begin{array}{llclcl}
\displaystyle  \reviewChanges{\hat{z}_{\text{acp}}} = & \text{maximize} && {f_1(\vec{x}) + \min\limits_{\vec{u} \in \Pi\left(\cp\right)} f_2(\vec{y}(\vec{u}))} \\
& \text{subject to} &&
 g_i(\vec{x}, \vec{y}(\vec{u}),\vec{u}_i)\leq b_i, \quad \forall \vec{u} \in \Pi\left(\cp\right), \quad \forall i \in [m],
\end{array}}
\end{equation*}

which can be written as
\begin{equation}
\label{eq:adapt_zcp}
{
\begin{array}{llclcl}
\displaystyle \reviewChanges{-\hat{z}_{\text{acp}}} = & \text{minimize} && {-f_1(\vec{x}) + \max\limits_{\vec{u} \in \Pi\left(\cp\right)} -f_2(\vec{y}(\vec{u}))} \\
& \text{subject to} &&
 g_i(\vec{x}, \vec{y}(\vec{u}),\vec{u}_i) - b_i \leq 0, \quad \forall \vec{u} \in \Pi\left(\cp\right), \quad \forall i \in [m].
\end{array}}
\end{equation}
Note that the corresponding static problem is \reviewChanges{$\zcp$}.
\reviewChanges{Recall that $\Pi\left(\cp\right)$ is constraint-wise, compact, and convex, and
Assumption~\ref{ass:compact} holds. In addition,
$-f_2$ is convex, and 
for each $i \in [m]$, $g_i-b_i$ is convex in $\vec{y}(\vec{u})$ and concave in $\vec{u}$.
Then, $\hat{z}_{\text{acp}}=\zcp$~\citep[Theorem 2]{marandi2018static}.
Since $\arho \Pi\left(\cp\right) \subseteq \cp$, we can apply the proof technique from Theorem~\ref{cor:mainadapt} to establish
$$\arho\zacp \leq \hat{z}_{\text{acp}} = \zcp. $$
Next, $\arho \geq \ari$ follows from Lemma~\ref{lem:aro_arho}.
}
\reviewChanges{
To show the tightness of the bounds, let $\zro, \zaro$ denote the objectives of the static and adaptive robust problems under $U$, and consider the case where $\Pi\left(\cp\right) = U$. Then we have $\hat{z}_{\text{acp}}= \zaro = \zro = \zcp$, and $\arho = \ari$.
Then, \begin{equation*}
    \frac{\zacp}{\zcp} = \frac{\zacp}{\zro} = \frac{\zacp}{\zaro} \leq \frac{1}{\ari} = \frac{1}{\arho}.
\end{equation*}
}
The tightness of the bound then follows from Theorem~\ref{thm:mainnl2}.
\ifpreprint \end{proof} \else
\Halmos
\endproof
\fi

\ifpreprint \begin{proof}[Proof of Theorem~\ref{thm:mainnlconv2}] \else
\proof{Proof of Theorem~\ref{thm:mainnlconv2}.}
\fi
We prove the adaptive case and the static case result follows similarly as a special case.
Suppose $(\xaro,\yaro(\reviewChanges{\vec{u}}))$ is an optimal solution to~\eqref{eqn:aronl2}.
Thus
$$g_i(\xaro,\yaro(\reviewChanges{\vec{u}}),  \vec{u}_i) \leq b_i, \quad \forall \vec{u} \in U,\quad \forall i \in [m].$$
As $0 < \agi \leq 1$, then for all $\vec{u}\in U$ and $i \in [m]$:
\reviewChanges{
\begin{align*}
\agi g_i\left(\frac{\xaro}{\agi}, \frac{\yaro(\vec{u})}{\agi}, \agi\vec{u}_i\right)
& \leq \agi g_i\left(\frac{\xaro}{\agi}, \frac{\yaro(\vec{u})}{\agi}, \agi\vec{u}_i\right) + (1-\agi)g_i(0, 0, \agi\vec{u}_i)\\
& \leq  g_i(\xaro,\yaro(\vec{u}), \agi\vec{u}_i)\\
&  \leq \agi g_i(\xaro,\yaro(\vec{u}), \vec{u}_i) + (1-\agi)g_i(\xaro,\yaro(\vec{u}),0)\\
& \leq \agi g_i(\xaro,\yaro(\vec{u}), \vec{u}_i)\\
& \leq \agi b_i,
\end{align*}}
where the first inequality follows because $g_i(0,0,\vec{u}_i)\geq 0,\forall \vec{u}\in U$;
the second inequality follows because $g_i$ is concave in $\vec{x}$ and $\vec{y}$ \reviewChanges{and $0 < \agi \leq 1$};
the third because $g_i$ is convex in $\vec{u}_i$;
and the fourth because $g_i(\vec{x},\vec{y}(\vec{u}),0)\leq 0, \forall \vec{x},\vec{y}(\vec{u})$ feasible.

Then $\left(\frac{\xaro}{\agi}, \frac{\yaro(\reviewChanges{\vec{u}/\agi})}{\agi}\right)$ is feasible to the problem under $W$, thus also feasible to~\eqref{eqn:acpnl2} as $\cp \subseteq W$.
Then
\reviewChanges{
\begin{align*}
\agi \zacp & \geq  \agi f_1\left(\frac{\xaro}{\agi}\right) + \agi \min\limits_{\vec{u} \in \cp} f_2\left(\frac{\yaro(\vec{u}/\agi)}{\agi}\right)   \\
&\geq  \agi f_1\left(\frac{\xaro}{\agi}\right) + \agi \min\limits_{\vec{u} \in \agi U} f_2\left(\frac{\yaro(\vec{u}/\agi)}{\agi}\right)   \\
&= \agi f_1\left(\frac{\xaro}{\agi}\right) + \agi \min\limits_{\vec{u} \in U} f_2\left(\frac{\yaro(\vec{u})}{\agi}\right) + (1-\agi)(f_1(0)+f_2(0))  \\
&\geq f_1(\xaro) + \min\limits_{\vec{u} \in U}f_2(\yaro(\vec{u}))\\
&= \zaro,
\end{align*}
where the second inequality follows because $\cp \subseteq \agi U$; the first equality follows because $f_1(0) = 0,f_2(0) = 0$; and the
third inequality follows because $f_1,f_2$ are convex and $0 < \agi \leq 1$.
}
\ifpreprint \end{proof} \else
\Halmos
\endproof
\fi

\ifpreprint \begin{proof}[Proof of Theorem~\ref{thm:mainnl3}] \else
\proof{Proof of Theorem~\ref{thm:mainnl3}.}
\fi
We prove the adaptive case and the static case result follows similarly as a special case.
Suppose $\left(\xacp,\yacp(\vec{u})\right)$ is an optimal solution to~\eqref{eqn:acpnl3}.
We have proved in Theorem~\ref{thm:mainnl2} that $(\ari\xacp,\ari\yacp\left(\reviewChanges{\ari\vec{u}}\right))$ is feasible to~\eqref{eqn:aronl3}.
Let $\uaro = \arg  \min\limits_{\vec{u} \in U} f_1(\ari\xacp,\vec{u}) + f_2(\ari\yacp\left(\reviewChanges{\ari \vec{u}}\right),\vec{u})$.
Then
\begin{align*}
\zaro & \geq  f_1(\ari \xacp,\uaro) + f_2(\ari \yacp\left(\reviewChanges{\ari\uaro}\right),\uaro)   \\
&\geq \ari (f_1(\xacp,\uaro) + f_2(\yacp\left(\reviewChanges{\ari\uaro}\right),\uaro)) + (1-\ari)(f_1(0,\uaro)+f_2(0,\uaro)) \\
&= \ari (f_1(\xacp,\uaro) + f_2(\yacp\left(\reviewChanges{\ari\uaro}\right),\uaro))\\
&\geq \ari^2 (f_1(\xacp,\ari\uaro) + f_2(\yacp\left(\reviewChanges{\ari\uaro}\right),\ari\uaro))\\
&\geq \ari^2 \min\limits_{\vec{u} \in \cp} (f_1(\xacp,\vec{u}) + f_2(\yacp\left(\reviewChanges{\vec{u}}\right),\vec{u}))\\
&= \ari^2 \zacp,
\end{align*}
where the second inequality follows because $f_1,f_2$ are concave in $\vec{x},\vec{y}(\vec{u})$ and $0 < \ari \leq 1$;
the first equality follows because $f_1(0,\vec{u}) = 0,f_2(0,\vec{u}) = 0,\forall \vec{u}\in U$;
 the third inequality follows because for all $\vec{u}\in U$, $\vec{x},\vec{y}(\vec{u})$ feasible and $0 \leq \alpha \leq 1$, $f_1(\vec{x},\vec{u}) +  f_2(\vec{y}(\vec{u}),\vec{u}) \geq \alpha (f_1(\vec{x},\alpha\vec{u}) +  f_2(\vec{y}(\vec{u}),\alpha\vec{u}))$;
 and the fourth inequality follows because $\ari \uaro \in V \subseteq \cp$.
\ifpreprint \end{proof} \else
\Halmos
\endproof
\fi

\ifpreprint \begin{proof}[Proof of Theorem~\ref{cor:mainadaptnl3}] \else
\proof{Proof of Theorem~\ref{cor:mainadaptnl3}.}
\fi
Consider the constraint-wise adaptive problem
\begin{equation*}
{
\begin{array}{llclcl}
\displaystyle  \hat{z}_{\text{acp}} = & \text{maximize} && \min\limits_{\vec{u} \in \Pi\left(\cp\right)} ({f_1(\vec{x},\vec{u}) +  f_2(\vec{y}(\vec{u}),\vec{u})}) \\
& \text{subject to} &&
 g_i(\vec{x}, \vec{y}(\vec{u}),\vec{u}_i)\leq b_i, \quad \forall \vec{u} \in \Pi\left(\cp\right), \quad \forall i \in [m],
\end{array}}
\end{equation*}

which can be expressed as
\begin{equation*}
{
\begin{array}{llclcl}
\displaystyle  \reviewChanges{-\hat{z}_{\text{acp}} }= & \text{minimize} && t
\\
& \text{subject to} &&
{-f_1(\vec{x},\vec{u}) +  -f_2(\vec{y}(\vec{u}),\vec{u}) \leq t}, \quad \reviewChanges{\forall \vec{u} \in \Pi\left(\cp\right)}, \\
& && g_i(\vec{x}, \vec{y}(\vec{u}),\vec{u}_i) - b_i \leq 0, \quad \forall \vec{u} \in \Pi\left(\cp\right), \quad \forall i \in [m],
\end{array}}
\end{equation*}
and 
\reviewChanges{note that the corresponding static problem has objective $\zcp$.
Recall that $\Pi\left(\cp\right)$ is constraint-wise, compact, and convex, and
Assumption~\ref{ass:compact} holds. In addition, 
$-f_1(\vec{x},\vec{u}) -f_2(\vec{y}(\vec{u}),\vec{u})$ is concave in $\vec{u}$ and convex in $\vec{y}(\vec{u})$,
for each $i \in [m]$, and $g_i-b_i$ is convex in $\vec{y}(\vec{u})$ and concave in $\vec{u}$.
Then, $\hat{z}_{\text{acp}}=\zcp$~\citep[Theorem 2]{marandi2018static}.}
\reviewChanges{Since $\arho \Pi\left(\cp\right) \subseteq \cp$, we can apply the proof technique from Theorem~\ref{thm:mainnl3} to establish
$$\arho^2\zacp \leq \hat{z}_{\text{acp}} = \zcp. $$
Next, $\arho \geq \ari$ follows from Lemma~\ref{lem:aro_arho}.}
\ifpreprint \end{proof} \else
\Halmos
\endproof
\fi

\ifpreprint \begin{proof}[Proof of Theorem~\ref{thm:mainnlconv3}] \else
\proof{Proof of Theorem~\ref{thm:mainnlconv3}.}
\fi
We prove the adaptive case and the static case result follows similarly as a special case.
Suppose $(\xaro,\yaro(\vec{u}))$ is an optimal solution to~\eqref{eqn:aronl3} and $\uaro = \arg  \min\limits_{\vec{u} \in U} f_1(\xro,\vec{u})+ f_2(\yaro(\reviewChanges{\vec{u}/\agi})),\vec{u})$.
We have proved in Theorem~\ref{thm:mainnlconv2} that $\left(\frac{\xaro}{\agi}, \frac{\yaro(\reviewChanges{\vec{u}/\agi)}}{\agi}\right)$ is feasible to~\eqref{eqn:acpnl3}.
Let $\uacp = \arg  \min\limits_{\vec{u} \in \cp} f_1\left(\frac{\xaro}{\agi},\vec{u}\right) + f_2\left(\frac{\yacp\left(\reviewChanges{\vec{u}/\agi)}\right)}{\agi},\vec{u}\right)$.
Then
\reviewChanges{
\begin{align*}
\agi \zacp & \geq  \agi f_1\left(\frac{\xaro}{\agi},\uacp\right) + \agi f_2\left(\frac{\yaro(\reviewChanges{\uacp/\agi})}{\agi},\uacp\right)    \\
&= \agi \left(f_1\left(\frac{\xaro}{\agi},\uacp\right) + f_2\left(\frac{\yaro(\reviewChanges{\uacp/\agi})}{\agi},\uacp\right) \right)   + (1-\agi)(f_1(0,\uacp)+f_2(0,\uacp)) \\
&\geq f_1(\xaro,\uacp) + f_2(\yaro(\reviewChanges{\uacp/\agi}),\uacp)  \\
&\geq (f_1(\xaro,\uacp/\agi) + f_2(\yaro(\reviewChanges{\uacp/\agi}),\uacp/\agi))/\agi \\
&\geq \min\limits_{\vec{u} \in U} (f_1(\xacp,\vec{u}) + f_2(\yacp\left(\vec{u}\right),\vec{u}))/\agi\\
&= \zaro/\agi,
\end{align*}
where the first equality follows because $f_1(0,\vec{u}) = 0,f_2(0,\vec{u}) = 0,\forall \vec{u}\in U$;  the second inequality follows because $f_1,f_2$ are convex in $\vec{x},\vec{y}(\vec{u})$ and $0 < \agi \leq 1$; }
 the third inequality follows because for all $\vec{u}\in U$, $\vec{x},\vec{y}(\vec{u})$ feasible, $f_1(\vec{x},\vec{u}) +  f_2(\vec{y}(\vec{u}),\vec{u}) \leq \agi (f_1(\vec{x},\agi\vec{u}) +  f_2(\vec{y}(\vec{u}),\agi\vec{u}))
$;
 and the fourth inequality follows because $\uacp/\agi \in  U$.
\ifpreprint \end{proof} \else
\Halmos
\endproof
\fi

\ifpreprint \begin{proof}[Proof of Lemma~\ref{lem:tightadapt2}] \else
\proof{Proof of Lemma~\ref{lem:tightadapt2}.}
\fi
Consider the problem
\begin{equation*}
{
\begin{array}{llclcl}
\displaystyle \max  & \min\limits_{\vec{u}\in U} c_1 x_1 / u_1 +c_2 x_2 / u_2 &\vspace{3pt} \\
{\rm s.t.} & x_1 \leq y_1(\vec{u}) , && \forall \vec{u} \in U\\
 & x_2 \leq y_2(\vec{u}) , && \forall \vec{u} \in U\\
& u_1 y_1(\vec{u}) \leq 1 , && \forall \vec{u} \in U \\
& u_2 y_2(\vec{u}) \leq 1 , && \forall \vec{u} \in U
\end{array}}
\end{equation*}

Let the constraint-wise uncertainty set be
\begin{equation*}
    U = \left\{(u_1,u_2) \mid 0 \leq u_i \leq 1, \quad \forall i \in [2]\right\}.
\end{equation*}
Let the coupling set be
\begin{equation*}
    C = \left\{(u_1,u_2) \mid 0 \leq u_1+u_2\leq 1\right\}.
\end{equation*}
Then the coupled uncertainty set is
\begin{equation*}
    \overline{U} = \left\{(u_1,u_2) \mid 0 \leq u_1+u_2\leq 1, \quad 0 \leq u_i \leq 1, \quad \forall i \in [2]\right\}.
\end{equation*}
We have $U = \Pi(\cp)$, $\ari = \arho = 1/2$ and $\agi = 1$, and thus
$1 \leq \zacp/\zaro \leq 4, \zacp/\zcp \leq 4$.
The optimal solution for the adaptive variable is $y_1(\vec{u})=1/u_1, y_2(\vec{u})=1/u_2.$
If $c_1 = 1, c_2 = 1$, then $\zaro=\zcp =2,\zacp=8$, $\zacp/\zaro=  4 = 1/\ari^2, \zacp/\zcp =4=1/\arho^2$.
If $c_1 = 1, c_2 = 0$, then $\zaro=\zacp=1$, $\zacp/\zaro  =1 = 1/\agi^2$.
Therefore the bounds are tight.
\ifpreprint \end{proof} \else
\Halmos
\endproof
\fi

\section{Solution methods}
\label{sec:solution}
In this section, we describe different solution methods to solve robust and adaptive robust problems under coupled uncertainty, for both cases of right hand side and coefficient uncertainty.
\subsection{Robust optimization}
We describe one method for the simpler case of right hand side uncertainty, and three solution methods to solve problem~\eqref{eqn:cp2} under coefficient uncertainty.

\paragraph{Problem under right hand side uncertainty.}
To solve robust problem~\eqref{eqn:cp}, we first solve~\eqref{eq:maxu} to obtain $\overline{d}_i, \forall i \in [m]$, and then solve the problem
\begin{equation*}
{
\begin{array}{llclcl}
\displaystyle \zcp =
& \text{minimize} && {\vec{c}^T\vec{x} + \vec{d}^T\vec{y}} \\
& \text{subject to} &&
 \vec{a}_i^T\vec{x} + \vec{g}_i^T\vec{y}\geq \overline{d}_i, \quad \forall i \in [m].
\end{array}}
\end{equation*}

\paragraph{Constraint-wise formulation with direct projections.}
\label{subsec:dirproj}
Typically, robust counterpart of the problem is formulated constraint-wise by including the RC for each robust constraint.
The coupled problem~\eqref{eqn:cp2} is equivalent to the constraint-wise problem
\begin{equation*}
{
\begin{array}{llclcl}
\displaystyle \text{maximize} & \multicolumn{3}{l}{\vec{c}^T \vec{x}} \\
\text{subject to} & \vec{u}_i^T \vec{x} \leq b_i && \vec{u}_i \in \Pi_i\left(\cp\right), \quad \forall i \in [m],
\end{array}}
\end{equation*}
where $\Pi_i\left(\cp\right)$ is the projection of $\cp$ onto the $i_{\rm th}$ constraint.
We can solve the problem by first computing the projections $\Pi_i\left(\cp\right)$ and then solving the RC
\begin{equation*}
\label{rc}
{
\begin{array}{llclcl}
\displaystyle \text{maximize} & \multicolumn{3}{l}{\text{  }\vec{c}^T \vec{x}} \\
\text{subject to} &     \max\limits _{\vec{u}_i \in \Pi_i\left(\cp\right)}  \vec{u}_i^T \vec{x} \leq b_i, && \quad \forall i \in [m].
\end{array}}
\end{equation*}
The efficiency of this method depends on the efficiency to compute the projection of the uncertainty set.
For polyhedral uncertainty set $\cp$, we can use the library cddlib~\citep{fukuda2003cddlib} to compute polyhedron projection by Fourier-Motzkin elimination~\citep[Section 2.8]{bertsimas1997introduction}.
Since the projection of a set is often intractable to compute especially in high dimensions, we next describe two alternative methods that do not require computation of projections.

\paragraph{Robust counterpart.}
We construct a matrix $\vec{e}_i = \begin{bmatrix*}[r] 0 & \dots & \vec{I}_p & 0 & \dots & 0 \end{bmatrix*}\in \reals^{p \times mp}$ such that $\vec{e}_i \vec{u} = \vec{u}_i$.
Recall that the support function~\citep[Section 13]{rockafellar1970convex} of a set $S$ is
\begin{equation*}
    \delta^{\star} (\vec{y}|S) = \sup\limits_{\vec{u}\in S} \left\{\vec{y}^T\vec{u}\right\}.
\end{equation*}

\begin{theorem}
\label{thm:rc}
The Robust counterpart of~\eqref{eqn:cp2} for each constraint $i$ is
\begin{equation*}
    \begin{cases}
    \delta^\star(\vec{y}_i^1 \mid U) + \delta^\star(\vec{y}_i^2 \mid \cpc)  \leq b_i \\
    \vec{y}_i^1+\vec{y}_i^2 = \vec{e}_i^T \vec{x},
    \end{cases}
\end{equation*}
where $\vec{y}_i^1, \vec{y}_i^2 \in \reals^{mp}.$
\end{theorem}

\ifpreprint \begin{proof} \else
\proof{Proof.}
\fi
For each constraint $i \in [m]$, the robust counterpart is
\begin{equation*}
  \delta^\star(\vec{x} \mid \Pi_i\left(\cp\right)) \leq b_i.
\end{equation*}

The support function~\citep[Section 13]{rockafellar1970convex}  is  $$\delta^\star\left(\vec{x} \mid \Pi_i\left(\cp\right)\right)= \sup\limits_{\vec{u}_i \in \Pi_i\left(\cp\right)} \left\{\vec{x}^T \vec{u}_i \right\}  = \sup\limits_{\vec{u} \in \cp} \left\{\vec{x}^T \vec{e}_i \vec{u} \right\}
 = \delta^\star(\vec{e}_i^T\vec{x}\mid \cp).$$

Since $\cp = U \cap \cpc$, we know $\delta^\star\left(\vec{y} \mid \cp\right) = \min\limits_{y_1,y_2} \left\{\delta^\star(\vec{y}^1 |U) + \delta^\star(\vec{y}^2 |\cpc) \mid \vec{y}^1+\vec{y}^2=\vec{y}\right\}$~\citep{ben2015deriving}.
By the conjugate on the infimal convolution result from Theorem 4.17 of \cite{beck2017first}, the RC for each constraint $i \in [m]$ is
\begin{equation*}
    \begin{cases}
    \delta^\star(\vec{y}_i^1 \mid U) + \delta^\star(\vec{y}_i^2 \mid \cpc)  \leq b_i \\
    \vec{y}_i^1+\vec{y}_i^2 = \vec{e}_i^T \vec{x}.
    \end{cases}
\end{equation*}
\ifpreprint \end{proof} \else
\Halmos
\endproof
\fi
We can then compute the RC for different types of $U$.
We note that we utilize the support function of the intersection of sets to derive the robust counterpart~\citep{ben2015deriving}.
However, for constraint-wise problems the intersection set is $U_i$ for uncertainty in constraint $i$, whereas for coupled problems the intersection set is $\cp$ on the space of all constraints.
For polyhedral $U$ and $\cpc$, the extra numbers of variables and constraints in the coupled RC compared with the \acrlong{cw} RC are both $m\ell$, where $\ell$ denotes the number of coupling constraints in $C$.

\paragraph{Robust counterpart for polyhedral coupling set.}
Let $\cpc = \left\{\vec{u} \mid \vec{D} \vec{u} \leq \vec{d} \right\}$, where $\vec{D} \in \reals^{l \times mp}, \vec{d} \in \reals^l$. An illustration is shown in Figure~\ref{fig:polycpc}.
The RC of~\eqref{eqn:cp2} for each constraint $i$ is
\begin{equation*}
    \begin{cases}
      \delta^\star(\vec{y}_i^1 \mid U) + \vec{d}^T \vec{z}_i  \leq b_i \\
    \vec{D}^T \vec{z}_i = \vec{y}_i^2\\
    \vec{z}_i \geq 0\\
    \vec{y}_i^1+\vec{y}_i^2 = \vec{e}_i^T \vec{x},
    \end{cases}
\end{equation*}
where $\vec{y}_i^1, \vec{y}_i^2 \in \reals^{mp},$ $\vec{z}_i \in \reals^l$.

\paragraph{Robust counterpart for polyhedral coupled uncertainty set.}
Let $\cp$ be a polyhedron denoted by $\left\{\vec{u} \mid  \vec{Q} \vec{u} \leq \vec{h}\right\}$, where $\vec{Q} \in \reals^{(q+l) \times mp}, \vec{h} \in \reals^{q+l}$.
An illustration is shown in Figure~\ref{fig:polycp}.
Then the RC is
\begin{equation*}
    \begin{cases}
 \vec{h}^T \vec{z}_i  \leq b_i \\
    \vec{Q}^T \vec{z}_i = \vec{e}_i^T\vec{x}\\
    \vec{z}_i \geq 0,
    \end{cases}
\end{equation*}
where $\vec{z}_i \in \reals^{q+l}$.

\begin{figure}
    \centering
\begin{subfigure}{0.49\textwidth}
    \centering
\includegraphics[width=0.72\textwidth]{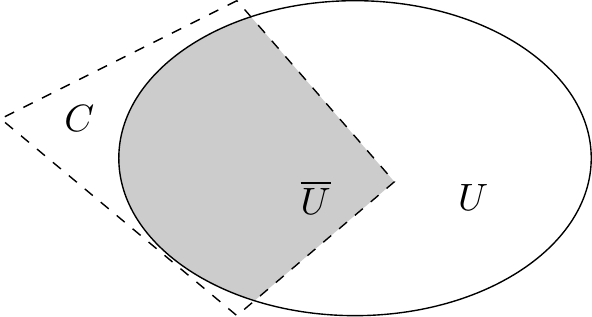}
\caption{Polyhedral coupling set.}
\label{fig:polycpc}
\end{subfigure}
\begin{subfigure}{0.49\textwidth}
    \centering
\includegraphics[width=0.52\textwidth]{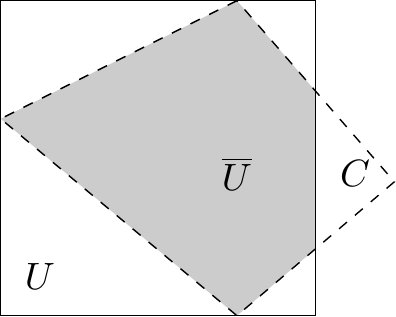}
\caption{Polyhedral coupled set.}
\label{fig:polycp}
\end{subfigure}
\caption{Special cases of coupled uncertainty sets.}
\end{figure}

\paragraph{Cutting plane algorithm.}
Since only a subset of uncertainty values for the constraints will be binding at a robust optimal solution, the cutting plane method generates constraints with uncertainty values as needed.
The cutting plane algorithm is described in Algorithm~\ref{cp},
where the master problem is
\begin{equation}
{
\begin{array}{llclcl}
\displaystyle \text{maximize} & \multicolumn{3}{l}{\vec{c}^T \vec{x}} \\
\text{subject to} & \vec{u}_i^T \vec{x} \leq b_i, && \forall \vec{u} \in U_t, \quad \forall i \in [m].
\end{array}}
\label{eqn:master}
\end{equation}

\begin{algorithm}
\caption{Cutting plane algorithm}
\label{cp}
\begin{algorithmic}[1]
\State Given the nominal value $\vec{u}_0: U_0 \gets \left\{\vec{u}_0\right\}$, tolerance $ \epsilon$, and $t \gets 1$, initiate the master problem~\eqref{eqn:master}
\Repeat
\State Solve the master problem and obtain a solution $ \vec{x}^\star_t$
\State $\phi_i(\vec{u}, \vec{x}^\star_t) \gets \vec{u}_i^T  \vec{x}^\star_t - b_i, \quad \forall i \in [m]$
\State $\vec{u}_t \gets \arg \max \limits_{i \in [m]} \max \limits_{\vec{u} \in \cp}  \phi_i(\vec{u}, \vec{x}^\star_t)$
\Comment{Maximize violation}
\State $U_t \gets \left\{\vec{u}_0,\dots,\vec{u}_t\right\}$
\State $t \gets t+1$
\Until{$\max \limits_{i \in [m]} \max \limits_{\vec{u} \in \cp} \phi_i(\vec{u}, \vec{x}^\star_t) < \epsilon$}
\Comment{No violation}
\end{algorithmic}

\end{algorithm}

\subsection{Adaptive robust optimization}
We describe three solution methods to solve the adaptive problem under coupled uncertainty for both cases of right hand side and coefficient uncertainty.

\paragraph{Robust counterpart with linear decision rules.}
We derive the robust counterpart with linear decision rules $\vec{y}(\vec{u}) = \vec{z} + \vec{V}\vec{u}$ for both cases of coupled problems.
For the case of right hand side uncertainty, $\vec{z} \in \reals^{n_2}, \vec{V} \in \reals^{n_2 \times m}$.
\begin{theorem}
\label{thm:rcadapt1}
The Robust counterpart of~\eqref{eqn:acp} under linear decision rules is
\begin{equation*}
{
\begin{array}{llclcl}
\displaystyle   \rm{minimize} & \multicolumn{3}{l}{\vec{c}^T\vec{x} + \tau} \\
\rm{subject} \text{ } \rm to &
\vec{d}^T \vec{z} - \tau +  \delta^\star(\vec{h}^1 \mid U) + \delta^\star(\vec{h}^2 \mid \cpc)  \leq 0\\
 & \vec{h}^1+\vec{h}^2 = \vec{V}^T \vec{d}\\
 & -\vec{a}_i^T\vec{x} - \vec{g}_i^T \vec{z} + \delta^\star\left(\vec{w}_i^1 \mid U\right) + \delta^\star\left(\vec{w}_i^2 \mid \cpc\right)  \leq 0, && \forall i \in [m] \\
 & \vec{w}_i^1+\vec{w}_i^2 = \vec{e}_i - \vec{V}^T \vec{g}_i, && \forall i \in [m],
\end{array}}
\end{equation*}
where $\vec{h}^1, \vec{h}^2, \vec{w}_i^1, \vec{w}_i^2, \vec{e}_i \in \reals^{m}.$
\end{theorem}
For the case of coefficient uncertainty, we consider a simpler problem with linear decision rules:
\begin{equation}
\label{eqn:ldrcp}
{
\begin{array}{llclcl}
\displaystyle  & \text{maximize} && \multicolumn{3}{l}{\vec{c}^T\vec{x}} \\
&\text{subject to} &&
\vec{u}_i^T\vec{x} + \vec{g}_i^T\left(\vec{z} + \vec{V}\vec{u}\right) \leq b_i, && \forall \vec{u} \in \cp, \quad \forall i \in [m],
\end{array}}
\end{equation}
where we have a fixed recourse $\vec{g_i} \in \reals^{n_2}$, and $\vec{z} \in \reals^{n_2}, \vec{V} \in \reals^{n_2 \times mp}, \vec{u} \in \reals^{mp}$.

\begin{theorem}
\label{thm:rcadapt}
The Robust counterpart of~\eqref{eqn:ldrcp} for each constraint $i$ is
\begin{equation*}
    \begin{cases}
  \vec{g}_i^T \vec{z} + \delta^\star\left(\vec{w}_i^1 \mid U\right) + \delta^\star\left(\vec{w}_i^2 \mid \cpc\right)  \leq b_i \\
    \vec{w}_i^1+\vec{w}_i^2 = \vec{e}_i^T \vec{x} + \vec{V}^T \vec{g}_i,
    \end{cases}
\end{equation*}
where $\vec{w}_i^1, \vec{w}_i^2 \in \reals^{mp}.$
\end{theorem}

\ifpreprint \begin{proof} \else
\proof{Proof.}
\fi
Since $\vec{e}_i \vec{u} = \vec{u}_i$, each constraint $i \in [m]$ in problem \eqref{eqn:ldrcp} can be rewritten as
\begin{equation*}
\vec{g}_i^T\vec{z}
+ \left(\vec{e}_i^T\vec{x} + \vec{V}^T\vec{g}_i\right)^T \vec{u} \leq b_i, \quad \forall \vec{u} \in \cp.
\end{equation*}
The robust counterpart is
\begin{equation*}
\vec{g}_i^T\vec{z} + \delta^\star\left(\vec{e}_i^T \vec{x}  + \vec{V}^T\vec{g}_i \mid \cp\right) \leq b_i.
\end{equation*}

Since $\cp = U \cap \cpc$, the RC for each constraint $i \in [m]$ is
\begin{equation*}
    \begin{cases}
  \vec{g}_i^T \vec{z} + \delta^\star\left(\vec{w}_i^1 \mid U\right) + \delta^\star\left(\vec{w}_i^2 \mid \cpc\right)  \leq b_i \\
    \vec{w}_i^1+\vec{w}_i^2 = \vec{e}_i^T \vec{x} + \vec{V}^T \vec{g}_i.
    \end{cases}
\end{equation*}
\ifpreprint \end{proof} \else
\Halmos
\endproof
\fi
The proof of Theorem~\ref{thm:rcadapt1} is similar and we omit it here.
We can then compute the RC for different types of $U$.
\paragraph{Robust counterpart with linear decision rule for polyhedral coupling set.}
Let $\cpc = \left\{\vec{u} \mid \vec{D} \vec{u} \leq \vec{d} \right\}$, where $\vec{D} \in \reals^{l \times mp}, \vec{d} \in \reals^l$.
The RC of~\eqref{eqn:cp2} for each constraint $i$ is
\begin{equation*}
    \begin{cases}
      \vec{g}_i^T \vec{z} +  \delta^\star\left(\vec{w}_i^1 \mid U\right) + \vec{d}^T \vec{t}_i  \leq b_i \\
    \vec{D}^T \vec{t}_i = \vec{w}_i^2\\
    \vec{t}_i \geq 0\\
    \vec{w}_i^1+\vec{w}_i^2 = \vec{e}_i^T \vec{x}+ \vec{V}^T \vec{g}_i,
    \end{cases}
\end{equation*}
where $\vec{w}_i^1, \vec{w}_i^2 \in \reals^{mp},$ $\vec{t}_i \in \reals^l$.

\paragraph{Robust counterpart with linear decision rule for polyhedral coupled uncertainty set.}
Let $\cp$ be a polyhedron denoted by $\left\{\vec{u} \mid  \vec{Q} \vec{u} \leq \vec{h}\right\}$, where $\vec{Q} \in \reals^{(q+l) \times mp}, \vec{h} \in \reals^{q+l}$.
The RC is
\begin{equation*}
    \begin{cases}
      \vec{g}_i^T \vec{z} + \vec{h}^T \vec{t}_i  \leq b_i \\
    \vec{Q}^T \vec{t}_i = \vec{e}_i^T\vec{x}+ \vec{V}^T \vec{g}_i,\\
    \vec{t}_i \geq 0,
    \end{cases}
\end{equation*}
where $\vec{t}_i \in \reals^{q+l}$.

\paragraph{Benders decomposition.}
A two-level algorithm is described in~\cite{bertsimas2012adaptive} to solve two-stage adaptive robust problems.
The outer level uses a Benders decomposition type cutting plane algorithm to obtain the first-stage decision with cuts generated from the inner level, which uses an outer approximation (OA) algorithm~\citep{duran1986outer,fletcher1994solving} to solve the second-stage bilinear optimization problem.
We refer the reader to~\cite{bertsimas2012adaptive} for a complete description of the algorithm.
One option to initiate the outer level is with a static robust solution under the coupled or uncoupled uncertainty set.
Since the bilinear problem is nonconvex, we can use multiple starts for the OA algorithm by sampling multiple uncertainties from the coupled set.

\paragraph{Finite scenario approach.}
A finite scenario approach is proposed by~\cite{hadjiyiannis2011finite}, to replace the original uncertainty set by a finite discrete subset and the adaptive variable is selected to adapt to each realization of the uncertainty.
For some problems with small sizes, using the set of the extreme points of the uncertainty set can give a fully adaptive solution.
When such vertex enumeration is not feasible, different sampling methods can be used to select a finite subset, such as Hit-and-run sampler~\citep{smith1984efficient} for polyhedral coupled uncertainty set.
Since the important scenarios are more likely to occur at the boundary of the uncertainty set, we
implement a variation that rescales sampled points to the boundary of the set, for example by enforcing the coupling constraint to be tight.

\ifpreprint \else
\end{APPENDICES}
\fi

\end{document}